\documentclass[final,reqno,dvipdfmx]{amsart}

\usepackage{fullpage}
\usepackage{mathtools,amssymb,amsthm}
\usepackage[foot]{amsaddr}
\usepackage{mathrsfs,euscript}
\usepackage{slashed,bm}
\usepackage[only llbracket,rrbracket,longmapsfrom]{stmaryrd}
\usepackage{enumitem}
\usepackage{showkeys}
\usepackage[bookmarks=false,draft=false,breaklinks,colorlinks]{hyperref} 
\usepackage{cleveref}
\usepackage{comment}

\numberwithin{equation}{section}
\allowdisplaybreaks

\newenvironment{Ack}%
{\par \vspace{\baselineskip}%
 \noindent \textbf{Acknowledgements.}}%
{\par \vspace{\baselineskip}}

\setenumerate{label=(\arabic*),nosep}
\setitemize{nosep}
\newlist{clist}{enumerate}{1}
\setlist*[clist]{label=(\roman*), nosep}

\theoremstyle{definition}
\newtheorem{thm}{Theorem}[section]
\newtheorem{dfn}[thm]{Definition}
\newtheorem{prp}[thm]{Proposition}
\newtheorem{lem}[thm]{Lemma}
\newtheorem{cor}[thm]{Corollary}

\newtheorem{fct}[thm]{Fact}
\newtheorem{rmk}[thm]{Remark}
\newtheorem{eg}[thm]{Example}
\newtheorem*{rmk*}{Remark}

\crefname{thm}{Theorem}{Theorems}
\crefname{dfn}{Definition}{Definitions}
\crefname{prp}{Proposition}{Propositions}
\crefname{lem}{Lemma}{Lemmas}
\crefname{cor}{Corollary}{Corollaries}
\crefname{clm}{Claim}{Claims}
\crefname{fct}{Fact}{Facts}
\crefname{rmk}{Remark}{Remarks}
\crefname{eg}{Example}{Examples}
\crefname{figure}{Figure}{Figures}
\crefname{table}{Table}{Tables}
\crefname{section}{\S\!}{\S\S\!}
\crefname{subsection}{\S\!}{\S\S\!}
\crefname{subsubsection}{\S\!}{\S\S\!}
\crefname{appendix}{Appendix}{Appendices}
\crefname{equation}{}{}

\newcommand{\bl}{\bullet}
\newcommand{\ep}{\epsilon}
\newcommand{\ve}{\varepsilon}

\newcommand{\ol}{\overline}
\newcommand{\ul}{\underline}

\newcommand{\wt}{\widetilde}
\newcommand{\ceq}{\coloneqq} 

\newcommand{\xr}{\xrightarrow}
\newcommand{\xrr}[1]{\xrightarrow{\ #1 \ }{}}

\newcommand{\srj}{\twoheadrightarrow}

\newcommand{\lto}{\longrightarrow}

\newcommand{\sto}{\xr{\sim}}
\newcommand{\lsto}{\xrr{\sim}}
\newcommand{\linj}{\lhook\joinrel\longrightarrow}
\newcommand{\lsrj}{\relbar\joinrel\twoheadrightarrow}
\newcommand{\mto}{\mapsto}
\newcommand{\lmto}{\longmapsto}

\newcommand{\low}{\textup{lower}}

\newcommand{\tbig}{\textup{big}}
\newcommand{\tmin}{\textup{min}}
\newcommand{\tred}{\textup{red}}

\newcommand{\tsum}{{\textstyle\sum}}

\newcommand{\bbC}{\mathbb{C}}
\newcommand{\bbG}{\mathbb{G}}
\newcommand{\bbN}{\mathbb{N}}
\newcommand{\bbP}{\mathbb{P}}
\newcommand{\bbQ}{\mathbb{Q}}
\newcommand{\bbV}{\mathbb{V}}
\newcommand{\bbZ}{\mathbb{Z}}
\newcommand{\clH}{\mathcal{H}}
\newcommand{\clK}{\mathcal{K}}
\newcommand{\clL}{\mathcal{L}}

\newcommand{\shO}{\EuScript{O}}
\newcommand{\frg}{\mathfrak{g}}

\newcommand{\fgl}{\mathfrak{gl}}
\newcommand{\fsl}{\mathfrak{sl}}
\newcommand{\fso}{\mathfrak{so}}
\newcommand{\psl}{\mathfrak{psl}}
\newcommand{\pgl}{\mathfrak{pgl}}
\newcommand{\osp}{\mathfrak{osp}}
\newcommand{\frR}{\mathfrak{R}}
\newcommand{\frs}{\mathfrak{s}}
\newcommand{\Mat}{\mathrm{Mat}}
\newcommand{\wtA}{\wt{A}}
\newcommand{\wtO}{\wt{O}}
\newcommand{\Flow}{F_{\low}}
\newcommand{\Vir}{\mathrm{Vir}}

\newcommand{\ev}{\ol{0}}
\newcommand{\od}{\ol{1}}
\newcommand{\Zt}{\bbZ/2\bbZ}
\newcommand{\pd}{\partial}
\newcommand{\sd}{\slashed{\pd}} 
\newcommand{\hf}{\frac{1}{2}}

\newcommand{\vac}{\bm{1}}

\newcommand{\abs}[1]{\left| #1 \right|} 
\newcommand{\dbr}[1]{\llbracket #1 \rrbracket} 
\newcommand{\dpr}[1]{(\!( #1 )\!)} 
\newcommand{\st}[1]{{}^{st}\!{#1}} 
\newcommand{\trp}[1]{{}^{t}\!{#1}} 
\newcommand{\Hop}[2]{\mathbin{\overunderset{#1}{#2}{\bl}}}
\newcommand{\Zop}[2]{\mathbin{\overunderset{#1}{#2}{*}}}

\DeclareMathOperator{\id}{id}
\DeclareMathOperator{\gr}{gr}
\DeclareMathOperator{\tr}{tr}
\DeclareMathOperator{\str}{str}

\DeclareMathOperator{\Der}{Der}
\DeclareMathOperator{\End}{End}
\DeclareMathOperator{\Exp}{Exp}

\DeclareMathOperator{\Ker}{Ker}

\DeclareMathOperator{\Log}{Log}
\DeclareMathOperator{\Zhu}{Zhu}
\DeclareMathOperator{\Span}{Span}
\DeclareMathOperator*{\res}{res}
\DeclareMathOperator*{\sres}{sres}

\newenvironment{sbm}{\left[\begin{smallmatrix}}{\end{smallmatrix}\right]}

\begin{document}

\title{Zhu algebras of superconformal vertex algebras}
\author{Ryo Sato$^1$}
\address{$^1$
 Center for General Education, Aichi Institute of Technology. 
 Yakusa-cho, Toyota, Japan, 470-0392.}
\email{rsato@aitech.ac.jp}
\author{Shintarou Yanagida$^2$}
\address{$^2$
 Graduate School of Mathematics, Nagoya University.
 Furocho, Chikusaku, Nagoya, Japan, 464-8602.}
\email{yanagida@math.nagoya-u.ac.jp}
\date{April 7, 2026}

\begin{abstract}
The purpose of this note is to demonstrate the advantages of 
Y.-Z.~Huang's definition of the Zhu algebra 
 (Comm.\ Contemp.\ Math., 7 (2005), no.~5, 649--706)
for an arbitrary vertex algebra, 
not necessarily equipped with a Hamiltonian operator or a Virasoro element, 
by achieving the following two goals:
(1) determining the Zhu algebras of $N=1, 2, 3, 4$ 
and big $N=4$ superconformal vertex algebras, 
and
(2) introducing the Zhu algebras of $N_K=N$ supersymmetric vertex algebras.
\end{abstract}

\maketitle
{\small \tableofcontents}

\setcounter{section}{-1}
\section{Introduction}

\subsection{Zhu algebra}\label{ss:intro:A}

The Zhu algebra $A(V)$ is an associative algebra associated to 
a vertex operator algebra (VOA) $V$. 
It was introduced by Yongchang Zhu \cite[\S2]{Zhu}, and plays a fundamental role 
in understanding the representation theory of VOAs.
%

The definition and fundamental properties of $A(V)$ 
will be reviewed in \cref{ss:wtA:A}.
Here we mention only that 
the algebra $A(V)$ is filtered by conformal weight, 
its associated graded object is a Poisson algebra \cite[\S4.4]{Zhu},  
and it admits a surjection from the $C_2$-Poisson algebra of $V$ 
(see, e.g., \cite[IV.3]{GG}).
Zhu's fundamental theorem \cite[Theorem 2.2.2]{Zhu} establishes 
a one-to-one correspondence between: 
\begin{itemize}
\item the isomorphism classes of simple $\bbZ_{\ge0}$-graded $V$-modules,
\item and the isomorphism classes of simple (left) $A(V)$-modules.
\end{itemize}

There have been many developments in the study of the Zhu algebra.
We mention only a few: 
\begin{itemize}
\item 
Determination of the structure of $A(V)$:
There is a vast amount of literature on this topic. 
Among the early works, let us mention 
\cite{FZ} for affine vertex algebras,
\cite{Wa} for Virasoro vertex algebras,
\cite{KWa} for the $N=1$ superconformal vertex algebra 
in the Neveu--Schwarz sector,
and \cite{DLM0} for lattice VOAs.



\item 
Twisted Zhu algebra $\Zhu_H(V)$:
a generalization to a graded vertex algebra $V$ with Hamiltonian $H$ \cite[\S2]{DK}.

\item 
Connections to the representation theory of finite $W$-algebras, 
especially via BRST cohomology \cite{DK,G24b}.

\end{itemize}

\subsection{Huang's formulation of Zhu algebra}\label{ss:intro:wtA}

As explained in \cref{ss:intro:A}, 
the original definition of the Zhu algebra $A(V)$ 
requires the vertex algebra $V$ to be a VOA,
that is, it must possess a conformal vector and a $\bbZ_{\ge0}$-grading.
The twisted Zhu algebra $\Zhu_H(V)$ introduced by De Sole and Kac \cite{DK} 
also requires $V$ to have a grading and a Hamiltonian $H$.

In contrast to these definitions, 
the associative algebra $\wtA(V)$ introduced by Yi-Zhi Huang in \cite[\S6]{Hu} 
is defined for an arbitrary vertex algebra $V$.
Note that while the paper \cite{Hu} introduces $\wtA(V)$ 
for a vertex operator algebra $V$, 
its definition can be straightforwardly extended to any vertex algebra $V$.
We will review the precise definition of $\wtA(V)$ in \cref{ss:wtA:wtA}.
If $V$ is a vertex operator algebra, 
then $\wtA(V)$ is isomorphic to $A(V)$ as an associative algebra.
The definition of $\wtA(V)$ is more conceptual and geometric 
than those of $A(V)$ and $\Zhu_H(V)$, 
and it requires no extra structure on the vertex algebra $V$.

The literature on Huang's Zhu algebra $\wtA(V)$ is still limited.
However, van Ekeren and Heluani \cite{vEH,vEH2} studied $\wtA(V)$ and 
used it to reveal the relationship between 
the chiral homology and the Hochschild cohomology of the Zhu algebra.
In addition, Huaimin Li and Qing Wang \cite{LW} introduced analogues 
of the higher Zhu algebras \cite{DLM} 
and presented several interesting results on the representation theory of VOAs.

The aim of this note is to clarify the advantages of using Huang's version $\wtA(V)$
rather than the original $A(V)$ or $\Zhu_H(V)$, 
and to contribute to the literature by further studying $\wtA(V)$.
Specifically, we will demonstrate:
\begin{clist}
\item \label{i:intro:motif1}
The algebra structure of $\wtA(V)$ can be determined more easily than that of $A(V)$.
\item \label{i:intro:motif2} 
A SUSY analogue of the Zhu algebra can be naturally introduced from a geometric viewpoint.
\end{clist}
To illustrate \ref{i:intro:motif1}, we determine the Zhu algebras of 
the $N=1, 2, 3, 4$ and big $N=4$ superconformal vertex algebras.

\subsection{Main result 1: Zhu algebras of superconformal vertex algebras}
\label{ss:intro:main}

Superconformal algebras are supersymmetric enlargements of the Virasoro algebra 
and play a fundamental role in the study of two-dimensional 
superconformal field theory. 
Since the works of \cite{KRW, KW}, it has been known that 
each of these superconformal algebras is isomorphic to an affine $W$-algebra 
(possibly after tensoring with free fields). 
Together with \cite[Theorem 5.10]{DK}, this implies that their Zhu algebras 
provide concrete examples of finite $W$-algebras \cite{P, DK}.

This paper focuses on the following vertex algebras 
associated with superconformal algebras:
\begin{itemize}
\item 
The $N=1$ superconformal vertex algebra $V^{N=1}$, 
also called Neveu--Schwarz vertex algebra \cite[Example 5.9b]{K}, \cite[\S8.2]{KW}.
\item 
The $N=2$ superconformal vertex algebra $V^{N=2}$ \cite[Example 5.9c]{K}, \cite[\S8.3]{KW}.
\item 
The $N=4$ superconformal vertex algebra $V^{N=4}$ \cite[\S8.4]{KW}.
See \cref{ss:N=4} for the precise definition.
\item 
The $N=3$ superconformal vertex algebra $V^{N=3}$ \cite[\S8.5]{KW}.
See \cref{ss:N=3}.
\item 
The big $N=4$ superconformal vertex algebra $V^{N=4,\tbig}$ \cite[\S8.6]{KW}.
See \cref{ss:big4}.
\end{itemize}

The vertex algebras $V^{N=1,2,4}$ are referred to 
as minimal $W$-algebras in the literature;
see \cite{KMPa,KMP-II,KMPb} for recent related studies.
The main results are summarized as follows
(precise statements are given in \cref{thm:N=4,cor:N=2,cor:N=1,thm:N=3,thm:big4}).

\begin{thm}\label{thm:main}
Let $(V,\frg)$ be one of the following:
\begin{align*}
&\text{Case I}\colon\quad
 (V^{N=1},\osp(1|2)),\quad
 (V^{N=2},\fsl(1|2)),\quad
 (V^{N=4},\psl(2|2));\\
&\text{Case II}\colon\quad
 (V^{N=3},\osp(3|2)),\quad
 (V^{N=4,\tbig},D(2,1;a)) \quad \text{where $a\in\bbC\setminus\{0,-1\}$},
\end{align*}
and let $\frg^f$ be the centralizer of 
a minimal nilpotent element $f=f_{\tmin}$ in $\frg$.
\begin{enumerate}
\item 
In Case I, the Zhu algebra $\wtA(V)$ is isomorphic to 
the universal enveloping algebra $U(\frg^f)$.

\item 
In Case II, $\wtA(V)$ contains a subalgebra which is isomorphic to $U(\frg^f)$,
and as a vector space, we have 
\begin{align}\label{eq:N=34:mod_str}
 \wt{A}(V^{N=3}) \cong U(\osp(3|2)^f) \otimes C(\bbC), \quad 
 \wt{A}(V^{N=4,\tbig}) \cong U(D(2,1;a)^f) \otimes \bbC[\xi] \otimes C(\bbC^4),
\end{align}
where $C(\bbC),C(\bbC^4)$ are the Clifford algebra with one and four generators,
respectively, and $\bbC[\xi]$ is the polynomial algebra of one even generator $\xi$.
(see \cref{ss:N=3,ss:big4} for details).
\end{enumerate}
\end{thm}

For the precise descriptions of the algebraic structure of $\wtA(V)$ in Case II, 
see \cref{thm:N=3,thm:big4}.
While Case I appears to be known to experts. 
we were unable to find any literature explicitly determining 
the structure of the Zhu algebras for Case II.

As mentioned in \ref{i:intro:motif1}, 
Huang's version $\wtA(V)$ has an advantage over the original $A(V)$, in that 
\begin{quote}
one can immediately read off the relations of $\wtA(V)$ from 
the $\lambda$-brackets (OPEs) of the given vertex algebra $V$.
\end{quote}
This is demonstrated in \cref{cor:wtA:pd=0}, \eqref{eq:wtA:a0b}.

We prove \cref{thm:main} uniformly in \cref{ss:N=124:bij}.
Let us outline the strategy of the proof.
In Case I, using \cref{cor:wtA:pd=0}, 
we construct a surjection $\varphi\colon U(\frg^f) \srj \wtA(V)$.
Thus, it remains to construct the inverse $\ol{\psi}$ 
of $\varphi$ (\cref{prp:Ngen:varphi}).
For this purpose, we use the linearity of the $\lambda$-brackets 
of $V^{N=1,2,4}$ to establish a well-defined correspondence between 
the PBW generators of $U(\frg^f)$ and their images in $\wtA(V)$.
This is carried out in \cref{lem:psi:wtO-Ker},
following the preparatory results in \cref{lem:psi:prep}.

Case II is treated in a similar manner, but the method cannot be applied directly
due to the nonlinearity of the $\lambda$-brackets.
To circumvent this problem, 
we use the ``linear realization'' of the finite $W$-algebras associated with 
the superconformal vertex algebras $V^{N=3}$ and $V^{N=4,\tbig}$, 
which dates back to \cite{GS}.
We follow the descriptions in \cite[\S8.5, \S8.6]{KW} 
to implement the aforementioned strategy..


\subsection{Relation to finite \texorpdfstring{$W$}{W}-algebras}\label{ss:intro:finW}

Let us explain the relationship between the Zhu algebras of superconformal 
vertex algebras discussed in \cref{thm:main} and finite $W$-algebras.
We divide the argument into Cases I and I\!I in \cref{thm:main}.
\begin{itemize}
\item 
In Case I, the superconformal vertex algebras $V$ are 
the affine $W$-algebras $W^k(\frg,f)$ associated with minimal nilpotent elements 
$f=f_{\tmin}$ (see \cite[\S8.2, \S8.3]{KW} for the details).
In these cases, 
the Zhu algebras $\wtA(V) \cong U(\frg^f)$ coincide with the finite $W$-algebras $U(\frg,f)$
(see \cite{P,DK} and also \cite[Theorem 3.7]{Zha} 
 for the definition of $U(\frg,f)$ in the super case):
\begin{gather*}
 \wtA(V^{N=1}_c) \cong U(\osp(1|2)^{f}) \cong U(\osp(1|2),f), \\
 \wtA(V^{N=2}_c) \cong U(\fsl(1|2)^f)   \cong U(\fsl(1|2),f), \\
 \wtA(V^{N=4}_c) \cong U(\psl(2|2)^f)   \cong U(\psl(2|2),f).
\end{gather*}
For the isomorphisms $U(\frg^f)\cong U(\frg,f)$, 
see \cite[Lemma 3.4]{PS16} and \cite[Introduction]{G24a} 
for the $N=1$ case, 
and see \cite[Theorem 2]{PS13} 
for the $N=2$ case.
The isomorphism in the $N=4$ case is obtained 
by combining the following results: 
$U(\psl(2|2)^f) \cong\wtA(V^{N=4}_c)$ from \cref{thm:main}, 
$\wtA(V^{N=4}_c)\cong A(V^{N=4}_c)$ from \cref{fct:Zhu:wtA=Zhu}, 
and $A(V^{N=4}_c)\cong U(\psl(2|2),f)$ from \cite[Theorem 8.4]{G24b}.
To the best of our knowledge, a direct proof of 
$U(\psl(2|2)^f) \cong U(\psl(2|2),f)$
has not yet appeared in the literature.

These results exemplify the fact that 
``reduction and taking the Zhu algebra commute,'' 
as proven in general in \cite[Theorem~5.10]{DK} 
and \cite[Theorem~8.4]{G24b} (including the twisted case).

\item
Case II presents a slightly different situation.
According to \cite[\S8.5, \S8.6]{KW} (see also \cite[(7.8)]{KMPb}), 
for minimal nilpotent $f=f_{\tmin}$, we have
\begin{align*}
&V^{N=3}_c \cong W^k(\osp(3|2),f) \otimes F^k(\bbC), \quad c=-6k-3,\\ 
&V^{N=4,\tbig}_{c,a} \cong W^k(D(2,1;a),f) \otimes B^k(\bbC\xi) \otimes F^k(\bbC^4),
 \quad c=-6k,
\end{align*}
where $F^k(\bbC)$ is the fermionic vertex algebra with an odd generator $\Phi$
whose $\lambda$-bracket is $[\Phi_\lambda\Phi]=-(k+\hf)$, 
$F^k(\bbC^4)$ is the fermionic vertex algebra 
with four odd generators $\sigma^{\pm\pm}$
whose non-zero $\lambda$-brackets are 
$[{\sigma^{--}}_\lambda\sigma^{++}]=[{\sigma^{+-}}_\lambda\sigma^{-+}]=k$, 
and $B^k(\bbC\xi)$ is the bosonic vertex algebra with an even generator $\xi$ 
whose $\lambda$-bracket is $[\xi_\lambda\xi]=k\lambda$.

Based on the relationship between $W^k(\frg,f)$ and $U(\frg,f)$ 
(see \cite[Theorem 5.10]{DK}, \cite[Theorem 8.4]{G24b}), 
the module isomorphisms in \eqref{eq:N=34:mod_str} should be compared 
with the following algebra isomorphisms:
\begin{align*}
 \wt{A}(V^{N=3}_c) \cong U(\osp(3|2),f) \otimes C(\bbC), \quad 
 \wt{A}(V^{N=4,\tbig}_{c,a}) \cong U(D(2,1;a),f) \otimes \bbC[\xi] \otimes C(\bbC^4).
\end{align*}
\end{itemize}

\subsection{Relation to the twisted Zhu algebras in Ramond sector}

Let us explain the significance of the Zhu algebras appearing in \cref{thm:main} 
from the viewpoint of representation theory of superconformal vertex algebras.
As mentioned in \cref{ss:intro:wtA}, the Zhu algebras in \cref{thm:main} are 
isomorphic to the twisted Zhu algebras of 
the corresponding superconformal vertex algebras; 
see Fact \ref{fct:Zhu:wtA=Zhu} for details.
Through this identification, together with \cite[Theorem 2.30]{DK}, 
one can classify simple $\bbZ_{\ge0}$-graded modules in the Ramond sector 
by means of the corresponding Zhu algebra.
Here, modules in the Ramond sector refer to $g$-twisted modules 
for $g=e^{2\pi\sqrt{-1}L_0}$ in the sense of \cite{DLMb}.
See \cite{KMP-II} for a recent study on the representation theory 
of minimal $W$-algebras in the Ramond sector.

\subsection{Main result 2: SUSY Zhu algebras}\label{ss:intro:main2}

Another main result of this note is the introduction 
of Zhu algebras for SUSY vertex algebras.
Recall that the notion of SUSY vertex algebras is a superfield formulation 
of vertex algebras equipped with supersymmetry, 
as introduced by Heluani and Kac \cite{HK}. 
In this note, we focus on $N_K=N$ SUSY vertex algebras, 
and give a brief recollection in \cref{ss:SUSY:NK=N}.
The superconformal vertex algebras considered in \cref{thm:main} 
are examples of SUSY vertex algebras (see \cref{eg:SUSY:N=123}).

Taking into account the geometric property of Huang's version of the Zhu algebra,
we will introduce the Zhu algebra $\wtA(\bbV)$ of 
an $N_K=N$ SUSY vertex algebra $\bbV$ in \cref{ss:SUSY:Zhu}, \cref{dfn:SUSY:Zhu}.
It is shown in \cref{prp:SUSY:Zhu} that 
our $\wtA(\bbV)$ is isomorphic to $\wtA(V_{\tred})$ as a superalgebra,
where $V_{\tred}$ is the reduced part of $\bbV$, i.e., 
the underlying vertex algebra structure of $\bbV$ (see \cref{lem:NK:red}). 
Consequently, $\wtA(\bbV)$ enjoys similar properties to 
the Zhu algebra $\wtA(V)$ for a non-SUSY vertex algebra $V$.

We will show in \cref{prp:SUSY:Zhu_sd} that the SUSY structure of $\bbV$ 
is encoded by the odd differentials on $\wtA(\bbV)$
induced by the odd translation operators $\sd^1,\dotsc,\sd^N$ on $\bbV$.
Moreover, we will see in \cref{cor:SUSY:limit} that, 
in the limit where the Zhu algebra becomes the $C_2$-Poisson algebra $R(\bbV)$ 
introduced by the second author in \cite{Y}, 
these odd differentials coincide with the odd differentials $[\sd^k]$
defined in \cite[\S4.1.1]{Y}.
Therefore, our SUSY Zhu algebra $\wtA(\bbV)$ provides a natural SUSY extension 
of the Zhu algebra $\wtA(V)$ of a non-SUSY vertex algebra.

\subsection{Relation to finite SUSY \texorpdfstring{$W$}{W}-algebras}

Affine SUSY $W$-algebras associated with basic classical Lie superalgebras 
and odd nilpotent elements were introduced by Molev, Ragoucy and Suh in \cite{MRS}. 
These form a large family of $N_K=1$ SUSY vertex algebras \cite[Corollary 4.5]{MRS}.
In the principal case, Genra, Song and Suh proved in \cite{GSS} 
that the Zhu algebra of an affine SUSY $W$-algebra is isomorphic to 
the corresponding finite SUSY $W$-algebra introduced in \cite{CCS}.
Combined with \cref{prp:SUSY:Zhu}, 
this result implies that the finite SUSY $W$-algebra associated with 
a principal odd nilpotent element carries an odd differential 
induced by the odd translation operator on the affine SUSY $W$-algebra.
It would be interesting to study this odd differential in specific cases.

\subsection{Future directions}

In this note, we have determined the Zhu algebras of universal 
superconformal vertex algebras $V$ using Huang's formulation.
For the study of representation theory of superconformal algebras, 
one is particularly interested in the Zhu algebra of the simple quotient of $V$.
We expect that Huang's formulation will be useful
for computing the Zhu algebra of the simple quotient.

As mentioned in \cref{ss:intro:main}, the linearity of the $\lambda$-brackets is 
a technical restriction in applying our strategy to 
determine the Zhu algebra $\wtA(V)$. 
For the original Zhu algebra $A(V)$, De Sole and Kac \cite{DK} introduced 
the notions of non-linear Lie algebras and non-linear Lie conformal algebras,
and showed that the original Zhu algebra of the affine $W$-algebras
(which may have non-linear $\lambda$-brackets)
are isomorphic to the corresponding finite $W$-algebras.
See also the recent work of Genra \cite{G24b}.
We expect that a similar theory can be developed for Huang's version $\wtA(V)$.

Furthermore, for the big $N=4$ superconformal algebra,
the ``linear realization'' appears in the recent paper \cite[\S3.4, (3.46)]{W},
which studies the sigma model whose target is 
the moduli space of instantons on $S^3 \times S^1$.
It may be valuable to investigate the relationship between our \cref{thm:big4} 
on the Zhu algebra $\wtA(V^{N=4,\tbig})$ with such studies, 
taking into account the $N_K=4$ SUSY vertex algebra structure.

There are many topics to investigate concerning our SUSY Zhu algebras.
One of these is the SUSY analogue of the theory of chiral homology 
and its relation to the Hochschild cohomology studied, as studied in \cite{vEH2}.
We also expect that the work \cite{HV} 
on the character of topological $N=2$ vertex algebras
can be re-examined from the perspective of our SUSY Zhu algebras.
Other potential direction include the representation theory of 
SUSY vertex algebras, SUSY analogues of Zhu bimodules, and the role of 
the differential within the SUSY Zhu algebra in representation theory.

\subsection{Organization}

In \cref{s:wtA}, we give a review of Huang's formulation of Zhu algebra.
\cref{ss:wtA:VA} gives the notation and terminology 
of vertex algebras used throughout this note.
\cref{ss:wtA:wtA} explains Huang's version $\wtA_\gamma(V)$
of Zhu algebra for a vertex algebra $V$, 
basically following the original article \cite[\S6]{Hu}
but with some refinements.
In \cref{ss:wtA:A}, we explain the relationship between 
Huang's version $\wt{A}_\gamma(V)$ 
and the original Zhu algebra $A(V)$.

The next \cref{s:N=124} presents the main result of this note.
Subsections \cref{ss:N=4}--\cref{ss:big4} explain the details of 
the superconformal vertex algebras and the corresponding Zhu algebras,
and state the main \cref{thm:N=4,thm:N=3,thm:big4}.
These are proved in a uniform way in \cref{ss:N=124:bij}.

In the final \cref{s:SUSY}, 
we consider $N_K=N$ SUSY vertex algebras and propose their Zhu algebra.
In \cref{ss:SUSY:NK=N}, we recall the framework of SUSY vertex algebras from \cite{HK}.
In \cref{ss:SUSY:Zhu}, we introduce the Zhu algebra for $N_K=N$ SUSY vertex algebras.
We explain the motivation and observation for the $N_K=1$ case in \cref{sss:SUSY:Zhu:NK=1},
and introduce \cref{dfn:SUSY:Zhu} in the general $N_K=N$ case in \cref{sss:SUSY:Zhu:NK=N}.

In \cref{ss:zero-mode}, 
we explain the relationship between the Zhu algebras and 
the centralizer of the zero-mode $L_0$ of the Virasoro element (in the Ramond sector).
It will be shown that the Zhu algebra can be regarded as the ``zero-mode algebra''
of the superconformal vertex algebra.

\subsection{Notations and terminology}
\begin{itemize}
\item 
$\bbN \ceq \{0,1,2,\dotsc\}$ denotes the set of non-negative integers.
\item 
$\bbZ$ denotes the ring of integers.
\item 
Parity is denoted by $\bbZ \to \Zt=\{\ev,\od\}$, $n \mto \ol{n}$.
\item 
We abbreviate the partial differential operator as $\pd_z \ceq \frac{\pd}{\pd z}$.
\item 
A linear (super)space is defined over the complex number field $\bbC$ unless otherwise stated.
\item 
As explained in \cref{ss:wtA:VA}, a vertex algebra means a vertex superalgebra.
\item 
The associative algebra $A(V)$ for a vertex operator algebra $V$ introduced in \cite{Zhu}
will be called \emph{the original Zhu algebra of the VOA $V$}, 
and the associative algebra $\wtA(V)$ for a vertex algebra $V$ introduced in \cite{H}
will be called (\emph{Huang's version of}) \emph{the Zhu algebra of the vertex algebra $V$}.
\end{itemize}

\section{Huang's version of Zhu algebra}\label{s:wtA}

The aim of this section is to recall from \cite[\S6]{Hu} 
Huang's associative algebra $\wtA(V)$ of a vertex algebra $V$.
When $V$ is a vertex operator algebra (VOA), then $\wtA(V)$ is isomorphic to 
Zhu's associative algebra $A(V)$ \cite[\S2]{Zhu}, 
as will be explained in \cref{ss:wtA:A}.

\subsection{Vertex algebras}\label{ss:wtA:VA}

We introduce the notation for vertex algebras used throughout this paper.

\subsubsection{}\label{sss:VA:VA}

We begin with some super notation. Let $V$ be a linear superspace. 
\begin{itemize}
\item 
We denote the parity decomposition of $V$ by $V=V^{\ev} \oplus V^{\od}$. 
An element of $V^p$, $p \in \Zt$, is called of pure parity.
For such $v \in V^p$, we denote $\ol{v} \ceq p$.

\item 
For a linear superspace $V=V^{\ev} \oplus V^{\od}$ with $m=\dim V^{\ev}$ and $n=\dim V^{\od}$,
we denote $\dim V \ceq m|n$.

\item 
We denote by $\End V$ the endomorphism superalgebra of $V$.
Hence, we have $\End V = (\End V)^{\ev} \oplus (\End V)^{\od}$, and 
$(\End V)^{\ev}$ (resp.\ $(\End V)^{\od}$) consists of even (resp.\ odd) linear transformations.

\item 
The super-commutator in $\End V$ is denoted by $[\cdot,\cdot]$.
Explicitly, for $x,y \in \End V$ of pure parity,  
\begin{align}\label{eq:VA:com-form}
 [x,y] \ceq xy-p(x,y)yx, \quad p(x,y) \ceq (-1)^{\ol{x} \ol{y}}.
\end{align}

\item 
We denote by $V\dbr{z} \ceq \{\sum_{j \in \bbN} z^{j} v_j \mid v_j \in V \}$ 
the space of formal power series in the variable $z$ with coefficients in $V$.
It is a linear superspace, where the parity of $z^jv_j$ is set to be $\ol{v_j}$.

\item 
We denote 
$V\dbr{z^{\pm1}} \ceq \{\sum_{j \in \bbZ} z^jv_j \mid v_j \in V \}$, and
$V\dpr{z} \ceq \{\sum_{j \in \bbZ} z^jv_j \in V\dbr{z^{\pm1}} \mid v_j=0 \text{ for } j \ll 0 \}$.
These are also linear superspaces, with parity defined in the same way as for $V\dbr{z}$.
We will also use the linear superspaces $V[z]$ and $V[z^{\pm1}]$, consisting of 
polynomials and Laurent polynomials with coefficients in $V$, respectively.

\item 
The residue map is the even linear map 
\begin{align}\label{eq:VA:res}
 \res_z[-]dz\colon V\dbr{z^{\pm1}} \lto V, \quad 
 \res_z\bigl[\tsum_{j \in \bbZ} z^j v_j\bigr]dz \ceq v_{-1}. 
\end{align}
\end{itemize}
Below, we suppress the prefix ``super'' if no confusion may occur.
For example, a linear superspace and the super-commutator are referred to simply as 
a linear space and the commutator, respectively.

\subsubsection{}

We use the word ``vertex algebra'' to mean a ``vertex superalgebra'' 
in the sense of \cite[\S1.3]{K} and \cite[1.3.2]{FB}.
Explicitly, a vertex algebra is a data $(V,Y,\vac,\pd)$ consisting of 
\begin{itemize} 
\item 
a linear superspace $V=V^{\ev} \oplus V^{\od}$, called the \emph{state space}, 
\item
an even linear map $V \to (\End V)\dbr{z^{\pm1}}$,
$a \mto Y(a,z) = \sum_{n\in\bbZ}z^{-n-1}a_{(n)}$ with $a_{(n)} \in \End V$,
where $z$ is a formal variable, 
\item
a non-zero even element $\vac \in V^{\ev}$, called the \emph{vacuum}, 
\item
an even operator $\pd \in (\End V)^{\ev}$, 
called the (even) \emph{translation} (\emph{operator}).
\end{itemize}
satisfying the following conditions for any $a,b \in V$. 
\begin{clist}
\item  
(field condition) 
$Y(a,z)b \in V\dpr{z}$, i.e., there exists $N(a,b)\in\bbZ$ such that 
\begin{align}\label{eq:VA:f}
 a_{(n)}b = 0 \quad \text{for any $n>N(a,b)$}.
\end{align}
The map $Y$ satisfying this condition is called the \emph{state-field correspondence},
and the series $Y(a,z)$ is called the \emph{field} (or \emph{vertex operator}) of $a \in V$.

\item 
(locality axiom)
The fields $Y(a,z)$ and $Y(b,w)$ are local, i.e., there exists $M(a,b) \in \bbN$ such that 
\begin{align}\label{eq:VA:loc}
 (z-w)^{M(a,b)}[Y(a,z),Y(b,w)]=0, 
\end{align}
where $[,]$ denotes the commutator in $(\End V)\dbr{z^{\pm1},w^{\pm1}}$
(c.f.\ \eqref{eq:VA:commutator}).

\item
(vacuum axiom)
Denoting by $O(z)$ some power series in $z$, we have
\begin{align}\label{eq:VA:vac}
 Y(\vac,z)=\id_V, \quad Y(a,z)\vac = a + O(z) \in V\dbr{z}. 
\end{align}

\item 
(translation axiom)
We have
\begin{align}\label{eq:VA:pd}
  \pd\vac=0, \quad [\pd,Y(a,z)] = \pd_z Y(a,z),
\end{align}
where $[\cdot,\cdot]$ is the commutator in $\End V$ 
(and its linear extension over $\bbC\dbr{z^{\pm1}}$).
\end{clist}
For simplicity, we sometimes denote the vertex algebra $(V,Y,\vac,\pd)$ by $V$.
We also use the standard abbreviation $a(z) \ceq Y(a,z)$ for $a \in V$.

Let us recall some consequences of the axiom. 
Let $V$ be a vertex algebra, and take any $a,b \in V$.
\begin{itemize}
\item
The translation $\pd$ can be recovered as $\pd a = a_{(-2)}\vac$. 
We also have 
$(\pd a)(z) = \pd_z a(z)$, which implies
\begin{align}\label{eq:VA:pd-com}
 (\pd a)_{(n)} = -n a_{(n-1)} \quad (n \in \bbZ).
\end{align}

\item 
The locality condition implies the commutator formula:
For any $m,n\in\bbZ$, we have \cite[(3.3.12)]{FB}:
\begin{align}\label{eq:VA:commutator}
 [a_{(m)},b_{(n)}] = \sum_{k \ge 0} \binom{m}{k}(a_{(k)}b)_{(m+n-k)}.
\end{align}
Using the formal delta function
\[
 w^{-1}\delta\Bigl(\frac{z-x}{w}\Bigr) \ceq 
 \sum_{n\in\bbZ}\frac{(z-x)^n}{w^{n+1}} = 
 \sum_{m\in\bbN}\sum_{n\in\bbZ}(-1)^m\binom{n}{m}w^{-n-1}z^{n-m}x^m,
\]
it can be rewritten as \cite[(2.3.13)]{FHL}:
\begin{align}\label{eq:VA:com-res}
 [a(z),b(w)] = 
 \res_x \Bigl[w^{-1}\delta\Bigl(\frac{z-x}{w}\Bigr) \bigl(a(x)b\bigr)(w)\Bigr] dx.
\end{align}

\item 
The locality condition implies the Jacobi identity:
\[
 x^{-1}\delta\Bigl(\frac{z-w}{ x}\Bigr)a(z)b(w) - p(a,b)
 x^{-1}\delta\Bigl(\frac{w-z}{-x}\Bigr)b(w)a(z) =
 w^{-1}\delta\Bigl(\frac{z-x}{ w}\Bigr)\bigl(a(x)b\bigr)(w).
\]
Taking $\res_z$ of both sides, we have
\begin{align}\label{eq:VA:Jac-res}
 \bigl(a(x)b\bigr)(w) = 
 \res_z\Bigl[x^{-1}\delta\Bigl(\frac{z-w}{ x}\Bigr)a(z)b(w)\Bigr]dz - p(a,b)
 \res_z\Bigl[x^{-1}\delta\Bigl(\frac{w-z}{-x}\Bigr)b(w)a(z)\Bigr]dz.
\end{align}
The Borcherds identity is a corollary of it:
\begin{align}\label{eq:VA:bor}
\begin{split}
 \sum_{j\ge0} (-1)^j \binom{n}{j}
 \bigl(a_{(m+n-j)}(b_{(k+j)}c)-&(-1)^np(a,b) b_{(n+k-j)}(a_{(m+j)}c)\bigr) \\
&=\sum_{j\ge0}\binom{m}{j}(a_{(n+j)}b)_{(m+k-j)}c.
\end{split}
\end{align}

\item 
We also have the skew-symmetry formula \cite[3.2.5]{FB}.
Using $e^{z\pd}\ceq\sum_{n\ge0}\frac{1}{n!}z^n\pd^n$, we have
\begin{align}\label{eq:VA:skew}
 a(z)b = p(a,b) e^{z\pd}b(-z)a.
\end{align}
In terms of the $(n)$-products, it is equivalent to 
\begin{align}\label{eq:VA:skew2}
 a_{(n)}b = p(a,b) \sum_{j\ge0} (-1)^{j+n+1}\pd^{(j)}(b_{(n+j)}a)
\end{align}
\end{itemize}

\subsection{Huang's formulation of Zhu algebra}\label{ss:wtA:wtA}

Here we explain the geometric version of the Zhu algebra from \cite[\S6]{Hu}, 
with slight modifications.

\subsubsection{}

Let $V$ be a vertex algebra, and let $\gamma$ be a complex number or a formal variable.
For $a,b \in V$ and $n \in \bbZ_{\ge 1}$, consider the expression
\begin{align}\label{eq:wtA:bl}
 a \Hop{\gamma}{n} b \ceq \res_z \bigl[f_n(z;\gamma) a(z)b \bigr] dz, \quad 
 f_n(z;\gamma) \ceq 
 \begin{cases}
  \dfrac{\gamma^n e^{\gamma z}}{(e^{\gamma z}-1)^n} & 
  (\gamma \ne 0 \text{ or $\gamma$ is a formal variable}) \\
  z^{-n} & (\gamma=0)
 \end{cases}.
\end{align}
For the definition of $\res_z$, see \eqref{eq:VA:res}.
When $\gamma$ is a formal variable, 
we regard $f_n(z;\gamma)$ as a formal series of $z$:
\begin{align}\label{eq:wtA:c(j,n)}
 f_n(z;\gamma) = \sum_{j\ge0}c(j,n)\gamma^{j}z^{-n+j} = 
 z^{-n}-\frac{n-2}{2}\gamma^{} z^{-n+1}+\cdots \in \bbQ[\gamma]\dpr{z}
 \quad (c(j,n) \in \bbQ).
\end{align}
For any $n \in \bbZ_{\ge 1}$ and $j \in \bbN=\bbZ_{\ge 0}$, We have 
\begin{align*}
 c(0,n) = 1, \quad c(1,n) = -\frac{n-2}{2}, \quad c(n-1,n) = \delta_{n,1}, \quad 
 c(j,1) = \frac{1}{j!}B_j^+,
\end{align*}
where $B_j^+$ are the Bernoulli numbers defined by the generating function
$\sum_{j\ge0}\frac{1}{j!}B_j^+t^j=\frac{t}{1-e^{-t}}$. 
We have $B_0^+=1$, $B_1^+=\hf$, $B_2^+=\frac{1}{6}$, $B_3^+=0$, $B_4^+=-\frac{1}{30}$, and so on.
Then, the field condition \eqref{eq:VA:f} implies 
\begin{align}\label{eq:wtA:Hopn}
 a\Hop{\gamma}{n}b = \sum_{j=0}^{n+N(a,b)} c(j,n) \gamma^{j}a_{(j-n)}b = 
 a_{(-n)}b-\frac{n-2}{2}\gamma^{} a_{(-n+1)}b + \cdots \in V[\gamma].
\end{align}
In particular, for $n=1$, we have 
\begin{align}\label{eq:wtA:Hop1}
 a\Hop{\gamma}{1}b = \sum_{j\ge0} \gamma^j \frac{B_j^+}{j!} a_{(j-1)}b = 
 a_{(-1)}b+\frac{\gamma}{2}a_{(0)}b+\frac{\gamma^2}{12}a_{(1)}b+\cdots.
\end{align}

\begin{rmk}
Some remarks are in order.
\begin{enumerate}
\item 
The definition of the operation $\Hop{}{n}$ in \cite[\S6]{Hu} 
is given for the case $\gamma=2\pi i$,
and we modified it to accommodate any $\gamma\in\bbC$.
We also modified the operation $\Hop{}{n}$ in loc.\ cit.\ 
by multiplying it by $\gamma^{n-1}$ to define $\Hop{\gamma}{n}$, 
so that $a\Hop{\gamma}{n}b\in V[\gamma]$, 
not just $a\Hop{\gamma}{n}b\in V[\gamma^{\pm1}]$.
Our $\Hop{\gamma}{1}$ coincides with the operation ${\cdot}_{(f)}\cdot$ 
with $f(z) \ceq f_1(z;c)=ce^{cz}/(e^{cz}-1)$,  
introduced in \cite[(3.11)]{vEH} (replacing $c$ with $\gamma$).

\item 
The operation $\Hop{\gamma}{n}$ can be rewritten as 
\begin{align}\label{eq:wtA:blXZ}
 a \Hop{\gamma}{n} b \cdot \gamma^{-n+1} = 
 \res_z \biggl[ \frac{\gamma e^{\gamma z}}{(e^{\gamma z}-1)^n}a(z)b \biggr] dz = 
 \res_x \bigl[ x^{-n} a\bigl(\gamma^{-1}\log(1+x)\bigr)b \bigr] dx.
\end{align}
This follows from the change of variables $x = e^{\gamma z}-1$; 
see \eqref{eq:wtA:f} for its geometric interpretation.
Furthermore, this identity extends the definition 
of $\Hop{\gamma}{n}$ to all $n \in \bbZ$.
\end{enumerate}
\end{rmk}

The following lemma will be useful in the argument below.

\begin{lem}
Let $V$ be a vertex algebra, and $\gamma$ be a complex number or a formal variable.
\begin{enumerate}
\item 
The vacuum $\vac$ of $V$ is the unit for the operation $\Hop{\gamma}{1}$.
In other words, for any $a \in V$, we have
\begin{align}\label{eq:wtA:Hop-vac}
 a \Hop{\gamma}{1} \vac = \vac \Hop{\gamma}{1} a = a.
\end{align}

\item
Let $\pd$ be the translation operator of $V$.
Then, for any $a,b \in V$ and $n \in \bbZ_{\ge1}$, we have
\begin{align}\label{eq:wtA:pd-Hop}
 (\pd a)\Hop{\gamma}{n}b = (n-1)\gamma a\Hop{\gamma}{n}b + n a\Hop{\gamma}{n+1}b.
\end{align}
\end{enumerate}
\end{lem}

\begin{proof}
We follow the fourth and fifth paragraphs of the proof of \cite[Proposition 6.1]{H}.
\begin{enumerate}
\item follows immediately from the vacuum axiom \eqref{eq:VA:vac} 
and the expression \eqref{eq:wtA:Hop1} of $\Hop{\gamma}{1}$.

\item 
We denote $f_n(z)\ceq f_n(z;\gamma)$, so that $a\Hop{\gamma}{n}b=\res_z[f_n(z)a(z)b]dz$.
If $\gamma$ is a non-zero complex number or a formal variable, 
then $f_n(z) = \gamma^n e^{\gamma z}/(e^{\gamma z}-1)^n$, and we have
\begin{align*}
   \pd_z f_n(z) 
&= \frac{ \gamma^{n+1}e^{\gamma z}}{(e^{\gamma z}-1)^n} -
   \frac{n\gamma^{n+1}e^{2\gamma z}}{(e^{\gamma z}-1)^{n+1}}  
 = \frac{ \gamma^{n+1}e^{\gamma z}}{(e^{\gamma z}-1)^n} -
   \frac{n\gamma^{n+1}e^{\gamma z}(e^{\gamma z}-1+1)}{(e^{\gamma z}-1)^{n+1}} \\
&=-\frac{(n-1)\gamma^{n+1}e^{\gamma z}}{(e^{\gamma z}-1)^n}  
  -\frac{n\gamma^{n+1} e^{\gamma z}}{(e^{\gamma z}-1)^{n+1}} = 
  - (n-1) \gamma f_n(z) - n f_{n+1}(z).
\end{align*}
If $\gamma=0$, then $f_n(z)=z^{-n}$, and we have the same relation.
Now, by the translation axiom \eqref{eq:VA:pd} and 
applying integration by parts, we have:
\begin{align*}
  (\pd a)\Hop{\gamma}{n}b 
&= \res_z\bigl[f_n(z)(\pd a)(z)b\bigr]dz
 = \res_z\bigl[f_n(z)\bigl(\pd_za(z)\bigr)b\bigr]dz
 =-\res_z\bigl[\bigl(\pd_z f(z)\bigr)a(z)b\bigr]dz \\
&= \res_z\bigl[\bigl((n-1) \gamma f_n(z) + n f_{n+1}(z)\bigr)a(z)b\bigr]dz 
 = (n-1) \gamma a\Hop{\gamma}{n}b + n a\Hop{\gamma}{n+1}b.
\end{align*}
\end{enumerate}
\end{proof}

Now, the following definitions make sense.

\begin{dfn}[{c.f.\ \cite[\S6]{Hu}}]\label{dfn:VA:wtA}
Let $V$ be a vertex algebra, and let $\gamma$ be a formal variable.
For $n \in \bbZ_{\ge1}$, we define
\begin{align*}
 V \Hop{\gamma}{n} V \ceq 
 \Span_{\bbC[\gamma]}\{a\Hop{\gamma}{n}b \mid a,b \in V\} \subset V[\gamma].
\end{align*}
We also define 
\begin{align}\label{eq:wtA:wtO}
 \wtO_\gamma(V) \ceq V \Hop{\gamma}{2} V, \quad
 \wtA_\gamma(V) \ceq V[\gamma]/\wtO_\gamma(V), 
\end{align}
and denote the equivalence class of $a \in V$ in the quotient space 
by $[a] \in \wtA_\gamma(V)$.
\end{dfn}

The next theorem provides an analogue, for our $\wtO_\gamma(V)$, 
of the properties of the subspace $O(V) \subset V$ used in the definition
of the original Zhu algebra $A(V)$ (or $\Zhu_H(V)$), as stated in \cite[Theorem 2.7]{DK}.

\begin{thm}\label[thm]{thm:wtA:wtO}
Let $V$ be a vertex algebra, and let $\wtO_\gamma(V)$ be as defined in \eqref{eq:wtA:wtO}.
\begin{enumerate}
\item \label{i:thm:wtO:ge2}
$\wtO_\gamma(V)$ contains all the elements $a\Hop{\gamma}{n}b$ 
for any $n\ge2$ and $a,b\in V$,
and $\wtO_\gamma(V)=\sum_{n\ge2}V\Hop{\gamma}{n}V$.

\item \label{i:thm:wtO:pd}
$\pd V$ is contained in $\wtO_\gamma(V)$, where $\pd$ is the translation operator of $V$. 

\item \label{i:thm:wtO:ideal}
$\wtO_\gamma(V)$ is a two-sided ideal with respect to the operation $\Hop{\gamma}{1}$.

\item \label{i:thm:wtO:assoc}
For any $a,b,c\in V$, we have
$(a\Hop{\gamma}{1}b)\Hop{\gamma}{1}c-a\Hop{\gamma}{1}(b\Hop{\gamma}{1}c) \in \wtO_\gamma(V)$.

\item \label{i:thm:wtO:(0)}
For any $a,b\in V$, we have
$a\Hop{\gamma}{1}b-p(a,b) b\Hop{\gamma}{1}a - \gamma a_{(0)}b \in \wtO_\gamma(V)$.
\end{enumerate}
\end{thm}

\begin{proof}
For simplicity, we denote $\wtO_\gamma\ceq\wtO_\gamma(V)$.
\begin{enumerate}
\item follows from $V\Hop{\gamma}{n+1}V \subset V\Hop{\gamma}{n}V$ for all $n\in\bbZ_{\ge2}$,
which is implied by the formula \eqref{eq:wtA:pd-Hop}.

\item 
The formulas \eqref{eq:wtA:Hop-vac} and \eqref{eq:wtA:pd-Hop} imply
$\pd a = (\pd a)\Hop{\gamma}{1}\vac = a\Hop{\gamma}{2}\vac \in \wtO_\gamma$ 
for any $a \in V$.

\item
For simplicity, we denote 
$f_n(z) \ceq f_n(z;\gamma) = \gamma^n e^{\gamma z}/(e^{\gamma z}-1)^n$,
so that $a\Hop{\gamma}{n}b=\res_z[f_n(z)a(z)b]dz$.
We follow the first paragraph of the proof of \cite[Proposition 6.1]{H} 
(where the case $\gamma=2\pi i$ is studied)
to show that $\wtO_\gamma$ is a left ideal with respect to $\Hop{\gamma}{1}$, 
that is, $V\Hop{\gamma}{1}\wtO_\gamma \subset \wtO_\gamma$.
For any $a,b,c \in V$ and $n\in\bbZ_{\ge1}$, 
the commutator formula \eqref{eq:VA:com-res} implies 
\begin{align}
\notag
  a\Hop{\gamma}{1}(b\Hop{\gamma}{n}c) 
&=\res_z \bigl[\res_w [f_1(z)f_n(w)a(z)b(w)c] dw \bigr] dz \\
\label{eq:wtA:1n-1}
&=p(a,b)\res_z \bigl[\res_w [f_1(z)f_n(w)b(w)a(z)c] dw \bigr] dz \\
\label{eq:wtA:1n-2}
& \quad 
 +\res_z \Bigl[\res_w \Bigl[\res_x \Bigl[
   f_1(z)f_n(w) w^{-1}\delta\Bigl(\frac{z-x}{w}\Bigr)\bigl(a(x)b\bigr)(w)c
  \Bigr] dx \Bigr] dw \Bigr] dz.
\end{align}
Now, we assume $n \ge 2$.
Then, note that the first term \eqref{eq:wtA:1n-1} of the right hand side is equal to
$\res_z\bigl[f_1(z) [\res_w [f_n(w)b(w)a(z)c] dw \bigr] dz$, and
the expression $\res_w [f_n(w) b(w)a(z)c] dw$ belongs to $(V\Hop{\gamma}{n}V)[z^{\pm1}]$.
Hence $\eqref{eq:wtA:1n-1} \in V\Hop{\gamma}{n}V$, and by \ref{i:thm:wtO:ge2} above, 
we have $\eqref{eq:wtA:1n-1} \in \wtO_\gamma$.
On the other hand, the second term \eqref{eq:wtA:1n-2} is equal to 
$\res_z\bigl[f_1(z) [\res_x [f_n(z-x)\bigl(a(x)b\bigr)(z-x)c] dx \bigr] dz$, 
and a similar argument shows $\eqref{eq:wtA:1n-2} \in \wtO_\gamma$.
Since the left hand side spans $V\Hop{\gamma}{1}\wtO_\gamma$, 
we have $V\Hop{\gamma}{1}\wtO_\gamma \subset \wtO_\gamma$.

Next, to show that $\wtO_\gamma$ is a right ideal, 
we follow the second paragraph of the proof of \cite[Proposition 6.1]{H}.
For any $a,b,c \in V$ and $n \in \bbZ_{\ge 1}$, 
applying the Jacobi identity \eqref{eq:VA:Jac-res}, we can compute as follows: 
\begin{align}
\notag
&(a\Hop{\gamma}{n}b)\Hop{\gamma}{1}c 
 =\res_w\bigl[\res_z\bigl[f_1(w)f_n(z)\bigl(a(z)b\bigr)(w)c\bigr]dz\Bigr]dw \\
\notag
&=\res_x\Bigl[\res_w\Bigl[\res_z\Bigl[
   f_1(w)f_n(z)z^{-1}\delta\Bigl(\frac{x-w}{ z}\Bigr)a(x)b(w)c\Bigr]dz\Bigr]dw\Bigr]dx \\
\notag
&\quad -p(a,b)
  \res_x\Bigl[\res_w\Bigl[\res_z\Bigl[
   f_1(w)f_n(z)z^{-1}\delta\Bigl(\frac{w-x}{-z}\Bigr)b(w)a(x)c\Bigr]dz\Bigr]dw\Bigr]dx \\
\notag
&=\res_x\bigl[\res_w\bigl[f_1(w)f_n(x-w)a(x)b(w)c\bigr]dw\bigr]dx \\
\notag
&\quad - p(a,b) 
  \res_x\bigl[\res_w\bigl[f_1(w)f_n(-(w-x))b(w)a(x)c\bigr]dw\bigr]dx \\
\notag
&=\res_x\biggl[\res_w\biggl[
   \frac{\gamma^n e^{\gamma(x-w)}}
        {\bigl((e^{\gamma x}-1)e^{-\gamma w}+(e^{-\gamma w}-1)\bigr)^n}
   f_1(w)a(x)b(w)c\biggr]dw\biggr]dx \\
\notag
&\quad - p(a,b) 
  \res_x\biggl[\res_w\biggl[
   \frac{\gamma^n e^{-\gamma(w-x)}}
        {\bigl((e^{-\gamma w}-1)e^{\gamma x}+(e^{\gamma x}-1)\bigr)^n}
   f_1(w)b(w)a(x)c\biggr]dw\biggr]dx \\
\label{eq:wtA:n1-1}
&=\sum_{k\ge0}\binom{-n}{k} \res_x\biggl[\res_w\biggl[
   \frac{\gamma^n e^{\gamma(x-w)}e^{(n+k)\gamma w}}{(e^{\gamma x}-1)^{n+k}}
    (e^{-\gamma w}-1)^k f_1(w)a(x)b(w)c\biggr]dw\biggr]dx \\
\label{eq:wtA:n1-2}
&\quad - p(a,b) 
  \sum_{k\ge0}\binom{-n}{k} \res_x\biggl[\res_w\biggl[
   \frac{\gamma^n e^{-\gamma(w-x)}e^{-(n+k)\gamma x}}{(e^{-\gamma w}-1)^{n+k}}
   (e^{\gamma x}-1)^k f_1(w)b(w)a(x)c\biggr]dw\biggr]dx.
\end{align}
Then, each summand in the first sum \eqref{eq:wtA:n1-1} is equal to 
\begin{align}\label{eq:wtA:n1-1zw}
 \res_x\bigl[\res_w\bigr[f_{n+k}(x)g_1(w)a(x)b(w)c\bigr]dw\bigr]dx
\end{align}
with $g_1(w) \ceq 
 f_1(w) e^{-\gamma w} e^{(n+k)\gamma w} \cdot \gamma^{-k} (e^{-\gamma w}-1)^k
 \in \bbQ[\gamma]\dpr{w}$.
Then, since the inner residue $\res_w\bigl[g_1(w)b(w)c\bigr]dw$
belongs to $V[\gamma]$, 
we have $\eqref{eq:wtA:n1-1zw} \in V\Hop{\gamma}{n+k}V$. 
Hence, by \ref{i:thm:wtO:ge2} above, 
the sum \eqref{eq:wtA:n1-1} belongs to $\wtO_\gamma$ if $n \ge 2$.
Similarly, each summand in the second sum \eqref{eq:wtA:n1-2} is equal to 
\[
 (-1)^{n+k} \res_x\bigl[\res_w\bigl[h(w)g_2(x)a(x)b(w)c\bigr]dw\bigr]dx
\]
with 
$h(w)\ceq \gamma^{n+k+1}e^{(n+k)\gamma w}/(e^{\gamma w}-1)^{n+k+1}$ and 
$g_2(x) \ceq 
 e^{-(n+k-1)\gamma x} \cdot \gamma^{-k}(e^{\gamma x}-1)^k \in \bbC[\gamma]\dbr{x}$.
We see that $h(w)$ is a $\bbZ[\gamma]$-linear combination of 
$f_2(w),f_3(w),\dotsc,f_{n+k+1}(w)$ if $n+k \ge 1$.
Thus, the sum \eqref{eq:wtA:n1-2} belongs to $\wtO_\gamma$ by \ref{i:thm:wtO:ge2} above.
Hence, when $n \ge 2$, we have $(a\Hop{\gamma}{n}b)\Hop{\gamma}{1}c \in \wtO_\gamma$,
which shows that $\wtO_\gamma\Hop{\gamma}{1}V \subset \wtO_\gamma$.

\item 
We follow the third paragraph of the proof of \cite[Proposition 6.1]{Hu}.
We set $n=1$ in the latter part of the proof of \ref{i:thm:wtO:ideal} above.
Then, \eqref{eq:wtA:n1-1} is a sum of 
$a\Hop{\gamma}{1}(b\Hop{\gamma}{1}c)$ and some elements in $\wtO_\gamma$,
and \eqref{eq:wtA:n1-2} belongs to $\wtO_\gamma$.
Hence, we have 
$(a\Hop{\gamma}{1}b)\Hop{\gamma}{1}c-a\Hop{\gamma}{1}(b\Hop{\gamma}{1}c)\in\wtO_\gamma$.

\item
Following the sixth paragraph of the proof of \cite[Proposition 6.1]{Hu}, 
we have 
\begin{align*}
 a\Hop{\gamma}{1}b 
&=\res_z\bigl[f_1(z)a(z)b\bigr]dz
 =p(a,b)\res_z\bigl[f_1(z)e^{z\pd}b(-z)a\bigr]dz \\
&\equiv p(a,b)\res_z\bigl[f_1(z)b(-z)a\bigr]dz \bmod \wtO_\gamma 
 =p(a,b)\res_w\bigl[-f_1(-w)b(w)a\bigr]dw \\ 
&=p(a,b)\res_w\bigl[(f_1(w)-\gamma)b(w)a\bigr]dw
 =p(a,b)\bigl(b\Hop{\gamma}{1}a-\gamma b_{(0)}a\bigr) \\
&\equiv p(a,b) b\Hop{\gamma}{1}a+\gamma a_{(0)}b \bmod \wtO_\gamma.
\end{align*}
In the second equality, we used the skew-symmetry \eqref{eq:VA:skew}.
In the third equivalence, we used \ref{i:thm:wtO:pd} above.
In the fourth equality, we changed the variable as $w\ceq-z$.
In the final equivalence, we used $b_{(0)}a\equiv -p(a,b)a_{(0)}b \bmod \wtO_\gamma$
which follows from the skew-symmetry \eqref{eq:VA:skew2} and \ref{i:thm:wtO:pd} above.
\end{enumerate}
\end{proof}

By \cref{thm:wtA:wtO} \ref{i:thm:wtO:ideal} and \ref{i:thm:wtO:assoc}, 
we have an associative $\bbC[\gamma]$-superalgebra
\[
 \wtA_\gamma(V) \ceq V[\gamma]/\wtO_\gamma(V),
\]
with the product induced by the operation $\Hop{\gamma}{1}$ on $V[\gamma]$ 
(extended linearly over $\bbC[\gamma]$):
\[
 [a] \Hop{\gamma}{} [b] \ceq [a \Hop{\gamma}{1} b] \in \wtA_\gamma(V) 
 \quad (a,b \in V[\gamma]).
\] 
Moreover, by \eqref{eq:wtA:Hop-vac}, 
the image $[\vac]$ of the vacuum is the unit of this superalgebra.

\begin{rmk}\label[rmk]{rmk:wtA:geom}
Some remarks on the product $\Hop{\gamma}{}$ are in order.
\begin{enumerate}
\item 
\cite[\S3]{vEH} presents a quite different proof of 
\cref{thm:wtA:wtO} \ref{i:thm:wtO:ideal} and \ref{i:thm:wtO:assoc}.

\item
As explained in \cite[Remark 6.2]{Hu}, the product $\Hop{\gamma}{}$ on $\wtA_\gamma(V)$ 
has the following geometric interpretation.
Consider the mutually inverse maps  
\begin{align}\label{eq:wtA:f}
 f\colon z \lmto x \ceq e^{\gamma z}-1, \quad 
 g\colon x \lmto z \ceq \gamma^{-1}\log(1+x).
\end{align}
Then $g$ maps an annulus in $\bbC \subset \bbP^1$ with non-homogeneous coordinate $x$
to a parallelogram in the universal covering of a torus $T$ with coordinate $z$.
Then the product $\Hop{\gamma}{}$ can be understood as the constant term of the pullback 
of the local vertex operator on $T$ by the map $g$ to the annulus.
\end{enumerate}
\end{rmk}

The definition of $\wtA_\gamma(V)$ can be specialized to 
$\gamma=c$ for any $c \in \bbC\setminus\{0\}$.
Explicitly, let $\wtO(V)_{\gamma=c}$ be the $\bbC$-span of 
the elements $\{a\Hop{\gamma=c}{2}b \mid a,b \in V\}$.
Then 
\[
 \wtA_{\gamma=c}(V) \ceq V/\wtO_{\gamma=c}(V)
\]
is a unital associative $\bbC$-superalgebra
with the product $\Hop{\gamma=c}{}$ induced by $\Hop{\gamma=c}{1}$ on $V$.
The case $c=1$ will be the primary focus of our discussion.

\begin{dfn}\label[dfn]{dfn:wtA:wtA}
The Zhu algebra of a vertex algebra $V$ is the unital associative $\bbC$-superalgebra
\[
 (\wtA(V),\bl,1) \ceq (\wtA_{\gamma=1}(V),\Hop{\gamma=1}{},[\vac]).
\] 
\end{dfn}

We collect some few basic properties of the algebra $\wtA(V)$.

\begin{cor}\label[cor]{cor:wtA:pd=0}
Let $(V,Y,\vac,\pd)$ be a vertex algebra.
The following formulas hold in $\wtA(V)$ for any $a,b \in V$. 
\begin{align}\label{eq:wtA:pd=0}
&[\pd a] = 0,   \\
\label{eq:wtA:a0b}
&[[a],[b]]\ceq [a] \bl [b] - p(a,b) [b] \bl [a] = [a_{(0)}b].
\end{align}

\end{cor}

\begin{proof}
These follow from \cref{thm:wtA:wtO} \ref{i:thm:wtO:pd} and \ref{i:thm:wtO:(0)}.
\end{proof}

\begin{rmk}
As will be recalled in \cref{fct:A:A} of \cref{ss:wtA:A}, 
formulas similar to \eqref{eq:wtA:pd=0} and \eqref{eq:wtA:a0b} 
hold in the original Zhu algebra $A(V)$ (or $\Zhu_H(V)$).
However, the corresponding formulas \eqref{eq:A:T+H-com} take a more complicated form.
Therefore, we may expect that determining $\wtA(V)$ is easier than determining $\Zhu_H(V)$.
In \cref{s:N=124} below, we see that this is indeed the case, 
at least when $V$ is an $N=1,2$ or $4$ superconformal vertex algebra.
\end{rmk}

We can also consider the specialization $\gamma=0$.
In this case, we have $a\Hop{\gamma=0}{n}b = a_{(n)}b$ for $a,b \in V$,  so that
\[
 \wtO_{\gamma=0}(V) = C_2(V) \ceq \Span_{\bbC}\{a_{(-2)}b \mid a,b \in V\}, \quad 
 \wtA_{\gamma=0}(V) =   R(V) \ceq V/C_2(V).
\]
The statements in \cref{thm:wtA:wtO} still hold, 
and we have an associative product 
\[
 [a]\cdot[b] \ceq [a \Hop{\gamma=0}{1}b] = [a_{(1)}b] \quad (a,b \in V).
\]
The algebra $R(V)$ is nothing but the $C_2$-Poisson algebra $R(V)$ \cite[\S4.4]{Zhu}. 
The product $\cdot$ is supercommutative, 
and $R(V)$ is equipped with a Poisson bracket $\{\cdot,\cdot\}$ given by
\begin{align}\label{eq:wtA:C2-PB}
 \{[a],[b]\} \ceq [a_{(0)}b] \quad (a,b \in V).
\end{align}
Hence, $\wt{A}_{\gamma=0}(V)$ is the semiclassical limit of
the algebra $\wtA_{\gamma=c}(V)$ under $c \to 0$.
Let us summarize as:

\begin{prp}\label[prp]{prp:wtA:gamma=0}
Let $V$ be a vertex algebra. 
Then the associative algebra $\wt{A}_{\gamma=0}(V)$ is isomorphic to 
the $C_2$-Poisson algebra $R(V)$.
Hence, it is supercommutative, 
and equipped with the Poisson bracket \eqref{eq:wtA:C2-PB}.
\end{prp}

Recall (from \cite[1.3.6 Lemma]{FB} for example) 
the tensor product $V \otimes W$ of vertex algebras $V$ and $W$,
whose state-filed correspondence is $(a \otimes b)(z) \ceq a(z)b(z)$ 
for $a \in V$ and $b \in W$.

\begin{prp}\label{prp:wtA:wtA-otimes}
Let $V$ and $W$ be  vertex algebras, and let $\gamma$ be an indeterminate. 
Then, we have a $\bbC[\gamma]$-superalgebra isomorphism
\begin{align*}
 \varphi\colon 
 \wtA_\gamma(V \otimes W) \lsto \wtA_\gamma(V) \otimes \wtA_\gamma(W), \quad 
 [a \otimes b] \lmto [a] \otimes [b].
\end{align*}
\end{prp}

\begin{proof}
We first show that, as subspaces of $V \otimes W$, we have
\begin{align}\label{eq:wtA:OVoW}
 \wtO_\gamma(V \otimes W) = \wtO_\gamma(V) \otimes W + V \otimes \wtO_\gamma(W).
\end{align}
The relation $\wtO_\gamma(V \otimes W) \supset \wtO_\gamma(V) \otimes W$ follows from 
\begin{align*}
 (a \otimes \vac) \Hop{\gamma}{2} (a' \otimes b')
&=\sum_{j\ge0}c(j,2)\gamma^j(a\otimes\vac_W)_{(-2+j)}(a'\otimes b') \\
&=\sum_{j\ge0}c(j,2)\gamma^j\sum_{k\in\bbZ}(a_{(k)}a')\otimes((\vac_W)_{(-3+j-k)}b') \\ 
&=\sum_{j\ge0}c(j,2)\gamma^j(a_{(j-2)}a')\otimes b' = (a\Hop{\gamma}{2}a')\otimes b 
 \in \wtO_\gamma(V) \otimes W
\end{align*}
for $a,a' \in V$ and $b'\in W$, 
where we used $c(j,2)$ defined in \eqref{eq:wtA:c(j,n)}, 
and used $(\vac_W)_{(n)}=\delta_{-1,n}\id_W$ in the second equality.
A similar argument shows $\wtO_\gamma(V \otimes W) \supset V \otimes \wtO_\gamma(W)$.
To show the other relation 
$\wtO_\gamma(V \otimes W) \subset \wtO_\gamma(V) \otimes W + V \otimes \wtO_\gamma(W)$,
note that, for any Laurent series $f(x)$ and $g(x)$ in $x$, we have 
\[
 \res_x[f(x)g(x)]dx = 
 \sum_{i \in \bbZ}\bigl(\res_x[x^{-i}f(x)]dx\bigr)\bigl(\res_x[x^{i-1}g(x)]dx\bigr).  
\]
Also, recall the expression \eqref{eq:wtA:blXZ} for the product $\Hop{\gamma}{2}$.
Then, for $a,a'\in V$ and $b,b'\in W$, we have
\begin{align}\label{eq:wtA:aa2bb}
\begin{split}
&(a\otimes a')\Hop{\gamma}{2}(b\otimes b')
 =\gamma \res_x[x^{-2}(a\otimes a')(\gamma^{-1}\log(1+x))(b\otimes b')]dx \\ 
&=\gamma \res_x[x^{-2}(a(\gamma^{-1}\log(1+x))b)\otimes(a'(\gamma^{-1}\log(1+x))b')]dx \\
&=\sum_{i\in\bbZ}
  \bigl(\gamma^{i+1}\res_x[x^{-i-2}a (\gamma^{-1}\log(1+x))b ]dx\bigr) \otimes 
  \bigl(\gamma^{ -i}\res_x[x^{ i-1}a'(\gamma^{-1}\log(1+x))b']dx\bigr) \\
&=\sum_{i\in\bbZ}(a\Hop{\gamma}{i+2}b)\otimes(a'\Hop{\gamma}{-i+1}b').
\end{split}
\end{align}
Note that, by \cref{thm:wtA:wtO} \ref{i:thm:wtO:ge2}, we have 
$(a\Hop{\gamma}{i+2}b)\otimes(a'\Hop{\gamma}{-i+1}b') \in \wtO_\gamma(V) \otimes W$
for $i\ge0$, and 
$(a\Hop{\gamma}{i+2}b)\otimes(a'\Hop{\gamma}{-i+1}b') \in V\otimes\wtO_\gamma(W)$
for $i<0$.
Hence we have 
$\wtO_\gamma(V \otimes W) \subset \wtO_\gamma(V) \otimes W + V \otimes \wtO_\gamma(W)$,
and \eqref{eq:wtA:OVoW} is proved.

Now, we have a well-defined $\bbC[\gamma]$-linear isomorphism
$\varphi\colon\wtA_\gamma(V \otimes W) \to \wtA_\gamma(V) \otimes \wtA_\gamma(W)$,  
$[a \otimes b] \mto [a] \otimes [b]$.
To show that this is a superalgebra morphism, 
for $a,a'\in V$ and $b,b'\in W$, we compute as in \eqref{eq:wtA:aa2bb}:
\begin{align*}
 [(a\otimes a')\Hop{\gamma}{1}(b\otimes b')]
&=\sum_{i\in\bbZ}
  \bigl[\bigl(\gamma^{ i}\res_x[x^{-i-1}a (\gamma^{-1}\log(1+x))b ]dx\bigr) \otimes 
        \bigl(\gamma^{-i}\res_x[x^{ i-1}a'(\gamma^{-1}\log(1+x))b']dx\bigr)\bigr] \\
&=\sum_{i\in\bbZ}\bigl[(a\Hop{\gamma}{i+1}b)\otimes(a'\Hop{\gamma}{-i+1}b')\bigr] 
  \stackrel{\varphi}{\lmto} 
  \sum_{i\in\bbZ}\bigl[a\Hop{\gamma}{i+1}b\bigr]\otimes\bigl[a'\Hop{\gamma}{-i+1}b'\bigr]
 =\bigl[a\Hop{\gamma}{1}b\bigr]\otimes\bigl[a'\Hop{\gamma}{1}b'\bigr].
\end{align*}
In the last equality, we used \cref{thm:wtA:wtO} \ref{i:thm:wtO:ge2}.
Hence $\varphi$ is a $\bbC[\gamma]$-superalgebra isomorphism.
\end{proof}

\begin{rmk}
For the original Zhu algebra (see \cref{ss:wtA:A} below for the details) 
of vertex operator algebras $V$ and $W$, 
Milas \cite[Theorem 4.1]{Mi96} showed that $A(V\otimes W) \cong A(V) \otimes A(W)$.
Our proof of \cref{prp:wtA:wtA-otimes} is a modification of his argument.
\end{rmk}

\subsubsection{}

In \cref{s:N=124}, we will determine $\wtA(V)$ 
for $N\leq 4$ superconformal vertex algebras $V$.
As its preliminary, we recall the $\lambda$-bracket and the related notions briefly.
See \cite[\S2.3]{K}, \cite[\S1.5]{DK} and \cite[\S2]{HK} for details.

Let $(V,Y,\vac,\pd)$ be a vertex algebra.
\begin{itemize}
\item
Let $\lambda$ be an indeterminate.
The \emph{$\lambda$-bracket} $[a_\lambda b] \in V[\lambda]$ of $a,b\in V$ is defined as
\begin{align}\label{eq:wtA:lam-br}
 [a_\lambda b] \ceq \sum_{n\ge0}\frac{1}{n!}\lambda^n a_{(n)}b.
\end{align}
Note that this is indeed a polynomial of $\lambda$ by the field condition \eqref{eq:VA:f}.

\item 
An even element $L \in V^{\ev}$ is called \emph{conformal} 
with \emph{central charge} $c \in \bbC$ if 
\begin{align}\label{eq:wtA:L}
 [L_\lambda L]=(\pd+2\lambda)L+\frac{c}{12}\lambda^3.
\end{align}
We will use the standard notation
\begin{align}\label{eq:wtA:L(z)}
 L(z) = \sum_{n\in\bbZ}z^{-n-2}L_n, \quad L_n=L_{(n+1)}\in\End V.
\end{align} 

\item
Given a conformal element $L \in V$, 
an element $a \in V$ is said to have \emph{conformal weight} $\Delta \in \bbC$ if 
\begin{align}\label{eq:wtA:cw}
 [L_\lambda a]=(\pd+\Delta \lambda)a+O(\lambda^2).
\end{align}
An element $a \in V$ is called \emph{primary} of conformal weight $\Delta \in \bbC$ if
\begin{align}\label{eq:wtA:primary}
 [L_\lambda a]=(\pd+\Delta \lambda)a.
\end{align}
\end{itemize}

The following statements are standard for the original Zhu algebra $\Zhu_H(V)$ \cite[\S2.1]{Zhu}
(see also \cref{ss:wtA:A}).

\begin{prp}\label[prp]{prp:wtA:LinZ}
Let $V$ be a vertex algebra, and $L \in V$ be a conformal element.
\begin{enumerate}
\item 
If $a \in V$ has conformal weight $\Delta\in\bbC$, 
then $[L]\bl[a]-[a]\bl[L]=0$ in $\wtA(V)$.
\item
If we have a decomposition $V=\bigoplus_{\Delta\in\bbC}V_\Delta$ with $V_\Delta$ consisting of 
elements of conformal weight $\Delta$, then $[L]$ is a central element of $\wtA(V)$.
\end{enumerate}
\end{prp}

\begin{proof}
(1) By \eqref{eq:wtA:a0b} and the assumption,
we have $[L]\bl[a]-[a]\bl[L]=[L_{(0)}a]=[\pd a]$.
Then, by \eqref{eq:wtA:pd=0}, it is equal to $0$.
(2) follows from (1).
\end{proof}

\subsection{Relationship with the original Zhu algebra}\label{ss:wtA:A}

Here we explain the relationship between Huang's algebra $\wtA(V)$
and Zhu's original algebra \cite[\S2]{Zhu}.
For the original algebra, we present the generalized version $\Zhu_H(V)$ 
introduced by De Sole and Kac \cite{DK}.

\subsubsection{}

Let $V$ be a vertex algebra.
\begin{itemize}
\item 
A \emph{Hamiltonian} $H$ of $V$ is an even semisimple linear operator 
$H\in(\End V)^{\ev}$ such that 
\begin{align}\label{eq:wtA:H}
 [H,Y(a,z)] = z\pd_z Y(a,z)+Y(Ha,z) 
\end{align}
for any $a \in V$.
We denote the $H$-eigenspace of the eigenvalue $\Delta\in\bbC$ by $V_\Delta$.

\item
For $a \in V_\Delta$, $b\in V$ and $m \in \bbZ$, we define 
\begin{align*}
 a \Zop{H}{m} b \ceq \sum_{j\ge0}\binom{\Delta}{j}a_{(j+m)}b \in V, 
\end{align*}
and extend the operation $\Zop{H}{m}$ by linearity.
Note that is $a \Zop{H}{m} b$ a finite sum by the field condition \eqref{eq:VA:f}.

\item
We define the quotient space $\Zhu_H(V)$ of $V$ by 
\begin{align*}
 \Zhu_H(V) \ceq V/O_H(V), \quad O_H(V) \ceq \Span_{\bbC}\{a\Zop{H}{-2}b \mid a,b \in V\}.
\end{align*}
We denote the equivalence class of $a \in V$ by $[a] \in \Zhu_H(V)$.
\end{itemize}

\begin{eg}\label[eg]{eg:wtA:CVA}
We give a standard example of Hamiltonian.
Assume that a vertex algebra $(V,Y,\vac,\pd)$ is equipped with a conformal element 
$L$ \eqref{eq:wtA:L} such that $L_0=L_{(1)}\in\End(V)$ is semisimple,
and that $L_{-1}=L_{(0)}\in\End(V)$ is equal to the translation operator $\pd$.
Then $L_0$ is a Hamiltonian of $V$, and we have the quotient space $\Zhu_{L_0}(V)$.
\end{eg}

\begin{fct}[{\cite[Theorem 2.1.1]{Zhu}, \cite[Theorem 2.7 (c)]{DK}}]\label[fct]{fct:A:A}
Let $V$ be a vertex algebra with Hamiltonian $H$.
Then, the even linear map $\Zop{H}{-1}\colon V \otimes V \to V$ 
induces a well-defined even linear map
\begin{align*}
 *\colon \Zhu_H(V) \otimes \Zhu_H(V) \lto \Zhu_H(V),
\end{align*}
and the triple $(\Zhu_H(V),*,[\vac])$ is a unital associative $\bbC$-superalgebra.
Moreover, for $a,b \in V$, we have 
\begin{align}\label{eq:A:T+H-com}
 [(T+H)(a)] = 0, \quad [a]\Zop{H}{}[b]-p(a,b)[b]\Zop{H}{}[a] = [a_Hb] 
 \quad \text{in $\Zhu_H(V)$}, 
\end{align}
where $a_Hb \ceq \sum_{j\ge0}\binom{\Delta-1}{j}a_{(j)}b \in V$ for $a \in V_\Delta$
(and extend the operation ${\cdot}_H{\cdot}$ by linearity).
\end{fct}

To state the relationship between $\Zhu_H(V)$ and $\wtA_\gamma(V)$,
let us assume that $V$ is a vertex operator algebra.
In other words, $V$ is a vertex algebra equipped with a conformal element $L$ 
that satisfies the conditions in \cref{eg:wtA:CVA}, and furthermore, 
the eigenvalues $\Delta$ of $L_0$ lie in $\bbN=\bbZ_{\ge0}$. 
In particular, we have the $L_0$-eigenspace decomposition 
\begin{align*}
 V = \bigoplus_{\Delta \in \bbN} V_\Delta, \quad V_\Delta \ceq \{v \in V \mid L_0 v = \Delta v\}.
\end{align*}
Then, for $\gamma_0\in\bbC\setminus\{0\}$, we define 
\begin{align}\label{eq:Zhu:U}
 U_{\gamma_0}(x) \ceq (\gamma_0 x)^{L_0} \exp\bigl(-\tsum_{j>0}c_j L_j\bigr),
\end{align}
where $L_j \in \End V$ are given by $L(z)=\sum_{j \in \bbZ}z^{-j-2}L_j$ \eqref{eq:wtA:L(z)}, 
and $c_j \in \bbC$ are given by the relation 
\begin{align*}
 \gamma_0^{-1} \log(1+\gamma_0 y) = \exp\bigl(\tsum_{j>0}c_j y^{j+1}\pd_y\bigr)y.
\end{align*}
Explicitly, $c_1 = -\frac{1}{2}\gamma_0$, $c_2=\frac{1}{12}\gamma_0^2$, 
$c_3=-\frac{1}{48}\gamma_0^3$ and so on.
$U_{\gamma_0}(x)$ is an invertible operator on $V$.

\begin{fct}[{\cite[Proposition 6.3]{H}}]\label[fct]{fct:Zhu:wtA=Zhu}
Let $V$ be a vertex operator algebra with conformal vector $L$,
and $\gamma_0\in\bbC\setminus\{0\}$. 
Then, the operator $U \ceq U_{\gamma_0}(1)$ in \eqref{eq:Zhu:U} satisfies
\begin{align*}
 U(a\Hop{\gamma=\gamma_0}{n}{b}) = 
 U(a)\Zop{L_0}{-n}U(b) \quad (a,b \in V, \, n \in \bbZ_{\ge1}), \quad
 U(\wtO(V)) = O(V).
\end{align*}
Moreover, 
the unital associative algebras $(\wtA_{\gamma=\gamma_0}(V),\Hop{\gamma=\gamma_0}{},[\vac])$ 
and $(\Zhu_{L_0}(V),\Zop{L_0}{},[\vac])$ 
are isomorphic under the linear map 
\begin{align*}
 \wtA_{\gamma=\gamma_0}(V) \lto \Zhu_{L_0}(V), \quad [a] \lmto [U_{\gamma_0}(1)a]
\end{align*}
\end{fct}

The proof is based on the coordinate-change 
formula \cite[Proposition 1.2]{Hu}, \cite[6.5.6]{FB}:
\begin{align*}
 U_{\gamma_0}(1) \cdot a(x) \cdot U_{\gamma_0}(1)^{-1} = 
 \bigl(U_{\gamma_0}(e^{\gamma_0 x})a\bigr)(e^{\gamma_0 x}-1).
\end{align*}
See also \cite[Remark 3.13]{vEH} for a comment on this isomorphism.

\begin{rmk}
As a corollary of \cref{fct:Zhu:wtA=Zhu}, 
for a vertex operator algebra $V$, 
we have $\wt{A}_{\gamma=\gamma_0}(V) \cong \wt{A}_{\gamma=\gamma_1}(V)$ 
for any $\gamma_0,\gamma_1 \in \bbC\setminus\{0\}$.
\end{rmk}

\begin{rmk}\label{rmk:1:AVk=U}
By \cite[\S3]{FZ}, for the universal affine vertex algebra $V^k(\frg)$ 
of a semisimple Lie algebra $\frg$ with level $k\ne-h^\vee$, 
we have $\Zhu_{L_0}(V^k(\frg)) \cong U(\frg)$,
where $L_0$ is the zero-mode of the Sugawara field $L(z)$.
Thus, we have $\wt{A}(V^k(\frg)) \cong U(\frg)$ for $k\ne-h^\vee$.
Even without knowing this fact, 
one can prove the isomorphism $\wt{A}(V^k(\frg)) \cong U(\frg)$ 
for any level $k$ by using the argument explained in \cref{s:N=124}.
\end{rmk}

\section{Zhu algebra of superconformal algebras}\label{s:N=124}

Here we compute the Zhu algebra $\wtA(V)$
for superconformal vertex algebras in a direct way.

We will freely use the language of Lie conformal algebras (or vertex Lie algebras) 
and the enveloping vertex algebra of a Lie conformal algebra.
See \cite[\S2.7, \S4.7]{K}, \cite[\S1.5]{DK}, \cite[Chap.\ 16]{FB} 
and \cite[\S2.4]{HK} for details.
We list some terminology and notations.
\begin{itemize}
\item 
A \emph{Lie conformal algebra} $\clL$ is a $\bbC[\pd]$-module equipped with an even bilinear map
$[\cdot_\lambda \cdot]\colon \clL \otimes \clL \to \bbC[\lambda] \otimes \clL$ 
satisfying some axioms 
\cite[Definition 2.2]{HK}.

\item
Given a Lie conformal algebra $\clL$, one can associate to it 
a vertex algebra $V(\clL)$ called the \emph{enveloping vertex algebra} 
of $\clL$ \cite[\S4.7]{K}, \cite[\S16.11]{FB}.

\item
Let $\clL$ be a Lie conformal algebra, 
$C \in \clL$ be a central element such that $\pd C=0$, 
and $c \in \bbC$.
We denote by $V_c(\clL)$ the quotient $V(\clL)/(C-c\vac)_{(-1)}V(\clL)$.
It has a natural vertex algebra structure.
\end{itemize}

As stated in \cref{ss:intro:main}, the vertex algebras we will study are 
$N=1,2,3,4$ and big $N=4$ superconformal vertex algebras.
As mentioned in \cref{ss:intro:finW}, these are known to be 
the affine $W$-algebras $W^k(\frg,f)$ associated with minimal nilpotent elements 
$f=f_{\tmin}$ (see \cite[\S\S8.2--8.6]{KW} for the details).
\begin{align*}
 &V^{N=1}_c \cong W^k(\osp(1|2),f), &  & c=\frac{3}{2}-\frac{12(k+1)^2}{2k+3},  \\
 &V^{N=2}_c \cong W^k(\fsl(1|2),f), & & c=-3(2k+1), \\
 &V^{N=4}_c \cong W^k(\psl(2|2),f), & & c=-6(k+1), \\
 &V^{N=3}_c \cong W^k(\osp(3|2),f) \otimes F^k(\bbC), & & c=-6k-3,\\ 
 &V^{N=4,\tbig}_{c,a} \cong W^k(D(2,1;a),f) \otimes B^k(\bbC\xi) \otimes F^k(\bbC^4),
 & & c=-6k, \ a \in \bbC\setminus\{0,-1\}.
\end{align*}
Here $F^k(\bbC)$ and $F^k(\bbC^4)$ are the fermionic vertex algebras,  
and $B^k(\bbC\xi)$ is the bosonic vertex algebra, as explained in \cref{ss:intro:finW}.

Note that $\fsl(1|2)$ and $D(2,1;1)$ are isomorphic to $\osp(2|2)$ and $\osp(4|2)$,
respectively, and that an appropriate degeneration of $D(2,1;a)$ as $a\to-1$ is 
isomorphic to $\psl(2|2)\rtimes\fsl(2)$ \cite[Theorem 4.2 (2)]{IG}. 
Hence, all the Lie superalgebras appearing above are subalgebras of 
$D(2,1;a)$ or its degeneration.

\subsection{\texorpdfstring{$N=1,2,4$}{N=1,2,4} superconformal algebras}\label{ss:N=4}

Let $V^{N=4}$ be the enveloping vertex algebra of the (small) 
$N=4$ superconformal algebra \cite[\S5.9]{K}, \cite[\S8.4]{KW}.
It is strongly generated by eight elements 
\begin{itemize}
\item 
a conformal element $L$ of central charge $c$ \eqref{eq:wtA:L},
\item
three even primary elements $J^0,J^+,J^-$ of conformal weight $1$ \eqref{eq:wtA:primary},
\item
and four odd primary elements $G^+,G^-,\ol{G}^+,\ol{G}^-$ of conformal weight $3/2$.
\end{itemize}
The remaining non-zero $\lambda$-brackets (see \eqref{eq:wtA:lam-br}) are 
\begin{gather*}
 [{J^0}_\lambda J^{\pm}] = \pm2 J^{\pm}, \quad 
 [{J^0}_\lambda J^0] = \frac{c}{3}\lambda, \quad 
 [{J^+}_\lambda J^-] = J^0+\frac{c}{6}\lambda, \quad
 [{J^0}_\lambda G^{\pm}] = \pm G^{\pm}, \quad 
 [{J^0}_\lambda \ol{G}^{\pm}] = \pm \ol{G}^{\pm}, \\
 [{J^\pm}_\lambda G^\mp] = G^\pm, \quad 
 [{J^\pm}_\lambda \ol{G}^\mp] = -\ol{G}^\pm, \quad
 [{G^{\pm}}_\lambda \ol{G}^{\pm}]=(\pd+2\lambda)J^{\pm}, \quad 
 [{G^{\pm}}_\lambda \ol{G}^{\mp}]=L\pm\frac{1}{2}(\pd+2\lambda)J^0+\frac{c}{6}\lambda^2.
\end{gather*}
Using the expansion $L(z)=\sum_{n\in\bbZ}z^{-n-2}L_n$, 
we have the $L_0$-eigenspace decomposition 
\begin{align}\label{eq:N=4:VD}
 V^{N=4}=\bigoplus_{\Delta\in\hf\bbN}V_\Delta, \quad 
 V_\Delta\ceq\{v \in V^{N=4} \mid L_0v=\Delta v\}
\end{align}

Now we study the Zhu algebra
\[
 \wtA(V^{N=4})=V^{N=4}/\wtO(V^{N=4}).
\]
It is generated by the classes $[L]$, $[G^{\pm}]$, $[\ol{G}^{\pm}]$ and $[J^\alpha]$.
By the decomposition \eqref{eq:N=4:VD}, 
we can apply \cref{prp:wtA:LinZ} to $\wtA$, and hence
\begin{align}\label{eq:N=4:LinZ}
 \text{$[L]$ is in the center of $\wtA(V^{N=4})$.}
\end{align}
Also, by \cref{cor:wtA:pd=0} and \eqref{eq:wtA:lam-br}, 
the other non-trivial commutation relations of the generators in $\wtA$ are given as follows:
\begin{gather}\label{eq:N=4:Arels}
\begin{split}
 [[J^0],[J^{\pm}]] = \pm2 [J^{\pm}], \quad 
 [[J^+],[J^-]] = [J^0], \quad
 [[J^0],[G^{\pm}]] = \pm [G^{\pm}], \quad 
 [[J^0],[\ol{G}^{\pm}]] = \pm [\ol{G}^{\pm}], \\
 [[J^\pm],[G^\mp]] = [G^\pm], \quad 
 [[J^\pm],[\ol{G}^\mp]] = -[\ol{G}^\pm], \quad
 [[G^{\pm}],[\ol{G}^{\pm}]]=0, \quad 
 [[G^{\pm}],[\ol{G}^{\mp}]]=[L].
\end{split}
\end{gather}

Now, let $\fsl(2|2)$ be the special linear Lie superalgebra,
and consider the quotient $\psl(2|2) \ceq \fsl(2|2)/\bbC I$
($I$ is the identity supermatrix).
We have $\dim\psl(2|2)=6|8$.
Using the supermatrix notation, we take a minimal nilpotent element $f=f_{\tmin} \in \psl(2|2)$ as
\begin{align*}
 f \ceq 
 \left[\begin{array}{@{}c@{\,}|@{\,}c@{}} O & O \\ \hline O & F \end{array}\right], \quad 
 F \ceq \begin{sbm} 0 & 0 \\ 1 & 0 \end{sbm}, \quad 
 O \ceq \begin{sbm} 0 & 0 \\ 0 & 0 \end{sbm}.
\end{align*}
To be precise, we should denote $f+\bbC I$, but we simplify the notation.
Then its centralizer in $\psl(2|2)$
\begin{align*}
 \psl(2|2)^f \ceq \{ x \in \psl(2|2) \mid [x,f]=0 \}.
\end{align*}
has a linear basis consisting of 
\begin{gather}\label{eq:N=4:basis}
 f, \quad 
 j^0 \ceq \left[\begin{array}{@{}c@{\,}|@{\,}c@{}}
                H & O \\ \hline O & O \end{array}\right], \quad 
 j^+ \ceq \left[\begin{array}{@{}c@{\,}|@{\,}c@{}}
                E & O \\ \hline O & O \end{array}\right], \quad 
 j^- \ceq \left[\begin{array}{@{}c@{\,}|@{\,}c@{}}
                F & O \\ \hline O & O \end{array}\right], \\ 
 g^+ \ceq \left[\begin{array}{@{}c@{\ }c@{\,}|@{\,}c@{\ }c@{}}
                0&0&1&0 \\ 0&0&0&0 \\ \hline 0&0&0&0 \\ 0&0&0&0 \end{array}\right], \quad 
 g^- \ceq \left[\begin{array}{@{}c@{\ }c@{\ }|@{\,}c@{\ }c@{}} 
                0&0&0&0 \\ 0&0&1&0 \\ \hline 0&0&0&0 \\ 0&0&0&0 \end{array}\right], \quad 
 \ol{g}^+ \ceq \left[\begin{array}{@{}c@{\ }c@{\,}|@{\,}c@{\ }c@{}}
                0&0&0&0 \\ 0&0&0&0 \\ \hline 0&0&0&0 \\ 0&1&0&0 \end{array}\right], \quad 
 \ol{g}^- \ceq \left[\begin{array}{@{}c@{\ }c@{\,}|@{\,}c@{\ }c@{}}
                0&0&0&0 \\ 0&0&0&0 \\ \hline 0&0&0&0 \\ 1&0&0&0 \end{array}\right], 
\end{gather}
where 
$H\ceq\begin{sbm}1&0\\0&{-1}\end{sbm}$ and
$E\ceq\begin{sbm}0&1\\0&0\end{sbm}$.
Hence we have $\dim\psl(2|2)^f=4|4$ and
\[
 (\psl(2|2)^f)^{\ev} \cong \fsl(2)\oplus\bbC f, \quad 
 (\psl(2|2)^f)^{\od} \cong \text{(lowest weight component of 
 $\bbC^2 \otimes \bbC^2 \oplus \bbC^2 \otimes \bbC^2$)}.
\]

One can check that these generators of $\psl(2|2)^f$ satisfy 
the same relations as \eqref{eq:N=4:Arels}, and hence 
we have a surjective morphism of associative $\bbC$-algebras
\begin{align}\label{eq:N=4:varphi}
 \varphi\colon U(\psl(2|2)^f) \lsrj \wtA(V^{N=4})
\end{align}
given by
\begin{gather}\label{eq:N=4:vp-gen}
 f \lmto [L], \quad 
 j^\alpha \lmto [J^\alpha], \quad 
 g^\beta \lmto [G^\beta], \quad 
 \ol{g}^\beta \lmto [\ol{G}^\beta],\quad(\alpha\in\{0,+,-\},\ \beta\in\{+,-\}).
\end{gather}
Furthermore, we have:

\begin{thm}\label{thm:N=4}
The algebra morphism $\varphi$ is bijective, and we have
\begin{align*}
 \varphi\colon U(\psl(2|2)^f) \lsto \wtA(V^{N=4}).
\end{align*}
\end{thm}

The proof will be given in \cref{ss:N=124:bij}.

Recall that the vertex subalgebra of $V^{N=4}$ 
strongly generated by $L,J^0,G^+$ and $\ol{G}^-$ 
is isomorphic to the $N=2$ superconformal vertex algebra $V^{N=2}$, i.e., 
the (quotient of the) enveloping vertex algebra of 
the $N=2$ superconformal algebra \cite[\S5.9]{K}, \cite[\S8.3]{KW}.
Restricting \cref{thm:N=4} to it, we have:

\begin{cor}\label[cor]{cor:N=2}
We have an algebra isomorphism
\begin{align*}
 U(\fsl(1|2)^f) \lsto \wtA(V^{N=2}), \quad 
 f \lmto [L], \ j \lmto [J], \ g^+\lmto [G^+], \ g^-\lmto [\ol{G}^-],
\end{align*}
where, using the standard supermatrix notation, we set 
\begin{align}\label{eq:N=2:basis}
 f   \ceq \left[\begin{array}{@{}c@{\ }|@{\ }c@{\ }c@{}} 
                0&0&0 \\ \hline 0&0&0 \\ 0&1&0 \end{array}\right], 
 \quad 
 j   \ceq \left[\begin{array}{@{}c@{\ }|@{\ }c@{\ }c@{}}
                2&0&0 \\ \hline 0&1&0 \\ 0&0& 1 \end{array}\right], 
 \quad
 g^+ \ceq \left[\begin{array}{@{}c@{\ }|@{\ }c@{\ }c@{}}
                0&1&0 \\ \hline 0&0&0 \\ 0&0&0 \end{array}\right], 
 \quad
 g^- \ceq \left[\begin{array}{@{}c@{\ }|@{\ }c@{\ }c@{}} 
                0&0&0 \\ \hline 0&0&0 \\ 1&0&0 \end{array}\right], 
\end{align}
and $\fsl(1|2)^f$ denotes the centralizer of $f$ in 
\begin{align*}
 \fsl(1|2) \ceq \{X \in \fgl(1|2) \mid \str X=0\}, \quad 
 \str\left[\begin{array}{@{}c@{\,}|@{\,}c@{}} A & B \\ \hline C & D \end{array}\right] 
 \ceq \tr A - \tr D.
\end{align*}
Note that $j,g^+$ and $g^-$ form a basis of 
the Lie superalgebra $\fsl(1|2)^f \cong \fgl(1|1)$.
\end{cor}

Moreover, the vertex subalgebra of $V^{N=4}$ 
strongly generated by $L$ and $G \ceq \frac{1}{\sqrt{2}}(G^++\ol{G}^-)$
is isomorphic to the $N=1$ superconformal vertex algebra $V^{N=1}$, i.e., 
the (quotient of the) enveloping vertex algebra of of the $N=1$ superconformal algebra 
\cite[\S5.9]{K}, \cite[\S8.2]{KW}, \cite[Example 2.4]{HK},
also called the Neveu-Schwarz vertex algebra.
Restricting \cref{thm:N=4} to this vertex subalgebra, we have:

\begin{cor}[cf.\,{\cite[Appendix B]{Mi07}}]\label[cor]{cor:N=1}
We have an algebra isomorphism
\[
 U(\osp(1|2)^{f}) \lsto \wtA(V^{N=1}), \quad f \lmto [L], \ g \lmto [G].
\]
Here $\osp(1|2)$ is the ortho-symplectic Lie superalgebra
\begin{align*}
 \osp(1|2) \ceq \{X \in \fgl(1|2) \mid \st{X}J+JX=0\}, 
\end{align*}
with 
the supermatrix $J$ and the supertranspose $X \mto \st{X}$ given by 
\begin{align*}
 J \ceq \left[\begin{array}{@{}c@{\ }|@{\ }c@{\ }c@{}}
              1&0&0 \\ \hline 0&0&1 \\ 0&{-1}&0 \end{array}\right],
 \quad 
 \st{X} = \left[\begin{array}{@{}c@{\ }|@{\ }c@{}} 
                \trp{A} & \trp{C} \\ \hline -\trp{B} & \trp{D} \end{array}\right]
 \text{ for }
 X = \left[\begin{array}{@{}c@{\ }|@{\ }c@{}}A&B \\ \hline C&D \end{array}\right]. 
\end{align*}
Also, we used the following basis elements of $\osp(1|2)$:
\begin{align*}
 h \ceq \left[\begin{array}{@{}c@{\ }|@{\ }c@{\ }c@{}}
              0&0&0 \\ \hline 0&1&0 \\ 0&0&{-1} \end{array}\right],
 \quad
 e \ceq \left[\begin{array}{@{}c@{\ }|@{\ }c@{\ }c@{}}
              0&0&0 \\ \hline 0&0&1 \\ 0&0&0 \end{array}\right],
 \quad
 f \ceq \left[\begin{array}{@{}c@{\ }|@{\ }c@{\ }c@{}}
              0&0&0 \\ \hline 0&0&0 \\ 0&1&0 \end{array}\right],
 \quad 
 q \ceq \left[\begin{array}{@{}c@{\ }|@{\ }c@{\ }c@{}}
              0&0&1 \\ \hline {-1}&0&0 \\ 0&0&0 \end{array}\right],
 \quad
 g \ceq  \left[\begin{array}{@{}c@{\ }|@{\ }c@{\ }c@{}}
 0&1&0 \\ \hline 0&0&0 \\ 1&0&0 \end{array}\right].
\end{align*}
\end{cor}

Restricting the isomorphism to the vertex subalgebra strongly generated by $L$,
we recover the isomorphism $U(\osp(1|2)^{\ev})^f) \cong \bbC[f] \sto \wtA(V^{N=0})$.

\subsection{\texorpdfstring{$N=3$}{N=3} superconformal algebra}\label{ss:N=3}

For the $N=3$ case, following 
\cite[\S8.5]{KW}, we consider the vertex algebra $V^{N=3}_c$ strongly generated by
\begin{itemize}
\item 
a conformal element $L$ of central charge $c$ \eqref{eq:wtA:L},
\item
three even primary elements $A^i\ (i=1,2,3)$ of conformal weight $1$ \eqref{eq:wtA:primary},
\item
three odd primary elements $G^i\ (i=1,2,3)$ of conformal weight $3/2$,
\item
and an odd primary element $\Phi$ of conformal weight $1/2$,
\end{itemize}
with the remaining non-zero $\lambda$-brackets given by 
\begin{gather}
\label{eq:N=3:OPE1}
 [{A^i}_\lambda A^{j}] = \ve_{ijk}A^k+\frac{c}{3}\lambda\delta_{ij},\quad
 [{A^i}_\lambda G^j] = \ve_{ijk}G^k+\lambda \Phi\delta_{ij}, \\ 
\label{eq:N=3:OPE2}
 [{G^i}_\lambda G^j] 
  = 2L\delta_{ij}-\ve_{ijk}(\pd+2\lambda)A^k+\frac{c}{3}\lambda^2\delta_{ij},\quad
 [\Phi_\lambda G^i] = A^i, \quad
 [\Phi_\lambda \Phi] = -\frac{c}{3}
\end{gather}
for $i,j\in\{1,2,3\}$.
Here $\ve_{ijk}$ is antisymmetric with $\ve_{123}=1$, 
and $\delta_{ij}$ denotes the Kronecker's delta.
We also used Einstein's summation convention for repeated indices.
According to \cref{cor:wtA:pd=0} and \eqref{eq:wtA:lam-br}, 
the OPEs \eqref{eq:N=3:OPE1} and \eqref{eq:N=3:OPE2} imply 
the following commutation relations in $\wtA(V^{N=3}_c)$:
\begin{gather}
\label{eq:N=3:Zhu_rel1}
 \text{$[L]$ and $[\vac]$ are in the center},\quad 
 [[A^i],[A^j]]=\ve_{ijk}[A^k],\quad 
 [[A^i],[G^j]]=\ve_{ijk}[G^k],\\
\label{eq:N=3:Zhu_rel2}
 [[G^i],[G^j]]=2L\delta_{ij},\quad 
 [[\Phi],[A^i]]=0,\quad 
 [[\Phi],[G^i]]=[A^i],\quad
 [[\Phi],[\Phi]]=-\frac{c}{3}[\vac].
\end{gather}

To determine the Zhu algebra of $V^{N=3}_c$, 
let $\osp(3|2)$ be the ortho-symplectic Lie superalgebra of superdimension $6|6$,
which can be realized by supermatrices as 
\begin{align*}
&\osp(3|2) \ceq \biggl\{
 \left[\begin{array}{@{}c@{\,}|@{\,}c@{}} A & B \\ \hline C & D \end{array}\right] \, \bigg| \, 
  A \in \fso(3), \, D \in \fsl(2), \, B \in \Mat_{3\times2}, \, C=-J(\trp{B}) \biggr\},
 \quad
 J \ceq \begin{sbm}0&1\\-1&0\end{sbm}.
\end{align*}
Using this realization, we take a minimal nilpotent element $f=f_{\tmin}\in \osp(3|2)$ as
\begin{align*}
 f \ceq \left[\begin{array}{@{}c@{\,}|@{\,}c@{}} O & O \\ \hline O & F \end{array}\right], \quad
 F \ceq \begin{sbm}0&0\\1&0\end{sbm}.
\end{align*}
Then its centralizer in $\osp(3|2)$
\begin{align*}
 \osp(3|2)^{f} \ceq \{ x \in \osp(3|2) \mid [x,f]=0 \}
\end{align*}
has a basis consisting of
\begin{gather}\label{eq:N=3:basis}
 f, \quad 
 a^i\ceq\left[\begin{array}{@{}c@{\,}|@{\,}c@{}} J^i & O \\ \hline O & O \end{array}\right], 
 \quad 
 g^i\ceq E_{i4}+E_{5i}\quad(i=1,2,3),
\end{gather}
where $E_{pq}$ is the supermatrix unit in $\fgl(3|2)$ and
\begin{gather*} 
 J^1 \ceq \begin{sbm}0&0&0\\0&0&-1\\0&1&0\end{sbm},\quad
 J^2 \ceq \begin{sbm}0&0&1\\0&0&0\\-1&0&0\end{sbm},\quad
 J^3 \ceq \begin{sbm}0&-1&0\\1&0&0\\0&0&0\end{sbm}.
\end{gather*}
Hence we have $\dim\osp(3|2)^{f}=4|3$ and
\begin{align*}
 (\osp(3|2)^{f})^{\ev}\cong\bbC\oplus\fso(3), \quad
 (\osp(3|2)^{f})^{\od}\cong\bigl(\text{the natural representation $\bbC^3$ of }\fso(3)\bigr).
\end{align*}

The universal enveloping algebra of $\osp(3|2)^{f}$ is almost the desired Zhu algebra, 
but it needs to be extended. The correct form is: 

\begin{lem}\label{lem:R3}
Define the $5|4$-dimensional Lie superalgebra $\frR^{N=3}_c$ by
\begin{gather}
\notag
 \frR^{N=3}_c \ceq 
 (\bbC L\oplus\bbC A^1\oplus\bbC A^2\oplus\bbC A^3\oplus\bbC Z) \oplus
 (\bbC G^1\oplus\bbC G^2\oplus\bbC G^3\oplus\bbC\Phi), \\
\label{eq:N=3:R3_rel}
\begin{split}
 \text{$L$ and $Z$ are in the center},\quad 
 [A^i,A^j]=\ve_{ijk}A^k,\quad 
 [A^i,G^j]=\ve_{ijk}G^k,\\
 [G^i,G^j]=2L\delta_{ij},\quad 
 [\Phi,A^i]=0, \quad 
 [\Phi,G^i]=A^i,\quad
 [\Phi,\Phi]=-\frac{c}{3}Z.
 \end{split}
\end{gather}
Then, we have an embedding of Lie superalgebras
\begin{gather}
\notag
 \osp(3|2)^{f} \linj \frR^{N=3}_c, \\
\label{eq:N=3:vp-gen}
 f \lmto L, \quad a^i \lmto A^i,\quad g^i \lmto G^i\quad (i=1,2,3).
\end{gather}
Moreover, we have a surjective morphism of associative $\bbC$-algebras
\begin{gather}
\label{eq:N=3:varphi}
 \varphi\colon 
 U_1(\frR^{N=3}_c) \ceq U(\frR^{N=3}_c)/\langle Z-1\rangle
 \lsrj \wtA(V^{N=3}_c), \\ 
 \notag
 Z\lmto[\vac], \quad  X\lmto[X] \quad (X\in\{L,A^i,G^i,\Phi\mid i=1,2,3\}), 
\end{gather}
where $\langle Z-1\rangle$ is the two-sided ideal of $U(\frR^{N=3}_c)$ generated by $Z-1$.
\end{lem}

\begin{proof}
The former half follows from the description \eqref{eq:N=3:basis} of the basis of $\osp(3|2)^f$
and the relations \eqref{eq:N=3:R3_rel} in $\frR^{N=3}_c$. 
The latter half follows from a similar argument to the previous \cref{ss:N=4} 
using the $\lambda$-brackets \eqref{eq:N=3:Zhu_rel1} and \eqref{eq:N=3:Zhu_rel2} in $V^{N=3}_c$, 
the relations \eqref{eq:N=3:Zhu_rel1}, 
and the fact that $[\vac]$ is in the center of $\wtA(V^{N=3}_c)$.
\end{proof}

Moreover, we have:

\begin{thm}\label{thm:N=3}
The algebra morphism $\varphi$ is bijective, 
and the assignment \eqref{eq:N=3:vp-gen} induces an embedding
\begin{align*}
 U(\osp(3|2)^{f})\linj U_1(\frR^{N=3}_c)\lsto\wtA(V^{N=3}_c).
\end{align*}
\end{thm}

The proof will be given in \cref{ss:N=124:bij}.

\begin{rmk}\label[rmk]{rmk:N=3:wtA-freeUmod}
Set $U \ceq U(\osp(3|2)^{f})$. 
Then, by the PBW-type basis of $U_1(\frR^{N=3}_c)$, we have
\begin{align*}
 \wtA(V^{N=3}_c) \cong U[\vac]\oplus U[\Phi]
\end{align*}
as left $U$-modules. 
Thus, $\wtA(V^{N=3}_c)$ is a free $U$-module of rank $1|1$.
\end{rmk}

\subsection{Big \texorpdfstring{$N=4$}{N=4} superconformal algebra}\label{ss:big4}

Following \cite[\S8.6]{KW}, 
we fix $a \in \bbC\setminus\{0,-1\}$ and $k\in\bbC$, and set
\begin{align}
    k^+\ceq-(a+1)k,\quad 
    k^-\ceq-\frac{a+1}{a}k,\quad
    \gamma^+\ceq\frac{1}{a+1},\quad
    \gamma^-\ceq\frac{a}{a+1},\quad
    c\ceq -6k.
\end{align}
We consider the vertex algebra $V^{N=4,\tbig}_{c,a}$ strongly generated by
\begin{itemize}
\item 
a conformal element $\wt{L}$ of central charge $c=-6k$ \eqref{eq:wtA:L},
\item
seven even primary elements $\wt{J}^\alpha,\ \wt{J}'^\alpha\ (\alpha\in\{0,+,-\})$ 
and $\xi$ of conformal weight $1$ \eqref{eq:wtA:primary},
\item
four odd primary elements $\wt{G}^{++},\wt{G}^{+-},\wt{G}^{-+},\wt{G}^{--}$ of conformal weight $3/2$,
\item
and four odd primary elements $\sigma^{++},\sigma^{+-},\sigma^{-+},\sigma^{--}$ of conformal weight $1/2$.
\end{itemize}
Here the elements $\sigma$'s are the free fermions with non-zero $\lambda$-brackets
\begin{align}\label{eq:big4:OPE_s}
     [{\sigma^{+\pm}}_\lambda \sigma^{-\mp}] = k = -\tfrac{c}{6},   
\end{align}
and the element $\xi$ is the free boson with $\lambda$-bracket
\begin{align}
    [\xi_\lambda \xi] = k \lambda = -\tfrac{c}{6}\lambda.
\end{align}
The remaining non-zero $\lambda$-brackets are given in \cite[\S8.6, pp.453--455]{KW}, 
and a few typographical errors were corrected in \cite[\S6]{KMPb}
(see \eqref{eq:big4:KWmiss1} and \eqref{eq:big4:KWmiss2} below): 
\begin{align}
&[\wt{J}^{0}{}_\lambda \wt{J}^{0}]
 \;=\;
 2\lambda k^+,\qquad
 [\wt{J}^{0}{}_\lambda \wt{J}^{\pm}]
 \;=\; \pm 2\wt{J}^{\pm},\qquad
 [\wt{J}^{+}{}_\lambda \wt{J}^{-}]
 \;=\; \wt{J}^{0} 
 + 
 \lambda k^+,\\
&[\wt{J}'^{0}{}_\lambda \wt{J}'^{0}] 
 \;=\; 
 2\lambda k^-,\qquad
 [\wt{J}'^{0}{}_\lambda \wt{J}'^{\pm}] 
 \;=\; \pm 2\wt{J}'^{\pm},\qquad
 [\wt{J}'^{+}{}_\lambda \wt{J}'^{-}] 
 \;=\; \wt{J}'^{0} 
 +\lambda k^-,\\
&[\wt{J}^{0}{}_\lambda \wt{G}^{+\pm}] \;=\; \wt{G}^{+\pm}-\lambda a\,\sigma^{+\pm},\qquad
 [\wt{J}^{0}{}_\lambda \wt{G}^{-\pm}] \;=\; -\wt{G}^{-\pm}+\lambda\,\sigma^{-\pm},\\
&[\wt{J}^{+}{}_\lambda \wt{G}^{-\pm}] \;=\; \mp\wt{G}^{+\pm}\pm\lambda a\,\sigma^{+\pm},\qquad
 [\wt{J}^{-}{}_\lambda \wt{G}^{+\pm}] \;=\; \mp\wt{G}^{-\pm}\pm\lambda\,\sigma^{-\pm},\\
&[\wt{J}'^{0}{}_\lambda \wt{G}^{+\pm}] \;=\; \pm\wt{G}^{+\pm}\pm\lambda \,\sigma^{+\pm},\qquad
 [\wt{J}'^{0}{}_\lambda \wt{G}^{-\pm}] \;=\; \pm\wt{G}^{-\pm}\pm\lambda\,a^{-1}\sigma^{-\pm},\\
\label{eq:big4:KWmiss1}
&[\wt{J}'^{\pm}{}_\lambda \wt{G}^{+\mp}] 
 \;=\; -\wt{G}^{+\pm}-\lambda\,\sigma^{+\pm},\qquad
 [\wt{J}'^{\pm}{}_\lambda \wt{G}^{-\mp}] 
 \;=\; \wt{G}^{-\pm}+\lambda\,a^{-1}\sigma^{-\pm}, 
\\
\label{eq:big4:KWmiss2}
&[\wt{G}^{\pm+}{}_\lambda \wt{G}^{\mp-}] 
    \;=\; \wt{L} 
    \;+\; \frac{1}{2}(\partial+2\lambda)(\pm\gamma^+\wt{J}^{0}+\gamma^-\wt{J}'^{0}) 
    \;+\; \frac{c}{6}\,\lambda^{2}, 
\\
&[\wt{G}^{s\pm}{}_\lambda \wt{G}^{s\mp}] 
    \;=\; 
    \mp\gamma^+(\partial+2\lambda)\wt{J}^{s},\qquad
 [\wt{G}^{\pm s}{}_\lambda \wt{G}^{\mp s}] 
    \;=\; 
    \mp\gamma^-(\partial+2\lambda)\wt{J}'^{s},\\
&[\wt{J}^{0}{}_\lambda \sigma^{\pm s}] 
 \;=\; \pm \sigma^{\pm s},\qquad
 [\wt{J}^{+}{}_\lambda \sigma^{-\mp}] \;=\; \pm a\,\sigma^{+\mp},\qquad
 [\wt{J}^{-}{}_\lambda \sigma^{+\pm}] \;=\; \mp a^{-1}\,\sigma^{-\pm},\\
&[\wt{J}'^{0}{}_\lambda \sigma^{s\pm}] \;=\; \pm \sigma^{s\pm},\qquad
 [\wt{J}'^{+}{}_\lambda \sigma^{\pm-}] \;=\; \mp \sigma^{\pm+},\qquad
 [\wt{J}'^{-}{}_\lambda \sigma^{\pm+}] \;=\; \mp \sigma^{\pm-},\\
&[\wt{G}^{+\pm}{}_\lambda \sigma^{-\mp}] 
    \;=\; 
    \frac{1}{2}\gamma^-\bigl(\wt{J}^{0}\mp \wt{J}'^{0}\bigr)
    \;+\; \Big(\frac{a}{2}\Big)^{\!1/2}\,\xi,\\
&[\wt{G}^{-\pm}{}_\lambda \sigma^{+\mp}] 
    \;=\; 
    \mp\frac{1}{2}\gamma^+\bigl(\wt{J}^{0}\pm \wt{J}'^{0}\bigr)
    \;+\; \Big(\frac{1}{2a}\Big)^{\!1/2}\,\xi,\\
\label{eq:big4:OPE_g} 
&[\wt{G}^{\pm s}{}_\lambda \sigma^{\mp s}] 
    \;=\; 
    \pm\gamma^{\mp}\wt{J}'^{s},\qquad  
    [\wt{G}^{s\pm}{}_\lambda \sigma^{s\mp}] 
    \;=\; 
    \mp\gamma^s\wt{J}^{s},\qquad
 [\wt{G}^{\pm s}{}_\lambda \,\xi] 
    \;=\; (\pd+2\lambda)\left(\frac{a^{\pm1}}{2}\right)^{1/2}\sigma^{\pm s},
\end{align}
where $s\in\{+,-\}$.
Similar to the $N=3$ case, by Corollary \ref{cor:wtA:pd=0} and \eqref{eq:wtA:lam-br}, 
the OPEs \eqref{eq:big4:OPE_s}--\eqref{eq:big4:OPE_g} 
imply the following commutation relations in $\wtA(V^{N=4,\tbig}_{c,a})$:
\begin{align}
&\label{eq:big4:Zhu_rel1}
 [\wt{L}], [\xi]\text{ and }[\vac]\text{ are in the center,}\quad 
 \{[\wt{J}^0],[\wt{J}^+],[\wt{J}^-]\}\text{ and }\{[\wt{J}'^0],[\wt{J}'^+],[\wt{J}'^-]\}
 \text{ are $\fsl(2)$-triples},\\
&[[\wt{J}^{0}],[\wt{G}^{\pm s}]] 
 \;=\; \pm[\wt{G}^{\pm s}],\quad
 [[\wt{J}^{+}],[\wt{G}^{-\pm}]] 
 \;=\; \mp[\wt{G}^{+\pm}],\quad
 [[\wt{J}^{-}],[\wt{G}^{+\pm}]] 
 \;=\; \mp[\wt{G}^{-\pm}],\\
&[[\wt{J}'^{0}],[\wt{G}^{s\pm}]] 
 \;=\; \pm[\wt{G}^{s\pm}],\quad
 [[\wt{J}'^{+}],[\wt{G}^{\pm-}]] 
 \;=\; \mp[\wt{G}^{\pm+}],\quad
 [[\wt{J}'^{-}],[\wt{G}^{\pm+}]] 
 \;=\; \mp[\wt{G}^{\pm-}],\\
&[[\wt{G}^{\pm+}],[\wt{G}^{\mp-}]] 
 \;=\; [\wt{L}] ,\quad
 [[\wt{G}^{\pm s}],[\wt{G}^{\mp s}]] 
 \;=\; [[\wt{G}^{s\pm}],[\wt{G}^{s\mp}]] \;=\; 0,\\
&[[\wt{J}^{0}],[\sigma^{\pm s}]] 
 \;=\; \pm [\sigma^{\pm s}],\quad
 [[\wt{J}^{+}],[\sigma^{-\pm}]] 
 \;=\; \mp a[\sigma^{+\pm}],\quad
 [[\wt{J}^{-}],[\sigma^{+\pm}]], 
 \;=\; \mp a^{-1}[\sigma^{-\pm}]\\
&[[\wt{J}'^{0}],[\sigma^{s\pm}]] 
 \;=\; \pm [\sigma^{s\pm}],\quad
 [[\wt{J}'^{+}],[\sigma^{\pm-}]] 
 \;=\; \mp[\sigma^{\pm+}],\quad
 [[\wt{J}'^{-}],[\sigma^{\pm+}]] 
 \;=\; \mp[\sigma^{\pm-}],\\
&[[\wt{G}^{+\pm}],[\sigma^{-\mp}]] 
 \;=\; \frac{1}{2}\gamma^-[\wt{J}^{0}]\;\mp\;\frac{1}{2}\gamma^- [\wt{J}'^{0}]
 \;+\; \Big(\frac{a}{2}\Big)^{\!1/2}\,[\xi],\\
&[[\wt{G}^{-\pm}],[\sigma^{+\mp}]] 
 \;=\; \mp\frac{1}{2}\gamma^+[\wt{J}^{0}] \;-\; \frac{1}{2}\gamma^+[ 
 \wt{J}'^{0}]
 \;+\; \Big(\frac{1}{2a}\Big)^{\!1/2}\,[\xi],\\ 
&\label{eq:big4:Zhu_rel2}
[[\wt{G}^{\pm s}],[\sigma^{\mp s}]] 
 \;=\; \pm\gamma^{\mp}[\wt{J}'^{s}],\quad
 [[\wt{G}^{s\pm}],[\sigma^{s\mp}]] 
 \;=\; \mp\gamma^s[\wt{J}^{s}],\quad
 [[\sigma^{+\pm}],[\sigma^{-\mp}]]
 \;=\; -\tfrac{c}{6}[\vac].
\end{align}

Next, we turn to the Zhu algebra of $V^{N=4,\tbig}_{c,a}$.
Let 
\begin{align}\label{eq:big4:frg}
 \frg \ceq D(2,1;a), \quad a\in\bbC\setminus\{0,-1\}, 
\end{align}
be the exceptional Lie superalgebra of superdimension $9|8$
(see \cite[\S2.5.2]{K77} and \cite[\S8.6]{KW} for example).
As described in \cite[\S8.6, pp.449--450]{KW}, $\frg$ is the contragredient Lie superalgbera 
associated to the simple roots $\alpha_i$ ($i=1,2,3$) with scalar products 
\[
 \bigl[(\alpha_i|\alpha_j)\bigr]_{i,j=1}^3 = 
 \begin{bmatrix}
  -2\gamma^+ & \gamma^+ & 0 \\ \gamma^+ & 0 & \gamma^- \\ 0 & \gamma^- & -2\gamma^-
 \end{bmatrix}, \quad
 \gamma^+ = \frac{1}{a+1},\quad\gamma^- = \frac{a}{a+1}.
\]
The even (resp.\ odd) positive roots are given by 
\[
 h_{100},h_{001},h_{121}, \quad 
 (\text{resp.\ } h_{010},h_{110},h_{011},h_{111}), \quad 
 h_{lmn} \ceq l\alpha_1+m\alpha_2+n\alpha_3.
\]
The highest root is $\theta=h_{121}$, 
and the scalar products is normalized in the way $(\theta|\theta)=2$.
Choose the positive root elements $e_{lmn}=e_{h_{lmn}} \in \frg$ 
such that the non-zero brackets are 
\begin{gather*}
 [e_{100},e_{010}] = e_{110}, \quad 
 [e_{100},e_{011}] = e_{111}, \quad 
 [e_{010},e_{001}] = e_{011}, \quad 
 [e_{010},e_{111}] = e_{121}, \\
 [e_{001},e_{110}] =-e_{111}, \quad 
 [e_{001},e_{110}] =-e_{111}, \quad
 [e_{110},e_{011}] =-e_{121}
\end{gather*}
and choose the negative root elements $f_{\alpha}=e_{-h_{\alpha}}$ 
such that $[e_\alpha,f_\alpha]=\alpha$.

We take the minimal nilpotent element
\begin{align}\label{eq:big4:f}
 f \ceq f_{121} = f_\theta.
\end{align}
Then, using the Cartan matrix and the root space decomposition of $\frg$, 
we find that the centralizer $\frg^f$ has a basis consisting of the following elements:
\begin{align}\label{eq:big4:g^f_basis_1}
 &f=f_{121},\quad 
 j^{+} \ceq e_{100},\quad 
 j^{0} \ceq -\tfrac{1}{\gamma^+}h_{100},\quad 
 j^{-} \ceq -\tfrac{1}{\gamma^+}f_{100},\\
&\label{eq:big4:g^f_basis_2}
 j'^{+} \ceq e_{001},\quad 
 j'^{0} \ceq -\tfrac{1}{\gamma^-}h_{001},\quad 
 j'^{-} \ceq -\tfrac{1}{\gamma^-}f_{001},\\
&\label{eq:big4:g^f_basis_3}
 g^{++} \ceq f_{010},\quad 
 g^{-+} \ceq f_{110},\quad 
 g^{+-} \ceq -f_{011},\quad
 g^{--} \ceq f_{111}.
\end{align}
In particular, we have $\dim\frg^f=7|4$.

\begin{rmk}
Let us describe the structure of $\frg^f$ in another way.
Since the even part $\frg^{\ev}$ is known to be isomorphic to $\fsl(2)^{\oplus3}$ 
and the odd part $\frg^{\od}$ is isomorphic to $(\bbC^2)^{\otimes3}$ 
as a $\frg^{\ev}$-module (see \cite{K77}), 
the minimal nilpotent element $f$ of the semisimple Lie algebra $\frg^{\ev}$
can be extended to an $\fsl(2)$-subalgebra $\frs$ in $\frg$, and 
\begin{align*}
 (\frg^f)^{\ev}\cong\bbC\oplus\fso(4),\quad
 (\frg^f)^{\od}\cong(\text{$\frs$-lowest weight component of 
 $\frs\oplus\fso(4)\overset{\natural}{\curvearrowright}\bbC^2\otimes\bbC^4$}),
\end{align*}
where the notation $\overset{\natural}{\curvearrowright}$ denotes the action on the tensor product of the two natural representations, and the latter is a $(\frg^f)^{\ev}$-module isomorphism.
Hence we have $\dim\frg^{f}=7|4$. 
\end{rmk}

Similarly to the $N=3$ case, we need to extend $\frg^f=D(2.1;a)^{f_\theta}$ 
to describe the Zhu algebra $\wtA(V^{N=4,\tbig}_{c,a})$.

\begin{lem}\label{lem:R4}
Define the $9|8$-dimensional Lie superalgebra $\frR^{N=4}_{c,a}$ by 
\begin{align*}
 \frR^{N=4}_{c,a} \ceq 
 \Bigl(\bbC\wt{L}\oplus
       \bigoplus_{\alpha}(\bbC\wt{J}^{\alpha} \oplus \bbC\wt{J}'^{\alpha})
       \oplus\bbC\xi\oplus\bbC Z\Bigr)
 \oplus\Bigl(\bigoplus_{\pm\pm}(\bbC \wt{G}^{\pm\pm}\oplus\bbC \sigma^{\pm\pm})\Bigr),
\end{align*}
\begin{align*}
&\text{$\wt{L}$, $\xi$ and $Z$ are in the center}, \quad 
 \{\wt{J}^0,\wt{J}^+,\wt{J}^-\}\text{ and }\{\wt{J}'^0,\wt{J}'^+,\wt{J}'^-\}
 \text{ are $\fsl(2)$-triples},\\
&[\wt{J}^{0},\wt{G}^{\pm s}] 
 \;=\; \pm\wt{G}^{\pm s},\quad
 [\wt{J}^{+},\wt{G}^{-\pm}] 
 \;=\; \mp\wt{G}^{+\pm},\quad
 [\wt{J}^{-},\wt{G}^{+\pm}] 
 \;=\; \mp\wt{G}^{-\pm},\\
&[\wt{J}'^{0},\wt{G}^{s\pm}] 
 \;=\; \pm\wt{G}^{s\pm},\quad
 [\wt{J}'^{+},\wt{G}^{\pm-}] 
 \;=\; \mp\wt{G}^{\pm+},\quad
 [\wt{J}'^{-},\wt{G}^{\pm+}] 
 \;=\; \mp\wt{G}^{\pm-},\\
&[\wt{G}^{\pm+},\wt{G}^{\mp-}] 
 \;=\; \wt{L} ,\quad
 [\wt{G}^{\pm s},\wt{G}^{\mp s}] 
 \;=\; [\wt{G}^{s\pm},\wt{G}^{s\mp}] \;=\; 0,\\
&[\wt{J}^{0},\sigma^{\pm s}] 
 \;=\; \pm\sigma^{\pm s},\quad
 [\wt{J}^{+},\sigma^{-\pm}] 
 \;=\; \mp a\sigma^{+\pm},\quad
 [\wt{J}^{-},\sigma^{+\pm}], 
 \;=\; \mp a^{-1}\sigma^{-\pm}\\
&[\wt{J}'^{0},\sigma^{s\pm}] 
 \;=\; \pm \sigma^{s\pm},\quad
 [\wt{J}'^{+},\sigma^{\pm-}] 
 \;=\; \mp\sigma^{\pm+},\quad
 [\wt{J}'^{-},\sigma^{\pm+}] 
 \;=\; \mp\sigma^{\pm-},\\
&[\wt{G}^{+\pm},\sigma^{-\mp}] 
 \;=\; \frac{1}{2}\gamma^-\bigl(\wt{J}^{0}\mp \wt{J}'^{0}\bigr)
 \;+\; \Big(\frac{a}{2}\Big)^{\!1/2}\,\xi,\\
&[\wt{G}^{-\pm},\sigma^{+\mp}] 
 \;=\; \mp\frac{1}{2}\gamma^+\bigl(\wt{J}^{0}\pm 
 \wt{J}'^{0}\bigr)
 \;+\; \Big(\frac{1}{2a}\Big)^{\!1/2}\,\xi,\\ 
&[\wt{G}^{\pm s},\sigma^{\mp s}] 
 \;=\; \pm\gamma^{\mp}\wt{J}'^{s},\quad
 [\wt{G}^{s\pm},\sigma^{s\mp}] 
 \;=\; \mp\gamma^s\wt{J}^{s},\quad
 [\sigma^{+\pm},\sigma^{-\mp}]
 \;=\; -\tfrac{c}{6}Z.
\end{align*}
Then, we have an embedding of Lie superalgebras
\begin{gather}
\notag
 \frg^{f} \linj \frR^{N=4}_{c,a}, \\
\label{eq:bN=4:vp-gen}
 f \lmto \wt{L}, \quad j^{\alpha} \lmto \wt{J}^{\alpha},\quad 
 j'^{\alpha} \lmto \wt{J}'^{\alpha},\quad
 g^{\pm\pm} \lmto \wt{G}^{\pm\pm}\quad(\alpha\in\{0,+,-\}).
\end{gather}
We also have a surjective morphism of associative $\bbC$-superalgebras
\begin{align}\label{eq:bN=4:varphi}
 \varphi\colon 
 U_1(\frR^{N=4}_{c,a}):=U(\frR^{N=4}_{c,a})/\langle Z-1\rangle
 \lsrj \wtA(V^{N=4}_{c,a})
\end{align}
induced by the correspondence $Z\mto[\vac]$ and $X\mto[X]$ for 
$X\in\{\wt{L},\wt{J}^{\alpha},\wt{J}'^{\alpha},\xi,\wt{G}^{\pm\pm},\sigma^{\pm\pm}\}$.
Here $\langle Z-1\rangle$ is the two-sided ideal of $U(\frR^{N=4}_{c,a})$ generated by $Z-1$.
\end{lem}

\begin{proof}
We first prove the former half. 
By the definition of $\frg^f$, the element $f$ is in its center. Since we have
\begin{align*}
&[h_{100},e_{100}]=-2\gamma^+e_{100},\quad
 [h_{100},f_{100}]= 2\gamma^+f_{100},\quad
 [e_{100},f_{100}]=          h_{100},\\
&[h_{001},e_{001}]=-2\gamma^-e_{001},\quad
 [h_{001},f_{001}]= 2\gamma^-f_{001},\quad
 [e_{001},f_{001}]=          h_{001},
\end{align*}
each of $\{j^0,j^+,j^-\}$ in \eqref{eq:big4:g^f_basis_1} and $\{j'^0,j'^+,j'^-\}$ in \eqref{eq:big4:g^f_basis_2} forms an $\fsl(2)$-triple.
Their mutual commutativity follows from the Cartan matrix and the root space decomposition of $\frg$.
The commutator relations between $\{j^\alpha,j'^{\alpha}\}$ and $\{g^{\pm\pm}\}$ in \eqref{eq:big4:g^f_basis_3} follow from
\begin{align*}
&[h_{100},f_{010}]=-\gamma^+f_{010},\quad
 [h_{100},f_{110}]= \gamma^+f_{110},\quad
 [h_{100},f_{011}]=-\gamma^+f_{011},\quad
 [h_{100},f_{111}]= \gamma^+f_{111},\\
&[h_{001},f_{010}]=-\gamma^-f_{010},\quad
 [h_{001},f_{110}]=-\gamma^-f_{110},\quad
 [h_{001},f_{011}]= \gamma^-f_{011},\quad
 [h_{001},f_{111}]= \gamma^-f_{111},\\
&[f_{100},f_{010}]= \gamma^+f_{110},\quad
 [f_{100},f_{011}]= \gamma^+f_{111},\quad
 [f_{001},f_{010}]=-\gamma^-f_{011},\quad
 [f_{001},f_{110}]=-\gamma^-f_{111}.
\end{align*}
Moreover, the commutator relations among $\{g^{\pm\pm}\}$ follow from
\begin{align}
 [f_{010},f_{111}]
 =(\gamma^++\gamma^-)f_{121}
 =f_{121},\quad
 [f_{110},f_{011}]
 =-(\gamma^++\gamma^-)f_{121}
 =-f_{121}.
\end{align}

The latter half follows from a similar argument to the previous \cref{ss:N=4,ss:N=3} 
using the $\lambda$-brackets \eqref{eq:big4:OPE_s}--\eqref{eq:big4:OPE_g} in $V^{N=4,\tbig}_{c,a}$, 
the relations \eqref{eq:big4:Zhu_rel1}--\eqref{eq:big4:Zhu_rel2}, 
and the fact that $[\vac]$ is in the center of $\wtA(V^{N=4,\tbig}_{c,a})$.
\end{proof}

Then, similarly to the $N=3$ case, we have:

\begin{thm}\label{thm:big4}
The algebra morphism $\varphi$ is bijective, and we have an embedding
\begin{align*}
 &U(\frg^{f}) \linj U_1(\frR^{N=4}_{c,a}) \lsto \wtA(V^{N=4,\tbig}_{c,a}).
\end{align*}
\end{thm}

The proof will be given in \cref{ss:N=124:bij}.

\begin{rmk}
Similarly to \cref{rmk:N=3:wtA-freeUmod},
$\wtA(V^{N=4,\tbig}_{c,a})$ is a free left $U(\frg^{f})$-module,
which can be proved by the PBW-type basis of $U_1(\frR^{N=4}_{c,a})$.
\end{rmk}

\subsection{Proof of the isomorphisms}\label{ss:N=124:bij}

Here we show \cref{thm:N=4,thm:N=3,thm:big4} in a unified form.
For simplicity, we set $V$ to be one of the superconformal algebras 
in the previous subsections, and denote 
\[
 a\Hop{}{n}b \ceq a\Hop{\gamma=1}{n}b \ (a,b \in V), \quad 
 \wtO\ceq\wtO(V)=V\Hop{}{2}V, \quad 
 \wtA\ceq\wtA(V)=V/\wtO.
\]
We denote by $\{X_i\}_{i \in I}$ the strong generators of $V$ 
given in \cref{ss:N=4}--\cref{ss:big4}.
Explicitly,
\[
 \{X_i\}_{i \in I} = 
 \{L,J^0,J^{\pm},G^{\pm}, \ol{G}^{\pm}\}, \ 
 \{L, A^i, G^i, \Phi\}, \ 
 \{ \wt{L},\wt{J}^\alpha, \wt{J}'^\alpha, \wt{G}^{\pm s}, \sigma^{\pm s}, \xi\} 
\]
for $N=4$, $N=3$ and big $N=4$. 
Furthermore, we denote 
\begin{align*}
 &U\ceq U(\frg^f),\quad
 (\frg,f) =  (\psl(2|2),f);\\ 
 &U\ceq U_1(\frR^{N=3}_c), \quad
 (\frg,f) = (\osp(3|2),f_{\tmin}); \\
 &U\ceq U_1(\frR^{N=4,\tbig}_{c,a}), \quad
 (\frg,f) = (D(2,1;a),f_{\tmin}),
\end{align*}
where each $f$ is given in \cref{ss:N=4}--\cref{ss:big4}.
Then, we have a surjective morphism of associative $\bbC$-algebras 
\[
 \varphi\colon U \lsrj \wtA.
\]
Let $x_i \in U$  ($i \in I$) be the elements given in \eqref{eq:N=4:vp-gen}, 
\eqref{eq:N=3:vp-gen} and \eqref{eq:bN=4:vp-gen} satisfying $\varphi(x_i)=[X_i]$.

Then, \cref{thm:N=4,thm:N=3,thm:big4} follow from the next statement.

\begin{prp}\label[prp]{prp:Ngen:varphi}
$\varphi$ is bijective.
\end{prp}

\begin{proof}
We consider the induced map $\gr\varphi$ on the graded spaces
\begin{align*}
 \gr\varphi\colon \gr U \lto \gr \wtA \ceq (\gr V)/(\gr \wtO). 
\end{align*}
Here $\gr U$ denotes the graded space associated to the PBW-type filtration $F_{\bl}U$
on $U$ by
\[
    F_0U\ceq\bbC 1,\quad
    F_{n+1}U\ceq\Span_{\bbC}\{u,x_iu\mid i\in I,\ u\in F_nU\}.
\]
Also, using the strong generators $\{X_i\}_i$ of $V$, 
we define an increasing filtration $F_{\bl}V$ on $V$ by 
\[
 F_0V \ceq \bbC\vac,\quad 
 F_{n+1}V \ceq \Span_{\bbC}\{v,(X_i)_{(-m)}v\mid i\in I,\ m>0,\ v\in F_nV\}.
\]
This is nothing but the PBW filtration on $V$ regarded as the vacuum module
of the corresponding Lie superalgebra.
It induces another filtration $F_{\bl}\wtO$ on the subspace $\wtO\subset V$.
Finally, $\gr V = \bigoplus_{k=0}^\infty \gr_kV$, $\gr_kV \ceq F_kV/F_{k-1}V$ and 
$\gr\wtO = \bigoplus_{k=0}^\infty \gr_k\wtO$, 
$\gr_k\wtO \ceq (F_kV \cap \wtO)/(F_{k-1}V\cap\wtO)$ 
are the associated graded spaces.

To prove that $\varphi$ is bijective, 
we construct the inverse $\psi$ of $\gr\varphi$ as follows.
Using the Kronecker's delta $\delta_{ij}$, we consider the assignment
\begin{align}\label{eq:N=124:psi-asn}
 \gr V \ni \pd^{{n}}X_{i} + F_{\low} \lmto \delta_{n0} x_i + F_{\low} \in \gr U\quad 
 (i \in I, \, n \in \bbZ_{\ge0}),
\end{align}
where for $v \in F_mV \subset V$, we denote $v+\Flow \ceq v+F_{m-1}V \in \gr V$.
Since $\gr V$ is a supercommutative algebra freely generated by 
$\{\pd^{n}X_{i} + \Flow \mid i\in I, \, n\in\bbZ_{\ge0}\}$ and 
$\gr U$ is a supercommutative algebra (not freely in general) 
generated by $\{x_i + \Flow \mid i \in I\}$,
the assignment \eqref{eq:N=124:psi-asn} uniquely extends to a well-defined morphism 
of supercommutative algebras
\[
 \psi\colon \gr V \lto \gr U. 
\]
To show that $\psi$ descends to a map from $\gr\wtA=\gr V/\gr\wtO$, 
it suffices to prove
\begin{align}\label{eq:psi:wtO-Ker}
 \gr_k\wtO \subset W_k \subset \Ker\psi \quad (\forall k\ge 0),
\end{align}
where $W_0 \ceq \{0\}$ and 
\begin{align}\label{eq:wtO-gen}
 W_k \ceq 
 \Span_{\bbC}\Bigl\{\pd^{n_1}X_{i_1}\Hop{}{1} (\cdots(\pd^{n_{k-1}}X_{i_{k-1}}\Hop{}{1}
  (\pd^{n_{k}}X_{i_{k}}\Hop{}{1} \vac))\cdots)+F_{k-1}V \mid 
  i_j\in I, \ \tsum_{j=1}^k n_j>0 \Bigr\}
\end{align}
for $k > 0$.
We will show \eqref{eq:psi:wtO-Ker} in the \cref{lem:psi:wtO-Ker} below.
Assuming that, we have $\gr\wtO = \bigoplus_{k\ge0}\gr_k\wtO \subset \Ker \psi$, 
and $\psi$ descends to
\begin{align*}
 \ol{\psi}\colon \gr\wtA=\gr V/\gr\wtO \lto \gr U, \quad 
 [X_i+F_{\low}] \lmto x_i+F_{\low},
\end{align*}
where $[ \, ]$ denotes the equivalence class in the quotient space $\gr\wt{A}$.
Then, the construction says that $\ol{\psi}$ and $\gr\varphi$ are inverse to each other,
and we conclude that $\varphi$ is bijective.
\end{proof}

Before stating the promised \cref{lem:psi:wtO-Ker}, we give some preliminaries.

\begin{lem}\label[lem]{lem:psi:prep}
Let $V$ be a vertex algebra, and $\{X_i \mid i \in I\} \subset V$ be 
some given $\bbZ/2\bbZ$-homogeneous elements.
We set
\begin{align}\label{eq:psi:PBW-F}
 F_0V \ceq \bbC\vac,\quad 
 F_{n+1}V \ceq \Span_{\bbC}\{v,(X_i)_{(-m)}v\mid i\in I,\ m>0,\ v\in F_nV\}
 \quad (n\in\bbN).    
\end{align}
Assume that the elements $\{X_i \mid i \in I\} \subset V$ satisfy 
\begin{align}\label{eq:psi:Xi-assump}
 (X_i)_{(m)}X_j \subset F_1 V \quad (\forall m \in \bbN, \, \forall i,j \in I).
\end{align}
Then, the following statements hold.
\begin{enumerate}
\item 
For any $p,q\in \bbN$ and $m \in \bbZ$, we have 
\begin{align}\label{eq:psi:1-2}
 (F_pV)_{(m)}(F_qV) \subset F_{p+q}V.
\end{align}
If moreover $m \in \bbN$, then we have 
\begin{align}\label{eq:psi:bl-1_prf}
 (F_pV)_{(m)}(F_qV) \subset F_{p+q-1}V.
\end{align}

\item 
For any $p,q\in\bbN$, we have
\begin{align}\label{eq:psi:12F}
 (F_pV)\Hop{}{1}(F_qV) \subset F_{p+q}V, \quad 
 (F_pV)\Hop{}{2}(F_qV) \subset F_{p+q}V.
\end{align}

\item
For any $a\in F_pV$, $b \in F_qV$ and $c \in F_rV$ with $p,q,r \in \bbN$,  we have 
\begin{align}
\label{eq:psi:bl-21}
    a \Hop{}{1} (b \Hop{}{1} c) + F_{p+q+r-1}V
 &=a_{(-1)}b_{(-1)}c+F_{p+q+r-1}V \\
\label{eq:psi:bl-22}
 &=\bigl(a_{(-1)}b\bigr)_{(-1)}c+F_{p+q+r-1}V \\
\label{eq:psi:bl-23}
 &=(a \Hop{}{1} b) \Hop{}{1} c + F_{p+q+r-1}V.
\end{align} 

\item 
For any $a\in F_pV$, $b \in F_qV$ and $c \in F_rV$ with $p,q,r \in \bbN$, we have 
\begin{align}
\label{eq:psi:3-1}
(a \Hop{}{1} b)\Hop{}{2}c+F_{p+q+r-1}V 
&=(a_{(-1)}b)_{(-2)}c+F_{p+q+r-1}V \\
\label{eq:psi:3-2}
&=a_{(-1)}b_{(-2)}c+a_{(-2)}b_{(-1)}c+F_{p+q+r-1}V \\
\label{eq:psi:3-3}
&=a\Hop{}{1}(b\Hop{}{2}c)
 +a\Hop{}{2}(b\Hop{}{1} c)+F_{p+q+r-1}V
\end{align}
\end{enumerate}
\end{lem}

\begin{proof}
For simplicity, we denote $F_p \ceq F_pV$ for each $p \in \bbN$.
\begin{enumerate}
\item 
Relation \eqref{eq:psi:1-2} easily follows from 
the definition \eqref{eq:psi:PBW-F} of $\{F_n\}_{n=0}^\infty$, and we omit the details.

For \eqref{eq:psi:bl-1_prf}, the case $p=0$ is obvious since $F_0=\bbC\vac$.
We show the general case by the induction on $p>0$.
Consider the case $p=1$.
Since the formula \eqref{eq:VA:pd-com} implies 
\begin{align}\label{eq:psi:F1}
 F_1 = \Span_{\bbC}\{\pd^nX_i,\vac \mid i \in I, \, n\in\bbN\},
\end{align}
it is enough to prove 
\begin{align}\label{eq:psi:bl-11}
 (\pd^n X_i)_{(m)}F_q \subset F_q
\end{align}
for any $m,n,q \in \bbN$ and $i \in I$.
We show it by the induction on $q \in \bbN$.
Note that, for any $a \in V$ and $m,n \in \bbN$ with $m<n$, we have
\begin{align}\label{eq:psi:binom=0}
 (\pd^n a)_{(m)} = (-1)^m \binom{m}{n} n! a_{(m-n)} = 0.
\end{align}
Then, we have $(\pd^n X_i)_{(m)}\vac=0$, so \eqref{eq:psi:bl-11} holds for $q=0$. 
Next, for $q\in\bbZ_{>0}$, the commutator formula \eqref{eq:VA:com-form} implies that,
for any $v \in F_{q-1}$, $m,n\in\bbN$, $n_1\in\bbZ_{>0}$, and $i,i_1 \in I$, we have 
\begin{align*}
 (\pd^n X_i)_{(m)}(X_{i_1})_{(-n_1)}v 
 &= 
 \Bigl[(X_{i_1})_{(-n_1)}(\pd^nX_i)_{(m)} + 
  \sum_{k\ge0}\binom{m}{k}\bigl((\pd^nX_i)_{(k)}X_{i_1}\bigr)_{(m-n_1-k)}\Bigr]v \\
 &=
 (X_{i_1})_{(-n_1)}(\pd^nX_i)_{(m)}v + 
  \sum_{k=n}^m\binom{m}{k}\binom{-k}{n}n! \bigl((X_i)_{(k-n)}X_{i_1}\bigr)_{(m-n_1-k)}v,
\end{align*}
where, in the second equality, we used a similar argument to \eqref{eq:psi:binom=0}
to restrict the range of $k$ to $n \le k \le m$.
Now the induction hypothesis implies $(\pd^nX_i)_{(m)}v \in F_{q-1}$, and
the definition \eqref{eq:psi:PBW-F} yields $(X_{i_1})_{(-n_1)}(\pd^nX_i)_{(m)}v \in F_q$.
Also, by the assumption \eqref{eq:psi:Xi-assump} and the inequality $k-n\ge 0$, 
we have $(X_i)_{(k-n)}X_{i_1} \in F_1$, so it is a linear span of $\vac$ and 
$(X_{i_2})_{(-n_2)}\vac$ for some $i_2 \in I$ and $n_2\in\bbZ_{>0}$.
We have $(\vac)_{(m-n_1-k)}v \in F_l$ obviously.
Also, the term $\bigl((X_{i_2})_{(-n_2)}\vac\bigr)_{(m-n_1-k)}v$ 
is proportional to $(\pd^{n_2-1}X_j)_{(N)}v$ with $N\ceq m-n_1-k-n_2+1$, 
which belongs to $F_{q-1}$ ($\subset F_q$) for $N\ge0$ by the induction assumption,
and belongs to $F_q$ for $N<0$ by the definition \eqref{eq:psi:PBW-F}.
Hence \eqref{eq:psi:bl-11} holds for any $q \in \bbN$,
and \eqref{eq:psi:bl-1_prf} holds for $p=1$.

Next, we consider the induction step $p>1$ of \eqref{eq:psi:bl-1_prf}.
It is enough to show $a_{(m)}c \in F_{p+q-1}$ for $a \in F_p$ and $c \in F_q$.
We can assume 
\[
 a=(\pd^{n_1}X_{i_1})_{(-1)}b \in F_p, \quad 
 b=(\pd^{n_2}X_{i_2})_{(-1)}\dotsb(\pd^{n_p}X_{i_p})_{(-1)}\vac \in F_{p-1}
\]
with $n_1,\dotsc,n_p\in\bbN$ and $i_1,\dotsc,i_p\in I$.
By the Borcherds identity \eqref{eq:VA:bor}, we have
\begin{align*}
 a_{(m)}c = \bigl((\pd^{n_1}X_{i_1})_{(-1)}b\bigr)_{(m)}c = 
 \sum_{k\ge0}\bigl((\pd^nX_i)_{(-1-k)}b_{(m+k)}c+p(X_i,a)b_{(m-k-1)}(\pd^nX_i)_{(k)}c\bigr).
\end{align*}
Then, we have $b_{(m+k)}c \in F_{p+q-2}$ by the induction hypothesis, and 
$(\pd^nX_i)_{(-1-k)}b_{(m+k)}c \in F_{p+q-1}$ by \eqref{eq:psi:F1} and \eqref{eq:psi:1-2}.
Also, we have $(\pd^nX_i)_{(k)}c \in F_q$ by \eqref{eq:psi:bl-11}, 
and $b_{(m-k-1)}(\pd^nX_i)_{(k)}c \in F_{p+q-1}$ by \eqref{eq:psi:1-2}. 
Hence we have $a_{(m)}c \in F_{p+q-1}$, and the induction step is proved.
Therefore, we have shown \eqref{eq:psi:bl-1_prf}.

\item 
This is a corollary of \eqref{eq:psi:1-2}, \eqref{eq:psi:bl-1_prf} 
and the formulas \eqref{eq:wtA:Hop1}, \eqref{eq:wtA:Hopn}: 
\begin{align}
\label{eq:psi:a1b}
&a\Hop{}{1}b
=a_{(-1)}b+\sum_{k\ge0}c_ka_{(k)}b, \quad 
 c_k \ceq B_{k+1}^+/(k+1)! \in \bbQ, \\
\label{eq:psi:a2b}
&a\Hop{}{2}b=a_{(-2)}b+\sum_{k\ge0}d_ka_{(k)}b, \quad d_k\in\bbQ.
\end{align}
Note that the therm $a_{(-1)}b$ does not appear in \eqref{eq:psi:a2b}.

\item
For \eqref{eq:psi:bl-21}, 
we have $b\Hop{}{1} c = b_{(-1)}c+v$ with 
$v \ceq \sum_{k\ge0}c_kb_{(k)}c$, 
where we used \eqref{eq:psi:a1b}.
Then, we have $b_{(-1)}c \in F_{q+r}$ and $v \in F_{q+r-1}$ by \eqref{eq:psi:bl-1_prf}.
Hence, by applying \eqref{eq:psi:bl-1_prf} repeatedly, we have 
\begin{align*}
a \Hop{}{1} (b \Hop{}{1} c) 
&= a_{(-1)}(b_{(-1)}c + v) 
 +\sum_{k\ge0}c_k a_{(k)}\bigl(b_{(-1)}c+v\bigr) \\ 
&\in
 a_{(-1)}b_{(-1)}c + a_{(-1)}F_{q+r-1}
 +\sum_{k\ge0}a_{(k)}(F_{q+r}+F_{q+r-1}) \\
&\subset
 a_{(-1)}b_{(-1)}c+F_{p+q+r-1}.
\end{align*}
Thus we have \eqref{eq:psi:bl-21}. The next equality \eqref{eq:psi:bl-22} follows from 
the Borcherds identity \eqref{eq:VA:bor} 
together with \eqref{eq:psi:bl-1_prf}.
The last equality \eqref{eq:psi:bl-23} can be shown similarly to \eqref{eq:psi:bl-21}, 
and we omit the details.

\item
The formula \eqref{eq:psi:a2b} implies 
\begin{align*}
  (a\Hop{}{1}b)\Hop{}{2}c + F_{p+q+r-1} 
&=(a\Hop{}{1}b)_{(-2)}c   + \sum_{k\ge0}d_k(a\Hop{}{1}b)_{(k)}c+F_{p+q+r-1}. \\
\intertext{
 Now, we have $a\Hop{}{1}b \in F_{p+q}$ by \eqref{eq:psi:1-2}, 
 and $(a\Hop{}{1}b)_{(k)}c \in F_{p+q+r-1}$ by \eqref{eq:psi:bl-1_prf}.
 Then, by \eqref{eq:psi:a1b} and a similar argument, we have}
&=(a_{(-1)}b)_{(-2)}c + \sum_{k\ge0}c_k(a_{(k)}b)_{(-2)}c+F_{p+q+r-1} \\ 
&=(a_{(-1)}b)_{(-2)}c +F_{p+q+r-1}. \\
\intertext{
 Hence we obtained \eqref{eq:psi:3-1}.
 Next, using the Borcherds identity \eqref{eq:VA:bor}, we have}
&=\sum_{k\ge0}
  \bigl(a_{(-1-k)}b_{(-2+k)}c+p(a,b)b_{(-3-k)}a_{(k)}c\bigr)+F_{p+q+r-1} \\
\intertext{
 By \eqref{eq:psi:bl-1_prf} and \eqref{eq:psi:1-2}, 
 the terms $a_{(-1-k)}b_{(-2+k)}c$ with $k\ge2$ and 
 $b_{(-3-k)}a_{(k)}c$ with $k \ge 0$ belongs to $F_{p+q+r-1}$. Hence,}
&=a_{(-1)}b_{(-2)}c+a_{(-2)}b_{(-1)}c+F_{p+q+r-1}, 
\end{align*}
which is \eqref{eq:psi:3-2}.
Finally, an argument similar to \eqref{eq:psi:bl-21} yields \eqref{eq:psi:3-3},
the details of which are omitted.
\end{enumerate}
\end{proof}

\begin{lem}\label[lem]{lem:psi:wtO-Ker}
With the notation of \cref{prp:Ngen:varphi}, 
for any $k \ge 0$, the following statements hold.
\begin{enumerate}
\item 
$W_k \subset \Ker\psi$.
\item
$\gr_k\wtO \subset W_k$.
\end{enumerate}
In particular, \eqref{eq:psi:wtO-Ker} holds.
\end{lem}

\begin{proof}
For simplicity, we denote $F_p \ceq F_pV$ for $p\in\mathbb{N}$ as before.
Note that \eqref{eq:wtA:pd-Hop} implies 
\begin{align}\label{eq:psi:Hop2-Hop1}
 a\Hop{}{2}b = (\pd a)\Hop{}{1} b \quad (a,b \in V).
\end{align}
\begin{enumerate}
\item 
The equality \eqref{eq:psi:bl-21} implies that 
the map $\psi\colon \gr V \to \gr U$ is a morphism of supercommutative algebras, 
and we have
\begin{align*}
 \psi(\pd^{n_1}X_{i_1}\Hop{}{1} \cdots \Hop{}{1} \pd^{n_k}X_{i_k}+F_{\low})
 =\psi(\pd^{n_1}X_{i_1})\psi(\pd^{n_2}X_{i_2})\cdots\psi(\pd^{n_k}X_{i_k})
\end{align*}
for any $n_1,\dotsc,n_k\ge0$ and $i_1,\dotsc,i_k \in I$.
Then, if $\sum_{j=1}^k n_j>0$, we have $n_j>0$ for some $j \in \{1,\dotsc,k\}$, 
and for this $j$, we have $\psi(\pd^{n_j}X_{i_j})=0$ 
by the definition \eqref{eq:N=124:psi-asn} of $\psi$.
Thus, we have 
\begin{align*}
 \psi(\pd^{n_1}X_{i_1}\Hop{}{1} \cdots \Hop{}{1} \pd^{n_k}X_{i_k}+F_{\low})=0,
\end{align*}
which implies the statement.


\item
We note the following two statements.
\begin{clist}
\item\label{i:psi:i}
$\gr_kV \subset
 \Span\{\pd^{n_1}X_{i_1}\Hop{}{1}\dotsb\Hop{}{1}\pd^{n_k}X_{i_k}+F_{\low} \mid 
 i_j\in I, \ n_j\ge 0 \ (j=1,\dotsc,k) \}$.
\item\label{i:psi:ii}
$\gr_k\wtO \subset \sum_{k_1+k_2=k}(\gr_{k_1}V)\Hop{}{2}(\gr_{k_2}V)$.
\end{clist}
The item \ref{i:psi:i} is a corollary of \eqref{eq:psi:bl-21},
and the item \ref{i:psi:ii} follows from 
$\wtO \cap F_l = \sum_{k_1+k_2 \le l}(F_{k_1}V)\Hop{}{2}(F_{k_2}V)$.
Hence, it suffices to show 
\begin{align}\label{eq:psi:wtO-gen-claim}
 (\pd^{m_1}X_{i_1}\Hop{}{1} \dotsb \Hop{}{1}\pd^{m_p}X_{i_p}\Hop{}{1}\vac)  
 \Hop{}{2}
 (\pd^{n_1}X_{j_1}\Hop{}{1} \dotsb \Hop{}{1}\pd^{n_q}X_{j_q}\Hop{}{1}\vac)
+ F_{\low}\in W_{p+q}
\end{align}
for $p,q\in\bbN$, $m_1,\dotsc,m_p,n_1,\dotsc,n_q\in\bbN$
and $i_1,\dotsc,i_p,j_1,\dotsc,j_q\in I$.
The case of $p=0$ follows from 
\[
 \vac\Hop{}{2}(\cdots) = (\pd\vac)\Hop{}{1}(\cdots) = 
 0\Hop{}{1}(\cdots)=0,
\]
where we used \eqref{eq:psi:Hop2-Hop1}.
Next, the case of $p=1$ follows from
\begin{align*}
 \pd^nX_i\Hop{}{2}(\cdots)+F_{\low}=\pd^{n+1}X_i\Hop{}{1}(\cdots)+F_{\low} \in W_{1+q},
\end{align*} 
which is a consequence of \eqref{eq:psi:Hop2-Hop1}.
The case $p>1$ follows from 
\begin{align*}
 (a\Hop{}{1}b)\Hop{}{2}c+F_\low
&=a\Hop{}{1}(b\Hop{}{2}c)
 +a\Hop{}{2}(b\Hop{}{1}c)+F_\low,
\end{align*}
for $a=\pd^{m_1} X_{i_1}$, $b=(\pd^{m_2} X_{i_2})\Hop{}{1}\dotsb\Hop{}{1}(\pd^{m_p} X_{i_p})$
and $c=(\pd^{n_1} X_{j_1})\Hop{}{1}\dotsb\Hop{}{1}(\pd^{n_q} X_{i_q})$, 
which is the consequence of \eqref{eq:psi:3-1}--\eqref{eq:psi:3-3}, 
together with the obvious 
\begin{align*}
 \pd^{n}X_{i}\Hop{}{1} W_{k-1}\subset W_{k} \quad(n\in\bbN, \, i \in I).
\end{align*}
\end{enumerate}
\end{proof}

\begin{rmk}
For the relations \eqref{eq:psi:wtO-Ker}, we actually have $\gr_k\wtO=W_k$ 
and $\bigoplus_k\gr_k\wtO=\bigoplus_kW_k=\Ker\psi$.
See \cite[Lemma 3.21 (a)]{DK} 
for the corresponding statement for the original Zhu algebra $A(V)=\Zhu_H(V)$.
\end{rmk}

\section{SUSY Zhu algebra}\label{s:SUSY}

\subsection{\texorpdfstring{$N_K=N$}{NK=N} SUSY vertex algebras}\label{ss:SUSY:NK=N}

We cite several notations for superfields from \cite[\S4]{HK}.
\begin{itemize}
\item
Let $N \in \bbZ_{\ge1}$, and denote $[N] \ceq \{1,2,\dotsc,N\}$.

\item 
Let $Z=(z,\zeta^1,\ldots,\zeta^N)$ be an \emph{$N_K=N$ supervariable}, 
i.e., a tuple of an even indeterminate $z$ and odd indeterminates $\zeta^1,\dotsc,\zeta^N$.
We denote 
\[
 Z^{j|J} \ceq z^j \zeta^J = z^j \zeta^{j_1} \cdots \zeta^{j_r} \quad  
 (j \in \bbZ, \, J = \{j_1 \le \dotsb \le j_r\} \subset [N]).
\]
Given two $N_K=N$ supervariables 
$Z_i=(z_i,\zeta_i^1,\dotsc,\zeta_i^N)$ ($i=1,2$), 
we have a group structure (see also \cite[Chap.\ 2, 7.3]{Ma}):
\begin{align}\label{eq:NK=1:+}
 Z_1 - Z_2 = 
 (z_1-z_2-\tsum_{k=1}^N\zeta_1^k\zeta_2^k,\zeta_1^1-\zeta_2^1,\ldots,\zeta_1^N-\zeta_2^N).
\end{align}
Hence 
$(Z_1 - Z_2)^{j|J} \ceq (z_1-z_2-\tsum_{k=1}^N\zeta_1^k\zeta_2^k)^j
 (\zeta_1^{j_1}-\zeta_2^{j_1}) \cdots (\zeta_1^{j_r}-\zeta_2^{j_r})$
for  $j \in \bbZ$ and $J = \{j_1,\ldots,j_r\} \subset [N]$.
In particular, we have
$(Z_1 - Z_2)^{-1|0} = \frac{1}{z_1-z_2}+\frac{\sum_{k=1}^N\zeta_1^k\zeta_2^k}{(z_1-z_2)^2}$
and $(Z_1 - Z_2)^{-1|[N]} = \frac{\prod_{k=1}^N(\zeta_1^k-\zeta_2^k)}{z_1-z_2}$.


\item 
For an $N_K=N$ supervariable $Z=(z,\zeta^1,\ldots,\zeta^N)$, we denote
\begin{align*}
 \sd_Z^k \ceq \pd_{\zeta^k}+\zeta^k \pd_z \quad (k \in [N]), \quad 
 \sd_Z^{j|J} \ceq \pd_z^j \sd_Z^{j_1} \cdots \sd_Z^{j_r}  \quad 
 (j \in \bbN, \, J =\{j_1 \le \dotsb \le j_r\} \subset [N]).
\end{align*}

\item 
For a linear superspace $V$, we set 
\begin{align*}
 V\dbr{Z} \ceq V\dbr{z}[\zeta] &= 
 \{v(Z)=\tsum_{j \in \bbN, \, J \subset [N]} Z^{j|J}v_{j|J} \mid v_{j|J} \in V \}, \\
 V\dbr{Z^{\pm1}} \ceq V\dbr{z^{\pm1}}[\zeta] &= 
 \{v(Z)=\tsum_{j \in \bbZ, \, J \subset [N]} Z^{j|J}v_{j|J} \mid v_{j|J} \in V \}, \\
 V\dpr{Z} \ceq V\dpr{z}[\zeta] &= 
 \{v(Z) \in V\dbr{Z^{\pm1}} \mid v_{j|J} =0 \text{ for } j \ll 0 \}.
\end{align*}
These are linear superspaces, 
where the parity of $Z^{j|J}v_{j|J}$ is set to be $\ol{r}+\ol{v_{j|J}} \in \Zt$
for $j \in \bbZ$, $J=\{j_1 \le \dotsb \le j_r\} \subset [N]$ and $v_{j|J} \in V$.

\item 
For a linear superspace $V$, we define the super residue 
$\sres_Z[\cdot]\delta Z\colon V\dbr{Z^{\pm1}} \to V$ by
\begin{align*}
 \sres_Z[v(Z)]\delta Z \ceq v_{-1|[N]}, \quad 
 (v(Z)=\tsum_{j \in \bbZ, \, J \subset [N]} Z^{j|J}v_{j|J} \in V\dbr{Z^{\pm1}}).
\end{align*}
\end{itemize}

As introduced in \cite[Definition 4.13]{HK}, an $N_K=N$ supersymmetric (SUSY for short) 
vertex algebra is a data $(\bbV,Y,\vac,\sd^1,\dotsc,\sd^N)$ consisting of 
\begin{itemize}
\item 
a linear superspace $\bbV=\bbV^{\ev} \oplus \bbV^{\od}$, 
\item
an even linear map $\bbV \to (\End \bbV)\dbr{Z^{\pm1}}$,
$a \mto Y(a,Z) = \sum_{j \in \bbZ, \, J \subset [N]} Z^{-1-j|[N]-J}a_{(j|J)}$, 
with each $Y(a,Z)$ called the \emph{superfield}, 
\item
a non-zero even element $\vac \in \bbV^{\ev}$, called the \emph{vacuum}, 
\item
odd operators $\sd^1,\dotsc,\sd^N \in (\End \bbV)^{\od}$, called the \emph{odd translation operators},
\end{itemize}
satisfying the following conditions for any $a,b \in \bbV$. 
\begin{clist}
\item  
(field condition) 
$Y(a,Z)b \in \bbV\dpr{Z}$. 

\item
(vacuum axiom)
$Y(\vac,Z)=\id_\bbV$ and $Y(a,Z)\vac = a + O(Z)$.

\item 
(translation axiom)
$[\sd^k,Y(a,z)] = (\pd_{\zeta^k}-\zeta^k\pd_z) Y(a,z)$ for any $k \in [N]$.

\item 
(locality axiom)
The superfields $Y(a,Z)$ and $Y(b,W)$ are local, i.e., there exists $M(a,b) \in \bbN$ such that 
$(z-w)^{M(a,b)}[Y(a,Z),Y(b,W)]=0$. 
\end{clist}
We will use the standard notation $a(Z) \ceq Y(a,Z)$ for $a \in \bbV$.

Below are some implications of the axioms of $N_K=N$ SUSY vertex algebras.
\begin{itemize}
\item 
The translation axiom implies \cite[Corollary 4.18]{HK}: 
\begin{align}\label{eq:NK:sd-a}
 (\sd^k a)(Z) = \sd_Z^k a(Z), \quad 
 \sd^k(a_{(j|J)}b) = (-1)^{N-J}\bigl((\sd^ka)_{(j|J)}b+(-1)^{\ol{a}}a_{(j|J)}(\sd^kb)\bigr)
\end{align}.
for $k \in [N]$, $a,b \in \bbV$, $j\in\bbZ$ and $J \subset [N]$.


\item 
All the operators $(\sd^k)^2 \in \End \bbV$ for $k \in [n]$ coincide, and  
\begin{align}\label{eq:NK:pd}
 \pd \ceq (\sd^1)^2 = \dotsb = (\sd^N)^2 \in (\End \bbV)^{\ev}
\end{align}
satisfies the translation axiom of non-SUSY vertex algebra (see \cref{sss:VA:VA}).
We call $\pd$ the \emph{even translation operator} of the SUSY vertex algebra $V$.

\end{itemize}

%
%

%
%

We can ``reduce'' a SUSY vertex algebra to a (non-SUSY) vertex algebra.

\begin{lem}\label[lem]{lem:NK:red}
Let $\bbV$ be an $N_K=N$ SUSY vertex algebra with vertex operators 
$a(Z)=a(z,\zeta^1,\dotsc,\zeta^N)$ for $a \in \bbV$.
Then the even linear map 
\begin{align*}
 \bbV \lto (\End \bbV)\dpr{z^{\pm1}}, \quad 
 a &\lmto a(z) \ceq a(z,0,\dotsc,0) = \sum_{j \in \bbZ} z^{-j-1}a_{(j|[N])}
\end{align*}
and the even translation $\pd$ gives a vertex algebra structure on $\bbV$. 
We call it the \emph{reduced part} of the $N_K=N$ SUSY vertex algebra $\bbV$,
and denote it by $\bbV_{\tred}$.
\end{lem}

\begin{proof}
The substitution $\zeta^1=\dotsb=\zeta^N=0$ makes the axioms of 
an $N_K=N$ SUSY vertex algebra into those of a vertex algebra.
\end{proof}

\begin{rmk}
In the geometry of supercurves (see \cite[Chap.\ 2]{Ma}),
for a supercurve $X=(\ul{X},\shO_X)$ with structure sheaf 
$\shO_X=\shO_X^{\ev}\oplus\shO_X^{\od}$,  
the reduced curve of $X$ is defined to be 
$(\ul{X},\shO_X^{\ev}/\sqrt{\shO_X^{\ev}})$.
Our terminology is an analogue of this notion.
\end{rmk}

To give some examples of $N_K=N$ SUSY vertex algebras, 
we will use the following terminology and notations.
\begin{itemize}
\item
We denote by $\clH$ the universal enveloping algebra of the $1|N$-dimensional Lie 
superalgebra with even basis $\lambda$ and odd basis $\chi^1,\dotsc,\chi^N$  satisfying
the relation $[\lambda,\chi^i]=0$ and $[\chi^i,\chi^j]=-2\delta_{i,j}\lambda$ ($i,j\in[N]$).
We also use another basis $\pd,\sd^1,\dotsc,\sd^N$ of $\clH$ subject to the relation 
$[\pd,\sd^i]=0$ and $[\sd^i,\sd^j]=2 \delta_{i,j} \pd$ ($i,j\in[N]$).

\item
We recall the $\Lambda$-bracket formalism \cite[3.2.1]{HK}. 
For $j \in \bbN$ and $J=\{j_1,\dotsc,j_r\} \subset [N]$, 
we denote $\Lambda^{j|J} \ceq \lambda^j\chi^{j_1}\dotsm\chi^{j_r}$.
Then, for an $N_K=N$ SUSY vertex algebra $\bbV$ and $a,b \in \bbV$, we define 
\begin{align*}
 [a_\Lambda b] \ceq \sum_{j \in \bbN, \, J \subset [N]} 
                    \frac{(-1)^{\abs{J}(N+(\abs{J}+1)/2)}}{j!}\Lambda^{j|J}a_{(j|J)}b,
\end{align*}
which is an element of $\clH \otimes V =V[\lambda,\chi^1,\dotsc,\chi^N]$.

\item 
An $N_K=N$ SUSY Lie conformal algebra \cite[Definition 4.10]{HK} 
is an abstraction of the properties of $\Lambda$-bracket,
and an SUSY analogue of Lie conformal algebras.
Formally, it is a pair of a $\Zt$-graded $\clH$-module $L$ equipped with a bilinear map
$[\cdot_\Lambda \cdot]\colon L \otimes L \to \clH \otimes L$ of parity $\ol{N}$, 
subject to certain conditions.

\item 
The enveloping $N_K=N$ SUSY vertex algebra $\bbV(L)$ of an $N_K=N$ SUSY Lie conformal 
algebra $L$ \cite[3.4, 4.11]{HK} is an SUSY analogue of the enveloping vertex algebra
of a Lie conformal algebra. 

\end{itemize}

\begin{eg}[{\cite[Example 5.5]{HK}}]\label[eg]{eg:SUSY:N=123}
For $N=1,2,3$, let $\wt{\clK}_N$ be the $N_K=N$ SUSY Lie conformal algebra generated 
as a free $\clH$-module by an element $\tau$ of parity $\ol{N}$ and an even element $C$, 
satisfying 
\[
 \pd C = \sd^iC =0 \ (i \in [N])
\]
and the $\Lambda$-bracket 
\[
 [\tau_\Lambda \tau]=\bigl(2\pd+(4-N)\lambda+\tsum_{i=1}^N\chi^i \sd^i\bigr)\tau + 
 \frac{1}{3}\lambda^{3-N}\chi^{[N]}C.
 \]
%
Let $\bbV^{N_K=N}_{SUSY}$ be the quotient of the enveloping 
SUSY vertex algebra $\bbV$ of $\wt{\clK}_N$ 
by $(C-c\vac)_{(-1|N)}\bbV$, where $c \in \bbC$ is called the \emph{central charge}.
Then $\bbV^{N_K=N}_{SUSY}$ is an $N_K=N$ SUSY vertex algebra
strongly generated by an odd element $\tau$ of parity $\ol{N}$ with $\Lambda$-bracket

The SUSY vertex algebras $V^{N_K=1,2,3}_{SUSY}$ encode 
the vertex algebra structures $V^{N=1,2,3}$ in \cref{s:N=124} as follows.
\begin{itemize}
\item
For $\bbV^{N_K=1}_{SUSY}$, if we expand the generating superfield $\tau(Z)$ as
\begin{align*}
 \tau(z,\zeta) = G(z) + 2\zeta L(z),
\end{align*} 
then $G$ and $L$ coincide with the generators of the (Neveu-Schwarz) vertex algebra 
$V^{N=1}$ in \cref{cor:N=1}.
Note that we can identify $\tau=G$ and $\sd \tau=2L$ in $\bbV^{N_K=1}_{SUSY}$,
where $\sd$ is the odd translation operator.
Hence, the reduced part $(\bbV^{N_K=1}_{SUSY})_{\tred}$ is isomorphic to $V^{N=1}$.

\item
For $\bbV^{N_K=2}_{SUSY}$, if we expand $\tau(Z)$ as
\begin{align*}
 \tau(z,\zeta^1,\zeta^2) = 
 J(z)+\zeta^1 (G^+(z)+G^-(z))+\zeta^2(G^+(z)-G^-(z))+2\zeta^1\zeta^2L(z), 
\end{align*} 
then $L,J,G^{\pm}$ coincide with the generators of 
the vertex algebra $V^{N=2}$ in \cref{cor:N=2}.
We can identify $\tau=J$, $\sd^1\tau=G^++G^-$, $\sd^2\tau=G^+-G^-$ and 
$\sd^1\sd^2 \tau=-2L$ in $V^{N_K=2}_{SUSY}$, 
where $\sd^1,\sd^2$ are the odd translation operators.
Thus, the reduced part $(\bbV^{N_K=2}_{SUSY})_{\tred}$ is isomorphic to $V^{N=2}$.

\item 
For $\bbV^{N_K=3}_{SUSY}$, the spaces $\bbC \tau$, $\Span_{\bbC}\{\sd^i\tau\}_{i=1}^3$, 
$\Span_{\bbC}\{\sd^i\sd^j\tau\}_{1\le i<j\le3}$ and $\bbC\sd^1\sd^2\sd^3\tau$
correspond to the subspaces $\bbC\Phi$, $\Span_{\bbC}\{A^i\}_{i=1}^3$, 
$\Span_{\bbC}\{G^i\}_{i=1}^3$, and $\bbC L$ in \cref{ss:N=3}, respectively.

\end{itemize}
The $N=4$ superconformal vertex algebra $V^{N=4}$ in \cref{ss:N=4} 
has a structure of $N_K=1$ SUSY vertex algebra \cite[Example 5.11]{HK}.
The big $N=4$ superconformal vertex algebra $V^{N=4,\tbig}$ in \cref{ss:big4}
has a structure of $N_K=4$ SUSY vertex algebra \cite[Example 5.5, (5.10), (5.11)]{HK}. 
\end{eg}

%

\subsection{Definition}\label{ss:SUSY:Zhu}

In this subsection, we will give the definition of Zhu algebra for a SUSY vertex algebra.
We first explain how to come up with our definition in the case $N_K=1$,
and then give the definition for the general $N_K=N$ case.

\subsubsection{}\label{sss:SUSY:Zhu:NK=1}

We seek an $N_K=1$ SUSY analogue of Zhu algebra $\wtA(V)$. 
Recalling the geometric meaning of $\wtA(V)$ for a non-SUSY vertex algebra $V$ 
(see \cref{rmk:wtA:geom}), we use the terminology of superconformal curves 
(or supersymmetric curves) in \cite[Chapter 2]{Ma} 
and consider a super-geometry analogue of the coordinate change \eqref{eq:wtA:f}.

Following \cite[Chap.\ 2, 7.3, 8.1]{Ma}, we denote by $\bbG_a^{1|1}$ 
the $1|1$-dimensional algebraic supergroup with super-coordinate $Z=(z,\zeta)$ 
whose product is given by \eqref{eq:NK=1:+}.
It is a super analogue of the additive group $\bbG_a$,
and has the standard $N_K=1$ SUSY structure $\sd_Z \ceq \pd_\zeta+\zeta\pd_z$.
We also denote by $\bbG_m^{1|1}$ the $1|1$-dimensional supergroup 
with super-coordinate $X=(x,\xi)$ whose product is given by 
\begin{align*}
 (x_1,\xi_1) \cdot (x_2,\xi_2) \ceq (x_1x_2+\xi_1\xi_2, x_1\xi_2+x_2\xi_1).
\end{align*}
It is a super analogue of the multiplicative group $\bbG_m$, 
and has a right-invariant $N_K=1$ SUSY structure $D_X \ceq x\pd_\xi+\xi\pd_x$.
Then, by \cite[Chap.\ 2, 8.2]{Ma}, the map 
\begin{align*}
 \Exp\colon \bbG_a^{1|1} \lto \bbG_m^{1|1}, \quad 
  Z=(z,\zeta) \lmto \Exp(Z) \ceq e^z(1+\zeta) = (e^z,\zeta e^z)
\end{align*}
is a surjective morphism of algebraic supergroups and compatible with SUSY structures.
We denote the inverse map of $\Exp$ by 
\begin{align*}
 \Log X \ceq (\log x,\xi/x).    
\end{align*}

Now, the following maps $F,G$ on $\bbC^{1|1}$ are natural $N_K=1$ SUSY analogue 
of the maps $f,g$ in \eqref{eq:wtA:f}. 
\begin{align*}
 F\colon Z = (z,\zeta) &\lmto \Exp(\gamma Z)-1 = (e^{\gamma z}-1,\gamma\zeta e^{\gamma z}), \\
 G\colon X = (x,\xi)   &\lmto \gamma^{-1}\Log(1+X) = 
                              \bigl(\gamma^{-1}\log(1+x),\gamma^{-1}\xi/(1+x)\bigr).
\end{align*}
Then we come up with:

\begin{dfn}
Let $\bbV$ be an $N_K=1$ SUSY vertex algebra (\cref{ss:SUSY:NK=N}),
and $\gamma$ be a non-zero complex number or a formal variable.
For $a,b \in \bbV$ and $n \in \bbZ$, we define $a \Hop{\gamma}{n} b \in \bbV[\gamma]$ by 
\begin{align}\label{eq:SUSY:N=1bl}
 a \Hop{\gamma}{n} b \ceq \gamma^{n-1}
 \sres_X \bigl[ x^{-n}\xi a\bigl(\gamma^{-1}\Log(1+X)\bigr)b \bigr] \delta X = 
 \sres_Z \Bigl[ \frac{\gamma^n \zeta e^{\gamma z}}{(e^{\gamma z}-1)^n}a(Z)b \Bigr] \delta Z
\end{align}
Here $\delta X$ denotes the Berezinian form of the $N_K=1$ supervariable $X=(x,\xi)$.
\end{dfn}

\begin{rmk}
The equality in \eqref{eq:SUSY:N=1bl} can be checked similarly 
as the non-SUSY case \eqref{eq:wtA:blXZ}, 
but here we should consider the Berezinian form $\delta X$ (section of the Berezinian sheaf) 
instead of the differential form $d X$ (section of the cotangent sheaf).
The coordinate change $X=\Exp(\gamma Z)-1$ affects the Berezinian form through superdeterminant, 
so we have $\delta X = (\pd_zx)(\pd_\zeta\xi)^{-1}\delta Z = \delta Z$, and obtain the equality.
\end{rmk}

Since $\sres_Z[\zeta h(Z)]\delta Z = \res_z[h(z,0)]dz$ 
for any series $h(Z)=h(z,\zeta)$, we have
\begin{align}\label{eq:SUSY:Nk=1bl}
 a \Hop{\gamma}{n} b = 
 \res_z \Bigl[ \frac{\gamma^n e^{\gamma z}}{(e^{\gamma z}-1)^n}a(z,0)b \Bigr] dz.
\end{align}
Hence it coincides with the non-SUSY operation $\Hop{\gamma}{n}$ 
on the reduced part $\bbV_{\tred}$ (\cref{lem:NK:red}).

\subsubsection{}\label{sss:SUSY:Zhu:NK=N}

We can immediately generalize the operation \eqref{eq:SUSY:Nk=1bl} 
to any $N_K=N$ SUSY vertex algebra.
Let $\bbV$ be an $N_K=N$ SUSY vertex algebra (\cref{ss:SUSY:NK=N}), 
and $\gamma$ be a complex number or a formal variable.
Also, let $a,b \in \bbV$ and $n \in \bbZ_{\ge 1}$.
If $\gamma \ne 0$, then we define $a \Hop{\gamma}{n} b \in \bbV[\gamma]$ by 
\begin{align}\label{eq:SUSY:bl}
 a \Hop{\gamma}{n} b \ceq 
 \sres_Z \Bigl[\frac{\gamma^n \zeta^{[N]} e^{\gamma z}}{(e^{\gamma z}-1)^n}a(Z)b\Bigr] \delta Z
= \res_z \Bigl[\frac{\gamma^n e^{\gamma z}}{(e^{\gamma z}-1)^n}a(z,0,\dotsc,0)b \Bigr] dz.
\end{align}
If $\gamma=0$, then we define 
\[
 a \Hop{\gamma=0}{n} b \ceq 
 \sres_Z \bigl[ z^{-n}\zeta^{[N]} a(Z)b \bigr] \delta Z = a_{(-n|[N])}b.
\]
Then the operation $\Hop{\gamma}{n}$ coincides 
with the non-SUSY operation $\Hop{\gamma}{n}$ on the reduced part $\bbV_{\tred}$.
Hence, \cref{thm:wtA:wtO} implies:

\begin{prp}\label[prp]{prp:SUSY:Zhu}
Let $\bbV$ be an $N_K=N$ SUSY vertex algebra, 
and $\gamma$ be a complex number or a formal variable. 
Then the quotient space 
\begin{align*}
 \wtA_\gamma(\bbV) \ceq \bbV[\gamma]/\wtO_\gamma(\bbV), \quad
 \wtO_\gamma(\bbV) \ceq \Span_{\bbC[\gamma]}\{a\Hop{\gamma}{n}b \mid a,b \in \bbV, \, n \ge 2\}
\end{align*}
is an associative algebra whose product $\Hop{\gamma}{}$ is induced by the operation 
$\Hop{\gamma}{1}$ as
\[
 [a]\Hop{\gamma}{}[b] \ceq [a\Hop{\gamma}{1}b] \quad (a,b \in \bbV), 
\]
and the equivalence class of $\vac\in\bbV$ is its unit.
Moreover, it is isomorphic to the Zhu algebra $\wtA_\gamma(\bbV_{\tred})$
for the non-SUSY vertex algebra $\bbV_{\tred}$.
\end{prp}

\begin{dfn}\label[dfn]{dfn:SUSY:Zhu}
The algebra $\wtA_\gamma(\bbV)$ in \cref{prp:SUSY:Zhu} is called 
the Zhu algebra of the $N_K=N$ SUSY vertex algebra $\bbV$.
We denote $\wtA(\bbV) \ceq \wtA_{\gamma=1}(\bbV)$ as in the non-SUSY case.
\end{dfn}

\begin{eg}
Consider the $N_K=1,2,3$ SUSY vertex algebra $\bbV^{N_K=1,2,3}_{SUSY}$ in \cref{eg:SUSY:N=123}.
Since its reduced part is isomorphic to $V^{N=1,2,3}$ in \cref{s:N=124}, 
we have $\wtA_\gamma(\bbV^{N_K=1,2,3}_{SUSY}) \cong \wtA_\gamma(V^{N=1,2,3})$.
\end{eg}

The algebra $\wtA_\gamma(\bbV)$ for a SUSY vertex algebra $\bbV$ enjoys 
the same properties of the algebra $\wtA(V)$ for a non-SUSY vertex algebra $V$
in \cref{cor:wtA:pd=0,prp:wtA:LinZ}.
In particular, for any $a \in \bbV$, we have
\begin{align}\label{eq:NK:pd=0}
 [\pd a] = 0.
\end{align}
We also have a similar statement as \cref{prp:wtA:gamma=0}:

\begin{lem}
Let $\bbV$ be an $N_K=N$ SUSY vertex algebra.
Then, the specialization $\wtA_{\gamma=0}(\bbV)$
is identical to the $C_2$-Poisson algebra \cite[Proposition 4.1.2]{Y}:
\[
 R(\bbV)\ceq\bbV/C_2(\bbV), \quad 
 C_2(\bbV) \ceq \Span_{\bbC}\{a_{(j|J)}b \mid a,b\in\bbV, \, j\ge2, \, J \subset [N]\}.
\]
In particular, the product is supercommutative, and it is equipped with a Poisson bracket:
\begin{align*}
 [a]\cdot[b]=[a_{(-1|N)}b]=[a]\Hop{\gamma=0}{}[b], \quad 
 \{[a],[b]\}=[a_{(0|N)}b]=\lim_{\gamma\to0}
 \frac{1}{\gamma}\Bigl([a]\Hop{\gamma}{}[b]-p(a,b)[b]\Hop{\gamma}{}[a]\Bigr).
\end{align*}
\end{lem}

\begin{proof}
It is an immediate consequence of the argument in \cref{prp:wtA:gamma=0} 
and the formula \eqref{eq:SUSY:bl}.
\end{proof}

However, the SUSY case has an additional structure.
Recall that an $N_K=N$ SUSY vertex algebra $\bbV$ 
has the odd translation operators $\sd^k$  ($k\in[N]$).
These are odd derivations with respect to the ${}_{(-1|N)}$-operation in $\bbV$.
Then, by the relation $(\sd^k)^2=\pd$ and \eqref{eq:NK:pd=0}, 
we have:

\begin{prp}\label[prp]{prp:SUSY:Zhu_sd}
The odd translation operators $\sd^k$ ($k\in[N]$) of $\bbV$
induces odd differentials of the superalgebra $\wtA_\gamma(\bbV)$,
which are denoted by the same symbols $\sd^k$.
\end{prp}

Recall that the odd derivations $\sd^k$ on $\bbV$ induces the odd differentials 
on the $C_2$-Poisson algebra $R(\bbV)$ \cite[\S4.1, Remark 4.1.3]{Y}.
Thus, we have: 

\begin{cor}\label[cor]{cor:SUSY:limit}
Under the limit $\gamma \to 0$, 
the odd linear operators $[a] \mto [\sd^k a]$ ($k\in[N]$) on $\wtA_\gamma(\bbV)$ 
converges to the differentials $\sd^k$ on $R(\bbV)$.
\end{cor}

\begin{eg}
Consider the $N_K=1,2,3$ SUSY vertex algebras $\bbV=\bbV^{N=1,2,3}$ in \cref{eg:SUSY:N=123}.
\begin{enumerate}
\item 
By \cref{cor:N=1} and \cref{prp:SUSY:Zhu}, 
we have $\wtA(\bbV^{N=1}) \cong U(\osp(1|2)^f)$, 
which is generated by even $f$ and odd $g$ satisfying $[g,g]=2f$ and $[[f,f]=f,g]=0$.
This superalgebra is equipped with the odd endomorphism $\sd$ 
satisfying $\sd g=f$, $\sd f=0$ and $\sd^2=0$.

\item 
By \cref{cor:N=2} and \cref{prp:SUSY:Zhu}, we have 
$\wtA(\bbV^{N=2}) \cong U(\fsl(1|2)^f)$, 
which is generated by even $f,j$ and odd $g^{\pm}$.
It is equipped with the odd endomorphisms $\sd^1$ and $\sd^2$ satisfying 
$g^\pm \propto (\sd^1\pm\sd^2)j$, $f \propto \sd^1\sd^2g$ and $(\sd^1)^2=(\sd^2)=0$.

\item 
By \cref{thm:N=3} and \cref{prp:SUSY:Zhu}, we have 
$\wtA(\bbV^{N=3}) \cong U_1(\frR^{N=3}_c)$, 
which is generated by even $[L],\{[A^i]\{_{i=1}^3$ and odd $\{[G^i]\}_{i=1}^3,[\Phi]$.
It is equipped with the odd endomorphisms $\sd^1$, $\sd^2$ and $\sd^3$ satisfying 
$\sum_{i=1}^3\bbC\sd^i[\Phi]=\sum_{i=1}^3\bbC[A^i]$,
$\sum_{1\le i<j\le3}\bbC\sd^i\sd^j[\Phi]=\sum_{i=1}^3\bbC[G^i]$,
$[L] \propto \sd^1\sd^2\sd^3[\Phi]$ and $(\sd^1)^2=(\sd^2)=(\sd^3)=0$.
\end{enumerate}
\end{eg}

\appendix

\section{\texorpdfstring{$L_0$}{L0}-centralizer in superconformal algebras}
\label{ss:zero-mode}

Here we explain the relationship between 
the Zhu algebras of superconformal vertex algebras
and the centralizer of the zero mode $\ell_0$ of the centerless Virasoro algebra 
in the corresponding superconformal algebras.
The latter are calculated in \cite{KL} using vector fields.

\subsection{Notation}

Following \cite{KL}, let $\Lambda(N)$ be the Grassmann superalgebra generated 
by $N$ odd variables $\theta_1,\ldots,\theta_N$, 
and set $\Lambda_t(N)\ceq\bbC[t,t^{-1}]\otimes\Lambda(N)$.
For $\ep=(1,1,\ldots,1)\in\bbZ^N$, let $K(N;\ep)$ be the Lie superalgebra 
of contact super vector fields with respect to the following contact form
\begin{align}
 \omega_\ep \ceq dt-\sum_{i=1}^N t\theta_id\theta_i,
\end{align}
which is known as the $SO(N)$-superconformal algebra in the Ramond sector 
(see \cite[(1.11)]{KL} for definition).
Then, the centerless Virasoro algebra $\ol{\Vir}\ceq\Der\bbC[t,t^{-1}]$ 
can be embedded into the derived subalgebra 
$\frR(N)\ceq K'(N;\ep)=[K(N;\ep),K(N;\ep)]$ by
\begin{align*}
 \ol{\Vir} \lto \frR(N), \quad -t^{n+1}\pd_t \lmto 
 \ell_n \ceq -t^{n+1}\pd_t-\hf nt^n\sum_{i=1}^N\theta_i\pd_{\theta_i}\quad(n\in\bbZ).
\end{align*}
By using this embedding, we write
\begin{align*}
 \frR(N)_0 \ceq C_{K'(N;\ep)}(\ell_0)
\end{align*}
for the centralizer of the zero mode $\ell_0=-t\pd_t$ contained in $\frR(N)= K'(N;\ep)$.

Let us also borrow from \cite[pp.80--81]{KL}
some notation for vector fields on $\Lambda_t(N)$.
We define $D_f$ for $f\in\Lambda_t(N)$ of parity $\ol{f}$ by
\begin{align*}
 D_f\ceq \Delta(f)D^\ep + \sum_{i=1}^N D^\ep(f)\theta_i\pd_{\theta_i} 
          + (-1)^{\ol{f}} \sum_{i=1}^N t^{-1}\pd_{\theta_i}(f)\pd_{\theta_i},
\end{align*}
where $\ep=(1,1,\ldots,1)$ and 
\begin{align*}
 \Delta(f)\ceq \left(2-\sum_{i=1}^N\theta_i\pd_{\theta_i}\right)(f)\in \Lambda_t(N),\quad
 D^\ep(f) \ceq 
 \left(\pd_t-\frac{1}{2}\sum_{i=1}^Nt^{-1}\theta_i\pd_{\theta_i}\right)(f)\in \Lambda_t(N).
\end{align*}
From \cite[(1.18)]{KL}, we have
\begin{align*}
 [D_f,D_g]=D_{\{f,g\}},\quad
 \{f,g\} \ceq \Delta(f)D^\ep(g) - D^{\ep}(f)\Delta(g) +
              (-1)^{\ol{f}}\sum_{i=1}^N t^{-1}\pd_{\theta_i}(f)\pd_{\theta_i}(g).
\end{align*}
For simplicity of notation, we set
\begin{align}
    &D_0\ceq D_t=2t\pd_t,\quad
    D_i\ceq D_{t\theta_i}=\theta_i t\pd_t-\pd_{\theta_i},\quad 
    D_{ij}\ceq D_{t\theta_i\theta_j}=\theta_i\pd_{\theta_j}-\theta_j\pd_{\theta_i},\\
    &D_{ijk}\ceq D_{t\theta_i\theta_j\theta_k}
    =-t\theta_i\theta_j\theta_k\pd_t
    -\theta_j\theta_k\pd_{\theta_i}
    -\theta_k\theta_i\pd_{\theta_j}
    -\theta_i\theta_j\pd_{\theta_k}.
\end{align}
Explicitly, when $N\leq4$, we have
\begin{align}
\label{eq:zero:DD:s}
&[D_0,D_{i}]=[D_0,D_{ij}]=[D_0,D_{ijk}]=0,\\
&[D_{i},D_{p}]=-\delta_{ip}D_0, \\
&[D_{i},D_{pq}]
    =-\delta_{ip}D_{q}+\delta_{iq}D_{p},\\
&[D_{i},D_{pqr}]
    =-\delta_{ip}D_{qr}+\delta_{iq}D_{pr}-\delta_{ir}D_{pq},\\
&[D_{ij},D_{pq}]
    =-\delta_{ip}D_{jq}+\delta_{iq}D_{jp}+\delta_{jp}D_{iq}-\delta_{jq}D_{ip},\\
&[D_{ij},D_{pqr}]
    =-\delta_{ip}D_{jqr}
    +\delta_{iq}D_{jpr}
    -\delta_{ir}D_{jpq}
    +\delta_{jp}D_{iqr}
    -\delta_{jq}D_{ipr}
    +\delta_{jr}D_{ipq},\\
\label{eq:zero:DD:g}
&[D_{ijk},D_{pqr}]=0.
\end{align}

\subsection{$N=1$ case}

Using the description of $\frR(1)_0$ in \cite[(1.14)]{KL}, we have
\begin{align*}
 \frR(1)_0 = \Span_{\bbC}\{D_0,D_{1}\}\subset
 \Der(\bbC[t]\otimes\Lambda(1)).
\end{align*}
Since $[L]$ and $[G]$ are linearly independent by \cref{cor:N=1}, 
we obtain an isomorphism of Lie superalgebras
\begin{align}
 \Span_\bbC\{[L],[G]\} \lsto \frR(1)_0, \quad
 [L]\lmto -\tfrac{1}{2}D_0, \quad [G]\lmto D_1
\end{align}
and an isomorphism of associative algebras
\begin{align*}
 \wtA(V^{N=1}) \cong U(\frR(1)_0).
\end{align*}

\subsection{$N=2$ case}

Using the description of $\frR(2)_0$ in \cite[(1.14)]{KL}, we have
\begin{align*}
 \frR(2)_0=\Span_{\bbC}\{D_0,D_{1},D_{2},D_{12}\}\subset
 \Der(\bbC[t]\otimes\Lambda(2)).
\end{align*}
Then, by \cref{cor:N=2}, we obtain an isomorphism of Lie superalgebras
\begin{gather*}
 \Span_\bbC\{[L],[J],[G^+],[G^-]\}\lsto\frR(2)_0, \\
 [L]\lmto -\tfrac{1}{2}D_0,\quad 
 [J]\lmto \sqrt{-1}D_{12},\quad
 [G^{\pm}]\lmto \tfrac{1}{2}(D_{1}\mp \sqrt{-1}D_{2})
\end{gather*}
and an isomorphism of associative algebras
\begin{align*}
\wtA(V^{N=2}) \cong U(\frR(2)_0).
\end{align*}

\subsection{$N=3$ case}

Using the description of $\frR(3)_0$ in \cite[(1.14)]{KL}, we have
\begin{align*}
 \frR(3)_0 = 
 \Span_{\bbC}\{D_0,D_{i},D_{ij},D_{123}\mid i,j\in\{1,2,3\}\}\subset
 \Der(\bbC[t]\otimes\Lambda(3)).
\end{align*}
We note that the centralizer $\osp(3|2)^f$ is realized in this Lie superalgebra as follows:
\begin{align*}
 \osp(3|2)^f\cong\Span_{\bbC}\{D_0,D_{i},D_{ij}\mid i,j\in\{1,2,3\}\}
 \subset \frR(3)_0.
\end{align*}

Unlike the preceding cases, the Zhu algebra $\wtA(V^{N=3}_c)$ is no longer 
isomorphic to the universal enveloping algebra $U(\frR(3)_0)$. 
Instead, we consider the following $1$-dimensional central extension
\begin{align}\label{eq:N=3:Zhu-L0}
 0 \lto \bbC[\vac] \lto \Span_\bbC\{[L],[A^i],[G^{i}],[\Phi],[\vac]\mid i=1,2,3\} 
   \lto \frR(3)_0 \lto 0,
\end{align}
which is isomorphic to the $5|4$-dimensional Lie superalgebra $\frR^{N=3}_c$ defined in \cref{lem:R3}.
Here the surjection onto $\frR(3)_0$ is given by
\begin{align*}
 [L]     \lmto -\frac{1}{2}D_0,\quad 
 [G^{i}] \lmto D_{i},\quad
 [A^i]   \lmto -D_{jk}\ (\text{if $\ve_{ijk}=1$}),\quad
 [\Phi]  \lmto D_{123},\quad
 [\vac]  \lmto 0.
\end{align*}
Using the relations \eqref{eq:zero:DD:s}--\eqref{eq:zero:DD:g}, one can check that this assignment is a well-defined Lie superalgebra homomorphism.
Then, by \cref{thm:N=3}, we can identify $\wtA(V^{N=3}_c)$ with a central quotient 
of the universal enveloping algebra 
of the central extension \eqref{eq:N=3:Zhu-L0} of the centralizer $\frR(3)_0$.


\subsection{Big $N=4$ case}

This case is a bit more complicated. 
Let $\frg=D(2,1;a)$ and $f$ be as in \eqref{eq:big4:frg} and \eqref{eq:big4:f}, respectively.
Using the description of $\frR(4)_0$ in \cite[(1.14)]{KL}, we have 
\begin{align*}
&\frR(4)_0 = 
 \Span_\bbC\{D_0,D_{i},D_{ij},D_{ijk}\mid i,j,k\in\{1,2,3,4\}\} \subset
 \Der(\bbC[t]\otimes\Lambda(4)).
\end{align*}
We note that the centralizer $\frg^f$ is realized in this Lie superalgebra as follows:
\begin{align*}
 \frg^f \cong \Span_\bbC\{D_0,D_{i},D_{ij}\mid i,j\in\{1,2,3,4\}\} \subset \frR(4)_0.
\end{align*}

Now, similarly to the $N=3$ case, we consider the following $2$-dimensional central extension:
\begin{align}\label{eq:big4:Zhu-L0}
 0 \lto \bbC[\xi]\oplus\bbC[\vac] \lto 
 \Span_\bbC\{[\wt{L}],[\wt{J}^\alpha],[\wt{J}'^\alpha],
             [\wt{G}^{\pm\pm}],[\sigma^{\pm\pm}],[\xi],[\vac]\} \lto \frR(4)_0 \lto 0,
\end{align}
which is isomorphic to the $9|8$-dimensional Lie superalgebra 
$\frR^{N=4}_{c,a}$ defined in \cref{lem:R4}.
In order to explain the surjection onto $\frR(4)_0$, 
we cite from \cite{STP} the following basis of the middle term in \eqref{eq:big4:Zhu-L0},
which may be viewed as serving as a bridge between the bases appearing in \cite{KW} and \cite{KL}.
\begin{align*}
&[L]\ceq[\wt{L}],\qquad 
 [G_1]\ceq [\wt{G}^{++}],\qquad 
 [G_2]\ceq [\wt{G}^{+-}],\qquad 
 [G_3]\ceq [\wt{G}^{-+}],\qquad 
 [G_4]\ceq [\wt{G}^{--}],\\
&[A_{+,1}]\ceq\tfrac{1}{2}[\wt{J}^0],\qquad
 [A_{+,2}]\ceq\tfrac{1}{2}([\wt{J}^+]+[\wt{J}^-]),\qquad
 [A_{+,3}]\ceq\tfrac{1}{2\sqrt{-1}}([\wt{J}^+]-[\wt{J}^-])\\
&[A_{-,1}]\ceq\tfrac{1}{2}[\wt{J}'^0],\qquad
 [A_{-,2}]\ceq\tfrac{1}{2}([\wt{J}'^+]+[\wt{J}'^-]),\qquad
 [A_{-,3}]\ceq\tfrac{1}{2\sqrt{-1}}([\wt{J}'^+]-[\wt{J}'^-])\\
&[Q^1]\ceq \tfrac{1}{\gamma^+}[\sigma^{++}],\qquad
 [Q^2]\ceq -\tfrac{1}{\gamma^+}[\sigma^{+-}],\qquad
 [Q^3]\ceq \tfrac{1}{\gamma^-}[\sigma^{-+}],\qquad
 [Q^4]\ceq \tfrac{1}{\gamma^-}[\sigma^{--}].
\end{align*}
See also \cite[\S3.4]{W} for the display of the big $N=4$ SCA in terms of this basis.  
Then, the surjection in \eqref{eq:big4:Zhu-L0} is determined by
\begin{align*}
&[L] \lmto -\frac{1}{2}D_0,\qquad
 [G_{a}] \lmto D_{a},\qquad
 [A_{\pm,i}] \lmto \frac{1}{2}\alpha^{ab}_{\pm,i}D_{ab},\\
&[Q^a] \lmto 
 -\frac{1}{6}\ve^{abcd}D_{bcd}, \qquad
 [\xi] \lmto 0,\qquad
 [\vac]\lmto 0,
\end{align*}
where 
$\alpha^{ab}_{\pm,i}\ceq
 \tfrac{1}{2}(\pm\delta_{ia}\delta_{b4}\mp\delta_{ib}\delta_{a4}+\ve_{iab})$, 
$\ve^{abcd}$ is antisymmetric with $\ve^{1234}=1$,
the upper (lower) signs are taken together,
and we used Einstein's summation notation for $a,b,c,d\in\{1,2,3,4\}$.

Finally, combining this realization with \cref{thm:big4}, 
we can identify $\wtA(V^{N=4,\tbig}_{c,a})$ with a central quotient of 
the universal enveloping algebra of the central extension \eqref{eq:big4:Zhu-L0} 
of the centralizer $\frR(4)_0$.


\subsection{$N=4$ case}

This case requires a Lie superalgebra different from $K(N;\ep)$.
Let $S(2;-1)=S'(2;-1)$ be the $SU(2)$-superconformal algebra, 
and $S(2)_0$ be the centralizer of $L_0$ in $S(2;-1)$ 
(see \cite[pp.82--85]{KL} for the definition). 
Then, based on the vector field realization in \cite[(1.21)]{KL}, we have
\begin{align*}
 S(2)_0 
 = \Span_\bbC\{L_0,T^\alpha_0,G^\pm_0,G^{\pm*}_0\} \subset
 \Der(\bbC[t]\otimes\Lambda(2)).
\end{align*}
where
\begin{align*}
&L_0\ceq-t\partial_t,\quad
T^1_0\ceq\tfrac{1}{2}(\theta_1\partial_{\theta_2}+\theta_2\partial_{\theta_1}),\quad
T^2_0\ceq\tfrac{\sqrt{-1}}{2}
 (\theta_2\partial_{\theta_1}-\theta_1\partial_{\theta_2}),\quad
T^3_0\ceq\tfrac{1}{2}(\theta_1\partial_{\theta_1}-\theta_2\partial_{\theta_2}),\\
&G^+_0\ceq-\sqrt{2}\theta_1 t\partial_t,\quad
G^-_0\ceq-\sqrt{2}\theta_2 t\partial_t,\quad
G^{+*}_0\ceq\sqrt{2}\partial_{\theta_1},\quad
G^{-*}_0\ceq\sqrt{2}\partial_{\theta_2}.
\end{align*}
Then, by \cref{thm:N=4}, we obtain an isomorphism of Lie superalgebras
\begin{gather*}
 \Span_\bbC\{[L],[J^\alpha],[G^{\pm}],[\ol{G}^{\pm}]\}\lsto S(2)_0, \\
 [L]\lmto L_0,\quad 
 [J^\alpha]\lmto T^\alpha_0,\quad
 [G^{\pm}]\lmto G^\pm_0,\quad
 [\ol{G}^{\pm}]\lmto G^{\pm*}_0
\end{gather*}
and an isomorphism of associative algebras
\begin{align*}
\wtA(V^{N=4}) \cong U(S(2)_0).
\end{align*}

\begin{Ack}
The authors are grateful to Naoki Genra
for helpful discussions and for explaining his work \cite{G24b}.
They are also grateful to Yoshiyuki Koga and Takuya Matsumoto 
for helpful discussions and comments.
R.S.\ was supported by JSPS KAKENHI Grant Number 23K19018. 
S.Y.\ was supported by Asahipen Hikari Foundation.
\end{Ack}

%


\end{document}